\documentclass[11pt]{amsart}

\usepackage[left]{lineno}

\usepackage{amsmath,amssymb,amsthm}
\usepackage[mathscr]{euscript}
\usepackage{color}
\usepackage{hyperref,doi}
\usepackage{epsf}
\usepackage{enumerate}
\usepackage{graphicx}
\usepackage{array}
\usepackage{color}

\usepackage[margin=1in]{geometry}

\newtheorem{theorem}{Theorem}[section]
\newtheorem{prop}[theorem]{Proposition}

\newtheorem{conjecture}[theorem]{Conjecture}

\newtheorem{lemma}[theorem]{Lemma}

\newtheorem{corollary}[theorem]{Corollary}

\theoremstyle{definition}
\newtheorem{definition}[theorem]{Definition}
\newtheorem{remark}[theorem]{Remark}
\newtheorem{example}[theorem]{Example}
\newtheorem{setup}[theorem]{Setup}

\newcommand{\conv}[1]{\mathrm{conv}\left\{#1\right\}}

\newcommand{\R}{\mathbb{R}}
\newcommand{\Z}{\mathbb{Z}}

\newcommand{\Frac}[2]{\genfrac{}{}{}{}{#1}{#2}}

\newcommand{\Pq}{\Delta_{(1,\bq)}}
\newcommand{\Prsq}{\Delta_{(1,\rs(\bq))}}
\newcommand{\Prsmq}{\Delta_{(1,\rs(\bq,m))}}
\newcommand{\Pp}{\Delta_{(1,\bp)}}
\newcommand{\Pw}{\Delta_{(1,\bw)}}
\newcommand{\cone}[1]{\mathrm{cone}\left(#1\right)}
\newcommand{\rs}{\mathrm{rs}}
\newcommand{\rsn}{\mathrm{rsn}}

\newcommand{\be}{\mathbf{e}}
\newcommand{\bp}{\mathbf{p}}
\newcommand{\bq}{\mathbf{q}}
\newcommand{\br}{\mathbf{r}}
\newcommand{\bs}{\mathbf{s}}
\newcommand{\bv}{\mathbf{v}}
\newcommand{\bx}{\mathbf{x}}
\newcommand{\by}{\mathbf{y}}
\newcommand{\bw}{\mathbf{w}}
\newcommand{\bc}{\mathbf{c}}
\newcommand{\sQ}{\mathcal{Q}}

\newcommand{\vep}{\varepsilon}
\newcommand{\lcm}[1]{\mathrm{lcm}\left(#1\right)}
\newcommand{\abs}[1]{ \left\lvert#1\right\rvert} 
\newcommand{\floor}[1]{ \left\lfloor#1\right\rfloor} 
\DeclareMathOperator{\ehr}{Ehr}

\newcommand\commentout[1]{}

\makeatletter
\def\@tocline#1#2#3#4#5#6#7{\relax
  \ifnum #1>\c@tocdepth 
  \else
    \par \addpenalty\@secpenalty\addvspace{#2}%
    \begingroup \hyphenpenalty\@M
    \@ifempty{#4}{%
      \@tempdima\csname r@tocindent\number#1\endcsname\relax
    }{%
      \@tempdima#4\relax
    }%
    \parindent\z@ \leftskip#3\relax \advance\leftskip\@tempdima\relax
    \rightskip\@pnumwidth plus4em \parfillskip-\@pnumwidth
    #5\leavevmode\hskip-\@tempdima
      \ifcase #1
       \or\or \hskip 1em \or \hskip 2em \else \hskip 3em \fi%
      #6\nobreak\relax
    \hfill\hbox to\@pnumwidth{\@tocpagenum{#7}}\par
    \nobreak
    \endgroup
  \fi}
\makeatother

\begin{document}



\title[The Integer Decomposition Property\ldots]{The Integer Decomposition Property and Weighted Projective Space Simplices}

\author{Benjamin Braun}
\address{Department of Mathematics\\
  University of Kentucky\\
  Lexington, KY 40506--0027}
\email{benjamin.braun@uky.edu}

\author{Robert Davis}
\address{Department of Mathematics\\
  Colgate University\\
  Hamilton, NY 13346}
\email{rdavis@colgate.edu}

\author{Derek Hanely}
\address{Department of Mathematics\\
  Penn State Behrend\\
  Erie, PA 16563}
\email{derek.hanely@psu.edu}

\author{Morgan Lane}
\address{Martha Layne Collins High School\\
  801 Discovery Boulevard\\
  Shelbyville, KY 40065}
\email{morgan.lane@shelby.kyschools.us}

\author{Liam Solus}
\address{Matematik\\
  KTH\\
  SE-100, 44 Stockholm, Sweden}
\email{solus@kth.se}

\subjclass[2020]{Primary: 52B20, 05A15, 11A05}


\date{21 November 2022}

\begin{abstract}
Reflexive lattice polytopes play a key role in combinatorics, algebraic geometry, physics, and other areas.
One important class of lattice polytopes are lattice simplices defining weighted projective spaces.
We investigate the question of when a reflexive weighted projective space simplex has the integer decomposition property.
We provide a complete classification of reflexive weighted projective space simplices having the integer decomposition property for the case when there are at most three distinct non-unit weights, and conjecture a general classification for an arbitrary number of distinct non-unit weights.
Further, for any weighted projective space simplex and $m\geq 1$, we define the $m$-th reflexive stabilization, a reflexive weighted projective space simplex.
We prove that when $m$ is $2$ or greater, reflexive stabilizations do not have the integer decomposition property.
We also prove that the Ehrhart $h^\ast$-polynomial of any sufficiently large reflexive stabilization is not unimodal and has only $1$ and $2$ as coefficients.
We use this construction to generate interesting examples of reflexive weighted projective space simplices that are near the boundary of both $h^*$-unimodality and the integer decomposition property.
\end{abstract}

\thanks{BB was partially supported by National Science Foundation award DMS-1953785.  RD was supported in part by NSF grant DMS-1922998.  LS was partially supported by the Wallenberg AI, Autonomous Systems and Software Program (WASP) funded by the Knut and Alice Wallenberg Foundation as well as Starting Grant (Etableringsbidrag) No.~2019-05195 from The Swedish Research Council (Vetenskapsr\aa{}det).}
\maketitle



\section{Introduction}

\subsection{Motivation}
Consider an integer partition $\bq\in \Z_{\geq 1}^n$ with the convention $q_1\leq\cdots\leq q_n$.
A lattice simplex defined by $\bq$ is
  \[
    \Pq := \conv{\be_1,\ldots,\be_n,-\sum_{i=1}^n q_i\be_i}
  \]
  where $\be_i$ denotes the $i^{th}$ standard basis vector in $\R^n$.
 Set $N(\bq):=1+\sum_iq_i$. 
 One can show, as for instance in \cite[Proposition 4.4]{NillSimplices}, that $N(\bq)$ is the normalized volume of $\Pq$.
 Let $\sQ$ denote the set of all lattice simplices of the form $\Pq$.

We call elements of $\sQ$ \emph{weighted projective space simplices}, as the simplices in $\sQ$ correspond to a subset of the simplices defining weighted projective spaces \cite{conrads}, namely those for which the vector $(1,\bq)$ gives the weights of the projective coordinates.
The set $\sQ$ is the focus of active study \cite{LaplacianSimplicesDigraphs,BraunDavisReflexive,BraunDavisSolusIDP,BraunHanelyTriangulation,BraunLiu,LiuSolus,SolusNumeralSystems}, motivated by questions regarding Ehrhart positivity, Ehrhart $h^\ast$-unimodality and real-rootedness, existence of unimodular triangulations, and other topics.
One reason for this interest is that $\sQ$ is a relatively general family of simplices with diverse lattice-point combinatorics, and another is because elements of $\sQ$ admit tractable number-theoretic characterizations of several important geometric and combinatorial properties.  
This makes $\sQ$ a natural testing ground for exploring combinatorial and geometric properties of lattice simplices, while maintaining some combinatorial control over the examples at hand.
Another reason for this interest is that algebraic and geometric properties of simplices in $\sQ$ correspond to geometric properties of weighted projective spaces, and thus are of interest to algebraic geometers.

Our focus in this paper is on two properties, reflexivity and the integer decomposition property, that are defined for general lattice polytopes as follows.
Recall that a subset $P\subseteq\R^n$ is a {\em $d$-dimensional (convex) lattice polytope} if it is the convex hull of finitely many points $\bv^{(1)},\ldots,\bv^{(k)}\in\Z^n$ that span a $d$-dimensional affine subspace of $\R^n$.
Many interesting geometric and algebraic properties of $P$ are revealed by considering the \emph{cone over $P$}, defined as the non-negative span of the vectors formed by prepending a $1$ to each vertex of $P$, i.e.,
\[
\cone{P}:=\textrm{span}_{\R_{\geq 0}}\{(1,\bv^{(i)}):i=1,\ldots,k\} \, .
\]
A lattice polytope $P$ is said to have the {\em integer decomposition property} (or to be IDP) if for every positive integer $m$ and each $(m,\bw) \in \cone{P} \cap \Z^{n+1}$, there exist $m$ points $\bx_1,\ldots,\bx_m \in P \cap \Z^n$ for which $(k,\bw) = \sum_i (1,\bx_i)$.
IDP polytopes are also known as \emph{projectively normal} polytopes.
Letting $K^\circ$ denote the topological interior of a space $K$, $P$ is said to be \emph{reflexive} if there exists an integer vector $\bc\in P^\circ \cap \Z^{n+1}$ such that
\[
  (1,\bc)+ \left(\cone{P} \cap \Z^{n+1}\right)= \cone{P}^\circ \cap \Z^{n+1} \, .
\]
Equivalently, $P$ is reflexive if, possibly after translation by an integer vector, the origin is contained in $P^\circ$ and the geometric dual (or polar body) of $P$ is also a lattice polytope.
  
There are many interesting questions about polytopes that are IDP and/or reflexive.
For example, it is typically difficult to identify when a given polytope is IDP.
As another example, the number of reflexive polytopes of a fixed dimension is not known in dimension five and above.
Our interest is determining when a reflexive simplex in $\sQ$ is IDP.
One motivation for this paper is the conjecture, which appeared first in the survey \cite{brentisurvey}, that the $h$-polynomial of a standard graded Gorenstein domain is unimodal.
(Recall that a polynomial $p = p_0+p_1x+\cdots+p_dx^d\in\Z_{\geq0}[x]$ is {\em unimodal} if there exists $t\in[d]\cup\{0\}:=\{1,\ldots,d\}\cup\{0\}$ such that $p_0\leq \cdots\leq p_t\geq \cdots\geq p_d$.)
It is known that the Ehrhart $h^\ast$-polynomial of an IDP lattice polytope is the $h$-polynomial of a corresponding standard graded semigroup algebra, and that reflexive polytopes produce Gorenstein algebras.
Thus, $h^\ast$-polynomials of IDP reflexive polytopes yield a special case of this conjecture.
It was previously known, due to work of Bruns and R\"{o}mer~\cite{BrunsRomer}, that if a reflexive lattice polytope $P$ admits a regular unimodular triangulation (implying IDP), then the $h^\ast$-polynomial of $P$ is unimodal.
A proof that the $h^\ast$-polynomial of an IDP reflexive lattice polytope is unimodal has recently been announced in a preprint by Adiprasito, Papadakis, Petrotou, and Steinmeyer~\cite{beyondpositivityehrhart}.
Thus, it is of interest to determine when a reflexive lattice polytope is IDP.

\subsection{Our Contributions}

In Section~\ref{sec:simplices}, we review properties of the weighted projective space simplices $\Pq$, for $\bq\in \Z_{\geq 1}^n$, that correspond to weighted projective spaces in which one projective coordinate has weight $1$ and the others are given by the entries in $\bq$.

In Section~\ref{sec:idp}, we state our main result, Theorem~\ref{thm:3suppclassify}, a classification of all IDP reflexive $\Pq$ where $\bq$ has three distinct entries. 
Because the proof of Theorem~\ref{thm:3suppclassify} is long and technical, we provide the proof in Section~\ref{sec:proofs}.
We prove Theorem~\ref{thm:xIDPrestriction}, which gives a bound on the possible entries of an IDP reflexive $\bq$ with fixed multiplicities on the distinct entries.
Motivated by these results and further experiments, we propose in Conjecture~\ref{conj:idpclassification} a classification of all reflexive IDP elements of $\sQ$.

In Section~\ref{sec:reflexivestabilizations}, we define for all $m\geq 1$ the $m$-th reflexive stabilization of  $\Pq$.
We prove that for $m \geq 2$, reflexive stabilizations produce non-IDP simplices, and for large $m$ they have non-unimodal Ehrhart $h^*$-polynomials with all coefficients taking values in $\{1,2\}$.
Finally, we show that reflexive stabilizations can be used to produce interesting examples of lattice simplices; for example, Theorem~\ref{thm:almostidpunimodal} provides an example of a $\Pq$ that is simultaneously ``almost IDP'' and ``almost $h^\ast$-unimodal''.


\section{Properties of Simplices in $\sQ$}\label{sec:simplices}

\subsection{Reflexivity and $\sQ$}
The simplices in $\sQ$ admit a natural parametrization (or stratification), based on the distinct entries in the vector $\bq$, that allows us to test for the reflexive and IDP conditions even when $\Delta_{(1,q)}$ is high-dimensional and/or has large volume.
Given a vector of distinct positive integers $\br = (r_1,\ldots,r_d)\in\Z_{>0}^d$, write
\[
  (r_1^{x_1},r_2^{x_2},\ldots,r_d^{x_d}):=(\underbrace{r_1,r_1,\ldots,r_1}_{x_1\text{ times}},\underbrace{r_2,r_2,\ldots,r_2}_{x_2\text{ times}},\ldots,\underbrace{r_d,r_d,\ldots,r_d}_{x_d\text{ times}}) \, .
\]
If $\bq=(q_1,\ldots,q_n)=(r_1^{x_1},r_2^{x_2},\ldots,r_d^{x_d})$, we say that both $\bq$ and $\Pq$ are \emph{supported} by the vector $\br = (r_1,\ldots,r_d)$, which has distinct entries, with \emph{multiplicity} $\bx=(x_1, \dots, x_d)$.
  We write $\bq=(\br,\bx)$ in this case and say that $\bq$ is \emph{$d$-supported}.

  We are particularly interested in the case when the simplex $\Delta_{(1,\bq)}$ is reflexive, and we say that $\bq$ is reflexive whenever $\Pq$ is reflexive.
  The following theorem and setup provide us with a number-theoretic basis for studying reflexive simplices in $\sQ$.

\begin{theorem}\cite{conrads}
\label{thm: conrads}
The simplex $\Delta_{(1,\bq)}\in\sQ$ is reflexive if and only if
  \begin{equation}
    q_i \text{ divides } 1 + \sum_{j =1}^n q_j \quad\text{for all $1 \le i \le n$ } \, .
    \label{equ:conrads}
  \end{equation}
  Equivalently, if $\bq=(\br,\bx)$, then $\Pq$ is reflexive if and only if $\lcm{r_1,\dots, r_d}$ divides $1+\sum_{i=1}^d x_i r_i$.
\end{theorem}

\begin{setup}
  \label{setup1}
  Let $\bq$ be reflexive and supported by the vector $\br=(r_1, \dots, r_d) \in \Z_{\geq1}^d$ with multiplicity $\bx=(x_1, \dots, x_d) \in  \Z_{\geq 1}^d$.
  Let $\ell = \ell(\bq)$ be the integer defined by
  \begin{equation}
    1+\sum_{j=1}^d x_jr_j=\ell \cdot \lcm{r_1,r_2,\dots,r_d} \, .
    \label{equ:elldefn}
  \end{equation}
  Finally, we define $\bs := (s_1, \dots, s_d)$ where
  \begin{equation}
    s_i := \frac{\lcm{r_1, \dots, r_d}}{r_i}
    \label{equ:sdefn}
  \end{equation}
  for each $1 \le i \le d$.
\end{setup}

This setup provides useful restrictions on $\bq$, such as the following lemma.

\begin{lemma}[Braun, Liu, \cite{BraunLiu}]
  \label{lem:gcdlcm}
  In Setup~\ref{setup1}, we have that $\gcd(r_1, \dots, r_d) = 1$ and thus
  \begin{equation}
    \lcm{s_1, \dots, s_d}= \lcm{r_1, \dots, r_d}.
    \label{equ:lcmeq}
  \end{equation}
\end{lemma}

\begin{remark}
  The Hermite normal form for the simplex $\Pq-e_1$ is
  \[
    \left[
      \begin{array}{cccccc}
        0 & 1 & 0 & \cdots & 0 & N(\bq)-q_2 \\
        0 & 0 & 1 & \cdots & 0 & N(\bq)-q_3 \\
        \vdots & \vdots & \vdots & \ddots & \vdots & \vdots \\
        0 & 0 & 0 & \cdots & 1 &N(\bq)-q_d \\
        0 & 0 & 0 & \cdots & 0 & N(\bq)
      \end{array}\right] \, ,
  \]
  that is, $\Pq$ is unimodularly equivalent to the convex hull of the columns of this matrix.
  This is an example of a simplex with what Hibi, Higashitani, and Li~\cite{hibihermite} call a one-column Hermite normal form.
  Their work provides a partial study of Ehrhart-theoretic properties of such simplices.
\end{remark}

\subsection{The Integer Decomposition Property and $\sQ$}

For each $\br$-vector, it is known~\cite{BraunDavisSolusIDP} that there are infinitely many reflexive $\Pq$'s supported on $\br$.
Given a pair $\Pq$ and $\Pp$, both reflexive and IDP, Braun and Davis~\cite{BraunDavisReflexive} proved that a new reflexive IDP $\Delta_{(1,\by)}$ can be constructed as shown in the following theorem.
Suppose that $\Pq\subset \R^n$ and $\Pp\subset\R^m$ are reflexive and the vertices of $\Pp$ are labeled as $v_0,v_1,\ldots,v_m$.
For every $i=0,1,\ldots,m$, define the \emph{affine free sum}
\[
  \Pq\ast_i \Pp:=\conv{(\Pq \times 0^m)\cup(0^n\times \Pp-v_i)}\subset\R^{n+m}.
\]
The notion of an affine free sum can be generalized significantly~\cite{beckjayawantmcallister}, but in this article it will not be necessary.

\begin{theorem}[Braun, Davis \cite{BraunDavisReflexive}]
  \label{thm:freesumdecomp}
The simplex $\Pq$ is reflexive and arises as a free sum $\Pp*_0\Pw$ if and only if $\Pp$ and $\Pw$ are reflexive and $\bq=(\bp,(1+\sum_ip_i)\bw)$.
If $\Pp$ is IDP reflexive and $\Pq$ is IDP, then $\Pq*_0\Pp$ is IDP.
Further, if $\Pp$ and $\Pq$ are reflexive, IDP, and $h^*$-unimodal (defined in Subsection~\ref{sec:ehrhart}), then so is $\Pq*_0\Pp$.
\end{theorem}

Thus, there are infinitely many reflexive IDP $\Pq$'s that arise as a result of the affine free sum operation.
However, the support vector for $\Pq*_0\Pp$ is distinct from that of $\bp$ and $\bq$, so this operation does not respect the stratification of $\sQ$ given by support vectors.
In fact, for many $\br$-vectors, it is impossible to generate infinitely many reflexive IDP $\Pq$'s supported on $\br$, as the following theorem shows.

\begin{theorem}[Braun, Davis, Solus \cite{BraunDavisSolusIDP}]
  \label{thm:finitelymanyidp}
  Given a support vector $\br\in  \Z_{\geq 1}^d$, if there exists some $j<d$ such that $r_j\nmid r_d$, then only finitely many reflexive IDP $\Pq$'s are supported on $\br$.
\end{theorem}

Computational experiments suggest that IDP reflexive $\Pq$ satisfying the criteria in Theorem~\ref{thm:finitelymanyidp} are rare.
Specifically, consider all $\br$-vectors that are partitions of $M\leq 75$ with distinct entries, such that there exist some $r_j$ such that $r_j\nmid r_d$.
Table~\ref{tab:idpcount} shows that only 509 IDP reflexives are supported on $\br$-vectors of this type.
While this suggests that IDP reflexive $\Pq$'s are rare, it is important to keep in mind that this represents a relatively small sample set of simplices. 
For example,
  \[
\bq=(210,211,211,211,211,\underbrace{1055,1055,\ldots,1055}_{41\text{ times}}) 
\]
is not among this sample, but it is both IDP and reflexive with $210 \nmid 1055$.

\begin{table}[ht]
  \centering
  \begin{tabular}{|c|c|}
    \hline
    \# of $\br$-vectors with some $r_j\nmid r_d$ & \# of IDP reflexives supported by these\\\hline
    501350 & 509\\\hline 
  \end{tabular}
  \caption{Experimental results for $\br$-vectors that are partitions of $M\leq 75$.}
  \label{tab:idpcount}
  \end{table}

Fortunately, the following theorem provides a number-theoretic characterization of the IDP property for reflexive $\Pq$.

\begin{theorem}[Braun, Davis, Solus~\cite{BraunDavisSolusIDP}]
  \label{thm:idpreflexives}
  The reflexive simplex $\Pq$ is IDP if and only if for every $j=1,\ldots,n$, for all $b=1,\ldots,q_j-1$ satisfying
  \begin{equation}
    \label{eqn 1}
    b\left(\Frac{1+\sum_{i\neq j}q_i}{q_j} \right) - \sum_{i\neq j}\left\lfloor \Frac{bq_i}{q_j} \right\rfloor \geq 2 
  \end{equation}
  there exists a positive integer $c<b$ satisfying the following equations, where the first is considered for all  $1\leq i\leq n$ with $i\neq j$:
  \begin{equation}
    \label{eqn 2}
    \left\lfloor \Frac{bq_i}{q_j} \right\rfloor - \left\lfloor \Frac{cq_i}{q_j} \right\rfloor = \left\lfloor \Frac{(b-c)q_i}{q_j} \right\rfloor, \text{  and}
  \end{equation}
  \begin{equation}
    \label{eqn 3}
    c\left(\Frac{1+\sum_{i\neq j}q_i}{q_j} \right) - \sum_{i\neq j}\left\lfloor \Frac{cq_i}{q_j} \right\rfloor = 1.
  \end{equation}
\end{theorem}

The next corollary of Theorem~\ref{thm:idpreflexives} provides a necessary condition for a reflexive $\Pq$ to be IDP.
This condition is an essential tool in our study of IDP reflexive elements of $\sQ$.

\begin{corollary}[Braun, Davis, Solus~\cite{BraunDavisSolusIDP}]
  \label{cor:idpnecessary}
  If $\Pq$ is reflexive and IDP, then for all $j=1,2,\ldots,n$, 
  \[
    1+\sum_{i=1}^n(q_i \bmod q_j) = q_j 
  \]
  or equivalently
  \[
    1+\sum_{i=1}^nx_i(r_i \bmod r_j) = r_j \, . 
  \]

\end{corollary}

\begin{definition}
  \label{def:idpnecessary}
  If $\bq$ satisfies one (hence both) of these equations for all $j=1,\dots,n$, we say $\bq$ satisfies the \emph{necessary condition} for IDP.
\end{definition}
 Note that if $\bq$ satisfies the necessary condition, then $\Pq$ must be reflexive.

\subsection{Ehrhart theory and $\sQ$}\label{sec:ehrhart}

The \emph{Ehrhart function} of $P$ is the lattice point enumerator $i(P;t):=|tP\cap\Z^n|$, where $tP:=\{t\mathbf{p}:\mathbf{p}\in P\}$ denotes the $t^{th}$ dilate of the polytope $P$.  
It is well-known \cite{Ehrhart} that $i(P;t)$ is a polynomial in $t$ of degree $d=
\dim(P)$.  
The \emph{Ehrhart series} of $P$ is the rational function
\[
	\ehr_P(z) := \sum_{t\geq0}i(P;t)z^t = \frac{h_0^*+h_1^*z+\cdots+h_d^*z^d}{(1-z)^{\dim(P)+1}},
\]
where the coefficients $h^\ast_0,h^\ast_1,\ldots,h^\ast_d$ are all nonnegative integers \cite{StanleyDecompositions}.  
The polynomial $h^\ast(P;z):=h_0^*+h_1^*z+\cdots+h_d^*z^d$ is called the \emph{(Ehrhart) $h^\ast$-polynomial} of $P$.

\begin{remark}
A result of Stanley \cite{StanleyGreenBook} states that $P$ is Gorenstein if and only if $h^\ast(P;z)$ is symmetric with respect to its degree $s$; i.e., $h_i^\ast = h_{s-i}^\ast$ for all $i$.
Hibi~\cite{HibiDualPolytopes} proved the special case that $P$ is reflexive if and only if $h_i^\ast = h_{d-i}^\ast$ for all $i$.
Further, Bruns and R\"omer~\cite{BrunsRomer} proved that if $P$ is IDP and Gorenstein, then $h^\ast(P;z)$ is the $h^*$-polynomial of an IDP reflexive polytope.
Thus, reflexive polytopes are identifiable via Ehrhart theory and play a key role in the general study of Gorenstein lattice polytopes.
\end{remark}

Unlike the situation for general lattice polytopes, for $\Pq$ there is an explicit formula for $h^*(\Pq;z)$.

\begin{theorem}[Braun, Davis, Solus~\cite{BraunDavisSolusIDP}]
  \label{thm:hstarq}
  The $h^*$-polynomial of $\Pq$ is 
  \[
    h^*(\Pq;z) = \sum_{b=0}^{q_1+\cdots +q_n}z^{w(\bq,b)}
  \]
  where 
  \begin{equation}\label{eqn:weight}
    w(\bq,b):=b-\sum_{i=1}^n\left\lfloor\frac{q_ib}{1+q_1+\cdots +q_n} \right\rfloor \, .
  \end{equation}
\end{theorem}

\begin{example}
  Let $\bq=(2,2,3)$. Then $\displaystyle w(\bq,b):=b-2\left\lfloor\frac{b}{4}\right\rfloor-\left\lfloor\frac{3b}{8}\right\rfloor$, and thus 
  \[
    h^*(\Pq;z)=1+2z+4z^2+z^3.
  \]
\end{example}

For the case where $\Pq$ is IDP reflexive, we will need to know that the $h^*$-polynomial of $\Pq$ often admits a geometric series as a factor, as the following definition and theorem demonstrate.

\begin{definition}\label{def:ellpoly}
  Suppose $\br, \bx$, $\ell$ and $\bs$ are as given in Setup~\ref{setup1}.
  We define 
  \[
    g_\br^\bx(z):=\sum_{0\leq \alpha < \lcm{r_1,\ldots,r_d}}z^{u(\alpha)}
  \]
  where
  \[
    u(\alpha) = u_\br^\bx(\alpha):=\alpha \ell - \sum_{i=1}^d x_i\left\lfloor  \frac{\alpha}{s_i} \right\rfloor \, .
  \]
\end{definition}

\begin{theorem}[Braun, Liu ~\cite{BraunLiu}]\label{thm:ellgeomfactor}
  Using the notation in Setup~\ref{setup1}, we have that
  \[
    h^*(\Pq;z) = \left(\sum_{t=0}^{\ell-1}z^t\right) \cdot g_\br^\bx(z) \,.
  \]
\end{theorem}

\begin{example}
  For $\bq=(1^7, 3^4, 5^5)$, we have
  \[
    (z^{2} + z + 1)(x^{14} + x^{11} + 2x^{10} + 2x^8 + 3x^7 + 2x^6 + 2x^4 + x^3 + 1) \, .
  \]
  Note that in this case, $\ell = 3$ and a factor of $z^{2} + z + 1$ appears in the $h^*$-polynomial.
\end{example}


\section{Classifying IDP reflexive $\Pq$}\label{sec:idp}

Our goal in this section is to classify all reflexive IDP elements of $\sQ$ that are supported on up to three distinct entries.
We begin by observing that the necessary condition for IDP allows us to deduce the following refinement of Theorem~\ref{thm:finitelymanyidp}.

\begin{theorem}\label{thm:xIDPrestriction}
  Let $(\br,\bx)=\bq\in \Z_{\geq 1}^n$, where $\bq$ has at least two distinct entries and $r_1<r_2<\cdots <r_d$.
  If $\Pq$ is reflexive and IDP, then
  \[
    x_i\leq r_{i+1}/r_i
  \]
  for all $i\leq d-1$.
  Further, if there exists some $j<d$ such that $r_j\nmid r_d$, then
  \[
    x_d\leq r_j/(r_d\bmod r_j) \, .
  \]
  Thus, if there exists some $j<d$ such that $r_j\nmid r_d$, then there are at most finitely many IDP reflexives supported on $\br$.
\end{theorem}

\begin{proof}
  Let $j<d$, and assume that $\Pq$ is reflexive and IDP.
  Then by Corollary~\ref{cor:idpnecessary}, we have
  \[
    x_jr_j\leq 1+\sum_{i=1}^dx_i(r_i\bmod r_{j+1})=r_{j+1} \, ,
  \]
  from which the first inequality follows.
  Similarly, if $r_j\nmid r_d$, then
  \[
    x_d(r_d\bmod r_j)\leq 1+\sum_{i=1}^dx_i(r_i\bmod r_{j})=r_{j} \, ,
  \]
  from which the second inequality follows.
\end{proof}

Theorem~\ref{thm:xIDPrestriction} indicates that there are important relationships between the multiplicity vector $\bx$ of $\bq$ and the support vector $\br$.
By shifting our primary focus to the multiplicity vector, we are able to give a complete classification of all reflexive IDP $\Pq$'s that are supported on up to $3$ distinct entries.
If $\bx$ has one entry, it is straightforward to prove the following.

    \begin{prop}
    \label{prop:1supp}
      For $\bq=(r_1^{x_1})$, if $\Pq$ is IDP reflexive, then $\bq=(1,1,\ldots,1)$.
    \end{prop}

    If $\bx$ has two entries, meaning that $\br$ has two distinct entries, the following theorem applies.
    
    \begin{theorem}[Braun, Davis, Solus~\cite{BraunDavisSolusIDP}]
      \label{thm:2suppclassify}
      For the vector $\bq=(r_1^{x_1},r_2^{x_2})$, $\Pq$ is IDP reflexive if and only if it satisfies the necessary condition.
      The following is a classification of all such vectors, for $x_1,x_2\geq 1$:
      \begin{enumerate}
      \item $q=(1^{x_1},(1+x_1)^{x_2})$
        \item $q=((1+x_2)^{x_1},(1+(1+x_2)x_1)^{x_2})$
        \end{enumerate}
      \end{theorem}
      
Note that in the first case $r_1\mid r_2$ while in the second case $r_1\nmid r_2$.
We can extend these results as follows to the $3$-supported case using Theorem~\ref{thm:xIDPrestriction} and Corollary~\ref{cor:idpnecessary}.
The proof of Theorem~\ref{thm:3suppclassify} is long and technical, so we include it separately in Section~\ref{sec:proofs}.

\begin{theorem}\label{thm:3suppclassify}
  Consider a $3$-supported vector $\bq=(\br,\bx)$ such that $\Pq$ satisfies the necessary condition given in Corollary~\ref{cor:idpnecessary}.
  If $\bx=(x_1,x_2,x_3)$ is the multiplicity vector, then $\br$ is of one of the following forms.
  \begin{enumerate}
	\item[(i)]   $\br = (1,1+x_1,(1+x_1)(1+x_2))$.
	\item[(ii)]  $\br = (1+x_2,1+x_1(1+x_2),(1+x_1(1+x_2))(1+x_2))$.
	\item[(iii)] $\br = ((1+x_2)(1+x_3),1+x_1(1+x_2)(1+x_3),(1+x_1(1+x_2)(1+x_3))(1+x_2))$.
	\item[(iv)]  $\br = (1,(1+x_1)(1+x_3),(1+x_1)(1+x_2(1+x_3)))$.
	\item[(v)]   $\br = (1+(1+x_3)x_2,(1+x_3)(1+x_1(1+(1+x_3)x_2)),(1+(1+(1+x_3)x_2)x_1)(1+(1+x_3)x_2))$.
	\item[(vi)]  $\br = ((1+x_3)(1+(1+x_3)x_2),(1+x_3)(1+x_1(1+x_3)(1+(1+x_3)x_2)),(1+(1+x_3)(1+(1+x_3)x_2)x_1)(1+(1+x_3)x_2))$.
	\item[(vii)] $\br = (1+x_3,(1+x_3)(1+x_1(1+x_3)),(1+(1+x_3)x_1)(1+(1+x_3)x_2))$.
	\item[(viii)] There exists some $k,s\geq 1$, where
		\[
			\br = (1+kx_2,(skx_2+s+k)(1+x_1(1+kx_2)),(1+x_1(1+kx_2))(1+x_2(skx_2+s+k)))
		\]
		and
		\[
			\bx = (x_1,x_2,skx_2+s-k+1) \, .
		\]
              \end{enumerate}
              Further, the first seven $\br$-vectors produce IDP $\Pq$'s, while (viii) does not.
\end{theorem}

Note that the first seven $\br$-vectors in Theorem~\ref{thm:3suppclassify} each correspond to a unique divisibility criteria for $\br=(r_1,r_2,r_3)$, as follows:
  \begin{enumerate}
  \item[\emph{(i)}]    $r_1\mid r_2, r_1\mid r_3, r_2\mid r_3$
  \item[\emph{(ii)}]   $r_1\nmid r_2, r_1\mid r_3, r_2\mid r_3$
  \item[\emph{(iii)}]  $r_1\nmid r_2, r_1\nmid r_3, r_2\mid r_3$
  \item[\emph{(iv)}]   $r_1\mid r_2, r_1\mid r_3, r_2\nmid r_3$
  \item[\emph{(v)}]    $r_1\nmid r_2, r_1\mid r_3, r_2\nmid r_3$
  \item[\emph{(vi)}]   $r_1\nmid r_2, r_1\nmid r_3, r_2\nmid r_3$
  \item[\emph{(vii)}]  $r_1\mid r_2, r_1\nmid r_3, r_2\nmid r_3$
  \item[\emph{(viii)}] $r_1\nmid r_2, r_1\mid r_3, r_2\nmid r_3$
  \end{enumerate}
  We see that \emph{(v)} and \emph{(viii)} share the same divisibility pattern, yet of these two families only \emph{(v)} contains IDP simplices.
Note that for each positive integer vector $\bx$ of length at most three, and for each divisibility condition on the support vector $\br$, there is at most one support vector $\br$ such that $\bq=(\br,\bx)$ is reflexive IDP.
We extend these observations to a general conjecture in the following manner.

Given a support vector $\br=(r_1,\ldots,r_d)$ and a naturally labeled poset $\Omega$ on $\{1,2,\ldots,d\}$, we say that $\br$ is \emph{division compatible} with $\Omega$ if we have $i<_\Omega j$ if and only if $r_i | r_j$.
For example, an $\br$-vector of the form \emph{(ii)} above would be division compatible with the poset $\Omega$ on $\{1,2,3\}$ having relations $1<_\Omega 3$ and $2<_\Omega 3$.

\begin{conjecture}\label{conj:idpclassification}
Given $\bx=(x_1,x_2,\ldots,x_d)\in \Z_{\geq 1}^d$ and a naturally labeled poset $\Omega$ on $\{1,2,\ldots,d\}$, there is a unique IDP reflexive $\Pq$ with multiplicity vector $\bx$ and support vector that is division compatible with $\Omega$.
\end{conjecture}


\section{Reflexive Stabilizations}\label{sec:reflexivestabilizations}

It is unlikely that a random integer partition $\bq$ is reflexive.
Further, Theorem~\ref{thm:xIDPrestriction} tells us that for many support vectors $\br$, $\br$ supports infinitely many reflexives but finitely many IDP reflexives.
Thus, it is interesting to consider both (i) how to assign reflexive elements of $\sQ$ to an arbitrary integer partition and (ii) the behavior of $h^*$-vectors for reflexive elements of $\sQ$ where at least one entry has a large multiplicity.
We can begin investigating both of these through a process we call \emph{reflexive stabilization}, which allows us to assign to each integer partition $\bq$ a sequence of reflexive simplices in $\sQ$.

\begin{definition}
  Let $\bq\in \Z_{\geq 1}^n$. 
  The \emph{first reflexive stabilization} of $\bq$, denoted $\rs(\bq)$, is the vector $(1,1,\ldots,1,\bq)$ such that $\Prsq$ is reflexive and the number of $1$'s prepended to $\bq$ is the minimum necessary for this condition to hold; if $\bq$ is reflexive, then no $1$'s are prepended.
  Note that $\rs(\bq)$ exists due to Theorem~\ref{thm: conrads}.
  We say the number of $1$'s prepended to $\bq$ in $\rs(\bq)$ is the \emph{reflexive stabilization number} of $\bq$, denoted $\rsn(\bq)$.
  Thus, we can write 
  \[
    \rs(\bq):=(1^{\rsn(\bq)},\bq) \, .
  \]
  More generally, the \emph{$m$-th reflexive stabilization} of $\bq$, denoted $\rs(\bq,m)$, is defined as
  \[
    \rs(\bq,m):=(1^{\rsn(\bq)+(m-1)\cdot\lcm{\bq}},\bq) \, .
  \]
\end{definition}

Note that by Theorem~\ref{thm: conrads}, when prepending $\lcm{\bq}$ copies of $1$ as many times as desired, the resulting simplex is reflexive.

\begin{example}
  Let $\bq=(2,2,3)$. Then $1+2+2+3=8$, and thus $\rs(\bq)=(1,1,1,1,2,2,3)$ is reflexive with $\rsn(\bq)=4$.
  Further, $\rs(\bq, 3)=(\br,\bx)=((1,2,3),(16,2,1))$.
\end{example}

We can restrict ourselves to only $\bq\in \Z_{\geq 2}^n$, since if $1$ is an entry of $\bq$, then $\rs(\bq)$ is $\rs(\bq',m)$ for some $m$ with $\bq'$ given by the entries of $\bq$ that are not equal to $1$.
We begin by observing that for $m\geq 2$, reflexive stabilizations are not IDP.

\begin{theorem}\label{thm:largemnotidp}
  Assume that $\bq\in \Z_{\geq 2}^n$.
  For $m\geq 2$, $\Prsmq$ is not IDP.
\end{theorem}

\begin{proof}
  For $\bq=(\br,\bx)$, we again set $\rs(\bq,m)=(\br(m),\bx(m))$.
  Assume that $\br$ is indexed from $1$ to $d$ and that $\br(m)$ is indexed from $0$ to $d$, i.e., that $r(m)_0=1$.
  We use Corollary~\ref{cor:idpnecessary} in the case where $j=1$.
  Observe that for all $m\geq 2$, we have
  \begin{align*}
     1+\sum_{i=0}^dx_i(r(m)_i\bmod r(m)_1) 
    = & \, 1+x(m)_0+\sum_{i=1}^dx_i(r_i\bmod r_1)\\
    > &\,  1+\lcm{\br}+\sum_{i=1}^dx_i(r_i\bmod r_1) 
    >  \, r_1 \, ,
  \end{align*}
  due to the growth of $x(m)_0$ as $m$ increases.
  Thus, $\Prsmq$ is not IDP.
  \end{proof}
  
We next show that for large $m$, the $h^\ast$-polynomial for reflexive stabilizations has coefficients from $\{1,2\}$. 
This is the motivation for the term \emph{reflexive stabilization}; as $m$ goes to infinity, the $h^\ast$-polynomial coefficient values stabilize to a fixed set.
  
\begin{theorem}\label{thm:largemnotunimodal}
  Assume that $\bq\in \Z_{\geq 2}^n$.
  For $m$ sufficiently large, $\Prsmq$ is not $h^*$-unimodal.
  Further, $h^*(\Prsmq)$ contains only $1$'s and $2$'s.
\end{theorem}

\begin{proof}
  For $\bq=(\br,\bx)$, set $\rs(\bq,m)=(\br(m),\bx(m))$.
  Assume that $\br$ is indexed from $1$ to $d$ and that $\br(m)$ is indexed from $0$ to $d$, i.e., that $r(m)_0=1$.
  Define $\ell(m)$ by
  \[
    1+\sum_{i=0}^dr(m)_i x(m)_i =1+x(m)_0+\sum_{i=1}^dr_    ix_i = \ell(m) \lcm{\bq}=\ell(m)\lcm{\br}\, .
  \]
  Using the fact that $r(m)_0=1$ implies $s(m)_0=\lcm{\br}$, Definition~\ref{def:ellpoly} gives 
  \[
    g_{\br(m)}^{\bx(m)}(z)=\sum_{0\leq \alpha < \lcm{\br}}z^{\alpha \ell(m)-\sum_{i=1}^dx_i\lfloor \alpha/s_i\rfloor} \, .
  \]
  Note that in the exponent $\displaystyle \alpha \ell(m)-\sum_{i=1}^dx_i\lfloor \alpha/s_i\rfloor$, the only term that varies with $m$ is $\ell(m)$.
  Note also that for all $\alpha$ we have $\displaystyle  \left(\sum_{i=1}^dx_i\right)\lcm{\br}\geq \sum_{i=1}^dx_i\lfloor \alpha/s_i\rfloor$.

  Assume that $m$ is large enough so that
  \[
    \ell(m)-1> \left(\sum_{i=1}^dx_i\right)\lcm{\br}\geq\sum_{i=1}^dx_i\lfloor \alpha/s_i\rfloor \, \text{ for all }\alpha \, .
  \]
  By Theorem~\ref{thm:ellgeomfactor}, we have that
  \begin{align*}
    h^*(\Prsmq;z) =&  \left(\sum_{t=0}^{\ell(m)-1}z^t\right) \cdot g_{\br(m)}^{\bx(m)}(z) 
    = \sum_{0\leq \alpha < \lcm{\br}}\left(\sum_{t=0}^{\ell(m)-1}z^{t+\alpha \ell(m)-\sum_{i=1}^dx_i\lfloor \alpha/s_i\rfloor}\right) \, .
  \end{align*}
  Our strategy is to show that for each value of $t$, the polynomial 
  \[
    p_\alpha(z):=\sum_{t=0}^{\ell(m)-1}z^{t+\alpha \ell(m)-\sum_{i=1}^dx_i\lfloor \alpha/s_i\rfloor}
  \]
  has common terms with only $p_{\alpha-1}(z)$ and $p_{\alpha+1}(z)$.
  Thus, the coefficients in $h^*(\Prsmq;z)$ are either $1$ or $2$.

  Observe that the largest term in $p_{\alpha-1}(z)$ has degree $\alpha\ell(m)-\sum_{i=1}^dx_i\lfloor (\alpha-1)/s_i\rfloor-1$, while the smallest term in $p_{\alpha+1}(z)$ has degree $(\alpha+1)\ell(m)-\sum_{i=1}^dx_i\lfloor(\alpha+1)/s_i\rfloor$.
  Because of our assumption regarding $m$, we have that
  \begin{align*}
    (\alpha+1)\ell(m)-\sum_{i=1}^dx_i\lfloor(\alpha+1)/s_i\rfloor \geq &  \, (\alpha+1)\ell(m)-(\ell(m)-1) \\
    = & \, \alpha\ell(m)+1 > \alpha\ell(m)\\
    > & \, \alpha\ell(m)-\sum_{i=1}^dx_i\lfloor (\alpha-1)/s_i\rfloor -1 \, .
  \end{align*}
  Thus, the lowest degree term in $p_{\alpha+1}(z)$ is strictly greater than the highest degree term in $p_{\alpha-1}(z)$.
  Further, the coefficient of $z^{\alpha\ell(m)}$ in $h^*(\Prsmq;z)$ is equal to $1$, which is contributed by $p_{\alpha}(z)$.
  Thus, all coefficients in $h^*(\Prsmq;z)$ are equal to either $1$ or $2$.

  To see that $h^*(\Prsmq;z)$ is not unimodal, observe that if $\alpha'$ is the least value of $\alpha$ such that $\sum_{i=1}^dx_i\lfloor \alpha/s_i\rfloor >0$, then the coefficient of $z^{\alpha'\ell(m)-1}$ is equal to $2$, while the coefficient of $z^{\alpha'\ell(m)}$ is equal to $1$.
\end{proof}

Theorems~\ref{thm:largemnotidp} and~\ref{thm:largemnotunimodal} imply that if reflexive stabilizations will result in simplices that are IDP and/or $h^\ast$-unimodal, we need to focus on the case $m=1$. 
Further, it makes sense to begin by considering the case where both $\rsn(\bq)$ and $\ell(\bq)$ are small.
For example, when $\rsn(\bq)=1$ and $\ell(\rs(\bq))=1$, we have the following curious result.

\begin{theorem}\label{thm:lcm-1}
  Suppose that $\rsn(\bq)=1$ and that $\ell(\rs(\bq))=1$, i.e., that $1 + \sum_{i =1}^n q_i = \lcm{\bq}-1$.
  Then for all $0\leq b\leq \sum_iq_i$, we have $w(\bq,b)=w(\rs(\bq),b)$ and hence
  \[
    h^*(\Prsq;z)=h^*(\Pq;z)+z^{n+1} \, .
  \]
\end{theorem}

\begin{proof}
Let $\bq$ be reflexive and supported by the vector $\br=(r_1, \dots, r_d) \in \Z_{\geq1}^d$ with multiplicity $\bx=(x_1, \dots, x_d) \in  \Z_{\geq 1}^d$. Using Setup~\ref{setup1}, we have that $r_i=\lcm{\bq}/s_i$.
  By Theorem~\ref{thm:hstarq}, it follows that 
  \[
    w(\bq,b):=b-\sum_{i=1}^dx_i\left\lfloor\frac{r_ib}{\lcm{\bq}-1}\right\rfloor\,=b-\sum_{i=1}^dx_i\left\lfloor\frac{\lcm{\bq}b}{s_i(\lcm{\bq}-1)}\right\rfloor\,
  \]
  with $0\leq b\leq\lcm{\bq}-2$. Rewriting $\lcm{\bq}b$ as $(\lcm{\bq}-1)b+b$ implies that
  \[
    w(\bq,b):=b-\sum_{i=1}^dx_i\left\lfloor\frac{b}{s_i}+\frac{b}{s_i(\lcm{\bq}-1)}\right\rfloor\,.
  \]
  Additionally, it follows from Theorem~\ref{thm:hstarq} that
  \[
    w(\rs(\bq),b):=b-\sum_{i=1}^dx_i\left\lfloor\frac{r_ib}{\lcm{\bq}}\right\rfloor\,=b-\sum_{i=1}^dx_i\left\lfloor\frac{b}{s_i}\right\rfloor\,
  \]
  with $0\leq b\leq\lcm{\bq}-1$. For all $0\leq b\leq\lcm{\bq}-2$, we have
  \[
    0\leq\frac{1}{s_i}\cdot\frac{b}{\lcm{\bq}-1}<\frac{1}{s_i}
  \qquad
  \mbox{ implies }
  \qquad
    \frac{b}{s_i}\leq\frac{b}{s_i}+\frac{b}{s_i(\lcm{\bq}-1)}<\frac{b+1}{s_i}\,  .
  \]
  If we let $b=ks_i+a$ with some integer $k$ and $0\leq a<s_i$, then $\left\lfloor\frac{b}{s_i}\right\rfloor=k$. Since 
  \[
    \frac{b}{s_i(\lcm{\bq}-1)}<\frac{1}{s_i},
  \qquad
  \mbox{ then }
  \qquad
    \left\lfloor\frac{b}{s_i}\right\rfloor=\left\lfloor\frac{b}{s_i}+\frac{b}{s_i(\lcm{\bq}-1)}\right\rfloor.
  \]This implies that $w(\bq,b)=w(\rs(\bq),b)$ for all $0\leq b\leq\lcm{\bq}-2$.
  When considering $\rs(\bq)$, for $b=\lcm{\bq}-1$ we have $w(\rs(\bq),b)=n+1$, and our proof is complete.
\end{proof}

It would be interesting to determine whether or not there are similar additive structures for the $h^*$-polynomial in other cases where $\rsn(\bq)$ and $\ell(\rs(\bq))$ are small.
Note that while $h^\ast$-unimodality is typically broken in reflexive stabilizations for $m>1$ by Theorem~\ref{thm:lcm-1}, it can be preserved for reflexive stabilizations with $m=1$ as the following example illustrates.

\begin{example}
For $\bq=(4,4,5,5)$, we have that $1 + \sum_{i =1}^n q_i = \lcm{\bq}-1$, and
\[
h^*(\Prsq;z)=1+2z+7z^2+7z^3+2z^4+z^5
\]
with
\[
h^*(\Pq;z)=1+2z+7z^2+7z^3+2z^4 \, .
\] 
\end{example}

It is interesting that among the $\Pq$ simplices, we can find examples that are ``near the boundary'' of both $h^*$-unimodality and IDP using reflexive stabilizations.
One sequence of examples is the following.
Recall that for a lattice polytope $P$, the \emph{Hilbert basis} of $\cone{P}$ is the minimal generating set of the monoid $\cone{P}\cap \Z^{n+1}$.
Thus, $P$ is IDP if and only if the Hilbert basis of $\cone{P}$ consists of the elements at height $1$ in $\cone{P}$, i.e. $(1,P)\cap \Z^{n+1}$.

\begin{theorem}\label{thm:almostidpunimodal}
  For $n\geq 1$, define $\br(n)=(1, 3n, 10n, 15n)$ and $\bx(n)=(2n-1,1,1,1)$.
  Thus, $\bq(n):=(\br(n),\bx(n))=\rs((3n,10n,15n))$.
  For $\bq=(\br(n),\bx(n))$, let $V(n)=\{(1,\bv):\bv\text{ a vertex of }\Pq\}$.
  The Hilbert basis for $\cone{\Pq}$ consists of $V(n)$ and the columns of the following matrix (where the height coordinate is the first entry):
  \[
    \begin{bmatrix}
      1 & 1&1 & 1&1 & 1&1 & 2&2\\
      0&0&0&0&0&0&0&0&0\\
      \vdots &\vdots &\vdots &\vdots &\vdots &\vdots &\vdots &\vdots &\vdots \\
      0&0&0&0&0&0&0&0&0\\
      0 & 0 & 0 &0 & 0 & 0 & -1 & -2& -3\\
      0 & 0 & -1&-1&-2&-3&-4&-7&-10\\
      0 & -1 &-1&-2&-3&-5&-6&-10&-15\\
    \end{bmatrix}
  \]
  Thus, there are exactly two elements in the Hilbert basis of height greater than $1$, both of which are at height $2$.
  For $n\geq 2$ with $\bq=(\br(n),\bx(n))$, we have
  \[
    h^*(\Pq;z)=(1+z^2+z^4+z^6+\cdots+z^{2n-2})\cdot (1+7z+14z^2+7z^3+z^4) \, ,
  \]
  which has coefficient vector
  \[
  (1, 7, 15, 14, 16, 14,16,14,\ldots, 16, 14, 16, 14, 15, 7, 1)\, .
  \]
\end{theorem}

\begin{proof}
  We first prove that the $h^*$-polynomial is correct.
  Fix $n\geq 2$ and set $\bq=(\br(n),\bx(n))$.
  Since $\lcm \bq=30n$, it follows from Theorem~\ref{thm:hstarq} that for $0\leq b<30n$
  \begin{align*}
    w(\rs(\bq),b):=&b-(2n-1)\left\lfloor\frac{b}{30n}\right\rfloor-\left\lfloor\frac{3n\cdot b}{30n}\right\rfloor-\left\lfloor\frac{10n\cdot b}{30n}\right\rfloor-\left\lfloor\frac{15n\cdot b}{30n}\right\rfloor \\
    = & \, b-\left\lfloor\frac{b}{10}\right\rfloor-\left\lfloor\frac{b}{3}\right\rfloor-\left\lfloor\frac{b}{2}\right\rfloor\,.
  \end{align*}
  Let $b=30\alpha+\beta$, with $0\leq \alpha<n$ and $0\leq\beta<30.$
  Then 
  \begin{align*}
    w(\rs(\bq),b):=&30\alpha+\beta-\left\lfloor\frac{30\alpha+\beta}{10}\right\rfloor-\left\lfloor\frac{30\alpha+\beta}{3}\right\rfloor-\left\lfloor\frac{30\alpha+\beta}{2}\right\rfloor \\
    =&\, 2\alpha+\beta-\left\lfloor\frac{\beta}{10}\right\rfloor-\left\lfloor\frac{\beta}{3}\right\rfloor-\left\lfloor\frac{\beta}{2}\right\rfloor.
  \end{align*}
  This implies that
  \[
    \sum_{0\leq b<30n}z^{w(b)}=\left(\sum_{0\leq \alpha<n}z^{2\alpha}\right)\left(\sum_{0\leq \beta<30}z^{\beta-\left\lfloor\frac{\beta}{10}\right\rfloor-\left\lfloor\frac{\beta}{3}\right\rfloor-\left\lfloor\frac{\beta}{2}\right\rfloor}\right)
  \]
  which evaluates to
  \[
    (1+z^2+z^4+z^6+\cdots+z^{2n-2})\cdot (1+7z+14z^2+7z^3+z^4) \, .
  \]
  
  We next prove our claim regarding the Hilbert basis for $\Pq$.
  For each $0\leq b<30n$, there exist unique values $0\leq \alpha, \beta,\gamma<30n$ such that the following linear combination of ray generators for $\Pq$ produces an element in the fundamental parallelepiped for the cone over $\Pq$:
  \[
          \bp(b):=
    \left[\begin{array}{ccccccc}
            1& 0 & 0 & \cdots & 0 & 0 & -1\\
            0 & 1 & 0 & \cdots & 0 & 0 & -1\\
            \vdots & \ddots & \ddots & \ddots & \vdots & \vdots & \vdots\\
            0 & 0 & 0 & \ddots & 0 & 0 & -3n\\
            0 & 0 & 0 & \cdots & 1 & 0 & -10n\\
            0 & 0 & 0 & \cdots & 0 & 1 & -15n\\
            1 & 1 & 1 & \cdots & 1 & 1 & 1             
            \end{array}
          \right]
          \left[\begin{array}{c}
                  b/30n\\
                  b/30n\\
                  \vdots\\
                  \alpha/30n\\
                  \beta/30n\\
                  \gamma/30n\\
                  b/30n
            \end{array}
          \right]\in \Z^{2(n+1)}
  \]
  From the final four entries of the resulting vector, we observe that each of the following values must be an integer:
  \[
\frac{\alpha-3bn}{30n}, \, \frac{\beta-10bn}{30n},\, \frac{\gamma-15bn}{30n},\,  \frac{2nb+\alpha+\beta+\gamma}{30n}
    \]
    To determine $\alpha$, we write $b=10p_{10}+r_{10}$ where $0\leq r_{10}<10$ and evaluate in the first expression, yielding
    \[
\frac{\alpha-30p_{10}n-3r_{10}n}{30n}\in \Z
\]
from which it follows that $\alpha=3nr_{10}=3n(b\bmod 10)$ and
\[
\frac{\alpha-30p_{10}n-3r_{10}n}{30n}=p_{10}=\lfloor b/10\rfloor \, .
\]
Similarly, it can be shown (by expressing $b=3p_{3}+r_{3}$) that $\beta=10n(b\bmod 3)$ with
\[
\frac{\beta-10bn}{30n}=p_{3}=\lfloor b/3\rfloor \, .
\]
Similarly, $\gamma=15n(b\bmod 2)$ and
\[
\frac{\gamma-15bn}{30n}=\lfloor b/2\rfloor \, .
\]
Thus, we have that
\[
  \bp(b)=\left[
    \begin{array}{c}
                  0\\
                  \vdots\\
                  0\\
                  -\lfloor b/10\rfloor \\
                  -\lfloor b/3\rfloor\\
                  -\lfloor b/2\rfloor\\
                  ( 2b+3(b\bmod 10)+10(b\bmod 3) + 15(b\bmod 2) )/30n
             \end{array}
             \right]\, .
  \]
  We now show that we can reduce $\bp(b)$ for $30<b< 30n$ to a case where $0\leq b\leq 30$.
  If $30<b$, then
  \begin{align*}
\bp(b)-\bp(30)=& \left[
    \begin{array}{c}
                  0\\
                  \vdots\\
                  0\\
                  -\lfloor b/10\rfloor +3\\
                  -\lfloor b/3\rfloor+ 10\\
                  -\lfloor b/2\rfloor+15\\
                  ( 2b+3(b\bmod 10)+10(b\bmod 3) + 15(b\bmod 2) )/30n -2
             \end{array}
           \right]\\
           =& 
           \left[
    \begin{array}{c}
                  0\\
                  \vdots\\
                  0\\
                  -\lfloor (b-30)/10\rfloor \\
                  -\lfloor (b-30)/3\rfloor\\
                  -\lfloor (b-30)/2\rfloor\\
                  ( 2b+3((b-30)\bmod 10)+10((b-30)\bmod 3) + 15((b-30)b\bmod 2) )/30n
             \end{array}
           \right]\\
           =&     \bp(b-30) \, .
    \end{align*}
    Thus, every lattice point in $\Pq$ is a sum of elements from $\{\bp(b):0\leq b\leq 30\}$.
    It is straightforward but tedious to check (either by hand or via computer algebra system) that the minimal additive generators of this set are precisely the the columns of the matrix in the theorem statement.
    Thus, our proof is complete.
\end{proof}

While reflexive stabilizations are one way to obtain elements of $\sQ$ that exhibit extremal behavior for heights of Hilbert basis elements and $h^*$-unimodality, the following example demonstrates that this phenomenon occurs outside of reflexive stabilizations as well.

\begin{example}\label{ex:crazyhilbert}

 for $n\geq 2$, let
\[
  \bq=((n,(2n-1)(n+1),2n(n+1)),(1,1,2(n-1))) \, .
  \]
For $n\leq 20$, it has been verified that
  \begin{align*}
    h^*(\Pq;z)=(1, (n+1)^2, &(2n+1)(n+1)+1,(2n+1)(n+1), (2n+1)(n+1)+1,\ldots ,\\  &(2n+1)(n+1)+1,(2n+1)(n+1), (2n+1)(n+1)+1, (n+1)^2, 1) 
  \end{align*}
  and that the Hilbert basis of $\cone{\Pq}$ consists of the points $(1,\Pq)\cap \Z^{2n+1}$ together with the following the lattice point at height two (as given by the first coordinate):
  \[
    (2,-1,-2n-1,-2(n-1),-2(n-1),-2(n-1),\ldots,-2(n-1))^T \, .
  \]
  Thus, this family of simplices is another example of polytopes on the boundary of both IDP and $h^*$-unimodality. These simplices are more arithmetically complicated than those given in Theorem~\ref{thm:almostidpunimodal}.
\end{example}

\section{Proof of Theorem~\ref{thm:3suppclassify}}
\label{sec:proofs}

\subsection{Proof of derivation of $(i)$-$(viii)$ in Theorem~\ref{thm:3suppclassify}}
We first suppose that $(\br,\bx)$ satisfies the necessity condition, and show that the resulting $\br$-vectors must be of one of the eight types listed.
Since $r_1 < r_2 < r_3$, reducing modulo $r_3$ gives
$
r_3=1+r_1x_1+r_2x_2 \, .
$
If we next consider the modulo $r_2$ necessary condition, then substituting the above for $r_3$ and simplifying gives
$
r_2=1+x_1r_1+x_3((1+r_1x_1)\bmod r_2) \, .
$
The challenge here is that we would like to specify $r_2$ using this formula, but it involves a remainder which could fluctuate.
The key observation is to recall that if the necessary condition for IDP holds, then Theorem~\ref{thm:xIDPrestriction} implies $x_1\leq r_2/r_1$.
Thus, we have
\begin{equation}
  1+x_1r_1\leq 1+(r_2/r_1)r_1=1+r_2 \, .\label{eqn:3supcases}
\end{equation}
There are now three cases to consider.

\textbf{Case 1:} Suppose we have equality in~\eqref{eqn:3supcases}.
		It is immediate that in this case
		\[
			(1+r_1x_1)\bmod r_2=(1+r_2)\bmod r_2=1\, ,
		\]
		and thus
		$
			r_2=1+x_1r_1+x_3 =r_2+1+x_3 \, .
		$
		As $x_3\geq 1$, this yields a contradiction, and thus this case does not occur. \\

\textbf{Case 2:} Consider if $1+x_1r_1= r_2$ in~\eqref{eqn:3supcases}.
		Then 
		\[
			(1+r_1x_1)\bmod r_2=r_2\bmod r_2=0\, .
		\]
		Thus, we have
		$
			(\br,\bx)=((r_1,1+x_1r_1,r_3),(x_1,x_2,x_3)) \, .
		$
		However, we know that
		\[
			r_3=1+r_1x_1+r_2x_2=(1+x_1r_1)(1+x_2) \, ,
		\]
		and thus
		$
			(\br,\bx)=((r_1,1+x_1r_1,(1+x_1r_1)(1+x_2)),(x_1,x_2,x_3)) \, .
		$
		If $r_1=1$, then the result is
		$
			(\br,\bx)=((1,1+x_1,(1+x_1)(1+x_2)),(x_1,x_2,x_3)) \, .
		$
		which corresponds to $(i)$ in our theorem statement.
		If $r_1\geq 2$, we consider our necessary condition modulo $r_1$ and obtain
		$
			r_1 = 1 + x_2 + x_3((1+x_2) \bmod r_1) \, .
		$
		Since (by this equality) we have $1+x_2\leq r_1$, it follows that there are two subcases. 

\textbf{Subcase 2.1:}
			Suppose $1+x_2=r_1$.
			Then our vector is 
			\[
				(\br,\bx)=((1+x_2,1+x_1(1+x_2),(1+x_1(1+x_2))(1+x_2)),(x_1,x_2,x_3)) \, ,
			\]
			yielding $(ii)$ in our theorem statement. \\
		
\textbf{Subcase 2.2:}
			Suppose $1+x_2<r_1$.
			Then 
			\[\begin{aligned}
				r_1 &= 1 + x_2 + x_3((1+x_2) \bmod r_1) = 1 + x_2 + x_3(1+x_2) =(1+x_2)(1+x_3)\, .
			\end{aligned}\]
			Then our vector is 
			\[\begin{aligned}
				(\br,\bx)= & (((1+x_2)(1+x_3),\\
					& \quad 1+x_1(1+x_2)(1+x_3),\\
					& \quad (1+x_1(1+x_2)(1+x_3))(1+x_2)),(x_1,x_2,x_3)) \, ,
			\end{aligned}\]
			yielding $(iii)$ in our theorem statement.

\textbf{Case 3:}
		If $1+x_1r_1< r_2$ in~\eqref{eqn:3supcases}, it is immediate that
		$
			(1+r_1x_1)\bmod r_2=1+r_1x_1\, ,
		$
		and thus
		$
			r_2=1+x_1r_1+(1+r_1x_1)x_3 =(1+x_3)(1+x_1r_1)\, .
		$
		Combining this with
		\[\begin{aligned}
			r_3&=1+r_1x_1+r_2x_2 =1+r_1x_1+(1+x_3)(1+x_1r_1)x_2 =(1+r_1x_1)(1+(1+x_3)x_2) \, ,
		\end{aligned}\]
		we obtain that
		\[
			(\br,\bx)=((r_1,(1+x_3)(1+x_1r_1),(1+r_1x_1)(1+(1+x_3)x_2)),(x_1,x_2,x_3)) \, ,
		\]
		which is a function of the multiplicities and the value $r_1$.
		If $r_1=1$, then we obtain
		\[
			(1,(1+x_1)(1+x_3),(1+x_1)(1+x_2(1+x_3)))
		\]
		which corresponds to item $(iv)$ in our theorem statement.
		If $r_1\geq 2$, we again consider the necessary condition modulo $r_1$, for which we obtain
		\begin{align*}
			r_1&=1+x_2((1+x_3)(1+x_1r_1) \bmod r_1)+x_3((1+r_1x_1)(1+(1+x_3)x_2)\bmod r_1)\\
			& = 1+x_2((1+x_3)\bmod r_1)+x_3((1+(1+x_3)x_2) \bmod r_1).
		\end{align*}
		We now have three subcases to consider.

\textbf{Subcase 3.1:}
		If $1+x_3<r_1$ then
		\[
			r_1=1+x_2(1+x_3)+x_3((1+(1+x_3)x_2) \bmod r_1) \, .
		\]
		This requires two subsubcases.

\textbf{Subsubcase 3.1.1:}
				Suppose $(1+(1+x_3)x_2) = r_1$.
				Then $x_3$ is arbitrary, and we have
				\[\begin{aligned}
					(\br,\bx)=&((1+(1+x_3)x_2,(1+x_3)(1+x_1(1+(1+x_3)x_2)), \\
						& \quad (1+(1+(1+x_3)x_2)x_1)(1+(1+x_3)x_2)),(x_1,x_2,x_3)) \, ,
				\end{aligned}\]
				establishing item $(v)$ in the theorem statement. 

			\textbf{Subsubcase 3.1.2:}
				Suppose $(1+(1+x_3)x_2) < r_1$.
				Then
				\[
					r_1=1+x_2(1+x_3)+x_3(1+(1+x_3)x_2) =(1+x_3)(1+(1+x_3)x_2)\, ,
				\]
				and thus it follows that for $\bx=(x_1,x_2,x_3)$, we have
				\[\begin{aligned}
					\br=&((1+x_3)(1+(1+x_3)x_2),\\
						&\quad(1+x_3)(1+x_1(1+x_3)(1+(1+x_3)x_2)),\\
						&\quad (1+(1+x_3)(1+(1+x_3)x_2)x_1)(1+(1+x_3)x_2)) \, ,
				\end{aligned}\]
				establishing item $(vi)$ in the theorem statement.
		
\textbf{Subcase 3.2:}
		Suppose $1+x_3= r_1$.
		Then for $\bx=(x_1,x_2,x_3)$ we have
		\[
			\br = (1+x_3,(1+x_3)(1+x_1(1+x_3)),(1+(1+x_3)x_1)(1+(1+x_3)x_2)) \, ,
		\]
		establishing item $(vii)$ in the theorem statement. \\
    
\textbf{Subcase 3.3:}
		Suppose $1+x_3> r_1$.
		If $r_1\mid (1+x_3)$, then we have
		\[
			r_1 = 1+x_2((1+x_3)\bmod r_1)+x_3((1+(1+x_3)x_2) \bmod r_1) = 1+x_3 \, .
		\]
		This is a contradiction, and thus it follows that $r_1\nmid 1+x_3$.
		If $r_1\mid x_2$, then
		\[
		r_1=1+x_2((1+x_3)\bmod r_1)+x_3\, ,
		\]
		which implies that $r_1\mid (1+x_3)$, again a contradiction.
		Thus, we must have that $r_1\nmid x_2$, and since $r_1\nmid (1+x_3)$ we also know $r_1>x_2$ since
		\[
			r_1 = 1+x_2((1+x_3)\bmod r_1)+x_3((1+(1+x_3)x_2) \bmod r_1) >x_2 \, .
		\]
		We now consider two subsubcases.

\textbf{Subsubcase 3.3.1:}
				Suppose $r_1\mid (1+x_2(1+x_3))$.
				Then
				\[
 					r_1 = 1+x_2((1+x_3)\bmod r_1) \, .
				\]
				Thus, there exists some $k\geq 1$ where $r_1=1+x_2k$, and we are forced to have
				\[
					r_1 =1+x_2k= 1+x_2((1+x_3)\bmod r_1) \, 
				\]
				implying that $k=(1+x_3)\bmod 1+x_2k$.
				Thus, for some $s\geq 1$ we set
				\[
					x_3=skx_2+s-k+1 
				\]
				and we have the case
				\[\begin{aligned}
					\br =& (1+kx_2,(2+skx_2+s-k)(1+x_1(1+kx_2)),\\
						& \quad (1+x_1(1+kx_2))(1+x_2(2+skx_2+s-k)))
				\end{aligned}\]
				with
				\[
					\bx = (x_1,x_2,skx_2+s-k+1) \, ,
				\]
				corresponding to item $(viii)$ in our theorem statement. \\
				
\textbf{Subsubcase 3.3.2:}
				Suppose $r_1\nmid (1+x_2(1+x_3))$, so we have
				\[
					r_1 = 1+x_2((1+x_3)\bmod r_1)+x_3((1+(1+x_3)x_2) \bmod r_1) \, .
				\]
				We also have that $1+x_3>r_1>x_2$.
				However, $r_1\nmid (1+x_2(1+x_3))$ implies that 
				\[
					r_1 = 1+x_2((1+x_3)\bmod r_1)+x_3((1+(1+x_3)x_2) \bmod r_1) \geq 1+x_3\, ,
				\]
				yielding a contradiction. 
				
This completes our analysis of possible cases based on the necessary condition.
In particular, all of the types listed yield reflexive simplices.

\subsection{Proof of IDP for types $(i)$, $(ii)$, and $(iv)$ and of non-IDP for type $(viii)$ in Theorem~\ref{thm:3suppclassify}}
We now show that type $(viii)$ $\br$-vectors yield non-IDP simplices, where we apply Theorem~\ref{thm:idpreflexives} and the notation therein.
Let $q_j=r_2=(skx_2+s+k)(1+x_1(1+kx_2))$, and set $b=skx_2+s+k$ which is strictly less than $r_2$ as needed for Theorem~\ref{thm:idpreflexives}.
It is tedious but straightforward to reduce the left-hand side of~\eqref{eqn 1} to the form $x_3+(x_3-1)x_2(skx_2+s+k)$.
Note that by the assumption of type $(viii)$, we have
\[
x_3=s+1+k(sx_2-1)\geq 2
\]
and thus 
\[
x_3+(x_3-1)x_2(skx_2+s+k)\geq 2+x_2(skx_2+s+k)\geq 2 \, ,
\]
satisfying~\eqref{eqn 1}.
Setting $q_i=r_3$, we now ask if there is a solution $0<c<b$ satisfying~\eqref{eqn 2} and~\eqref{eqn 3}.
Using the fact that we must have $0<c<b=skx_2+s+k$, we obtain that the left-hand side of~\ref{eqn 2} is:
\begin{align*}
& \left\lfloor \frac{(skx_2+s+k)(1+x_1(1+kx_2))(1+x_2(skx_2+s+k))}{(skx_2+s+k)(1+x_1(1+kx_2))}\right\rfloor -  \left\lfloor \frac{c(1+x_2(skx_2+s+k))}{skx_2+s+k}\right\rfloor \\
  = & \, 1+x_2(skx_2+s+k) -cx_2 - \left\lfloor \frac{c}{skx_2+s+k}\right\rfloor \\
  = & \, 1+x_2(skx_2+s+k-c)
\end{align*}
On the other hand, using $0<c<b=skx_2+s+k$, we find that the right-hand side of~\ref{eqn 2} is:
\begin{align*}
& \left\lfloor \frac{(skx_2+s+k-c)(1+x_2(skx_2+s+k))}{(skx_2+s+k)}\right\rfloor \\
  = & \, x_2(skx_2+s+k-c) + \left\lfloor 1-\frac{c}{skx_2+s+k}\right\rfloor \\
  = & \, x_2(skx_2+s+k-c)
\end{align*}
Thus, there is no $c$ value in this range that satisfies~\eqref{eqn 2}, and hence we find that simplices of type $(viii)$ are not IDP.

We next show that simplices of types $(i)$-$(vii)$ are IDP.
Types $(i)$, $(ii)$, and $(iv)$ all follow from affine free sum decompositions as follows.
For $(i)$, observe that
\[
\br=(1^{x_1})*_0(1^{x_2})*_0(1^{x_3}) \, ,
\]
and thus Theorem~\ref{thm:freesumdecomp} applies.
For $(ii)$, observe that
\[
\br=((1+x_1)^{x_2},(1+x_1(1+x_2))^{x_2})*_0(1^{x_3}) \, ,
  \]
and thus Theorem~\ref{thm:2suppclassify} and Theorem~\ref{thm:freesumdecomp} apply to finish this case.
  Finally, for $(iv)$ observe that
  \[
\br=(1^{x_3})*_0((1+x_2)^{x_3},(1+x_2(1+x_3))^{x_3}) \, ,
\]
and thus again Theorem~\ref{thm:2suppclassify} and Theorem~\ref{thm:freesumdecomp} apply to finish this case.

Types $(iii)$, $(v)$, $(vi)$, and $(vii)$ do not follow from affine free sum decompositions, and thus we must use Theorem~\ref{thm:idpreflexives} directly.
 Throughout the remainder of this proof, we use the notation
  \[
    h(b):=b\left(\Frac{1+\sum_{i\neq j}q_i}{q_j} \right) - \sum_{i\neq j}\left\lfloor \Frac{bq_i}{q_j} \right\rfloor 
  \]
  to denote the left-hand side of~\eqref{eqn 1}.

\subsection{Proof of IDP for type $(iii)$ in Theorem~\ref{thm:3suppclassify}}
    
  We next verify IDP for $\br$-vectors of type $(iii)$ using Theorem~\ref{thm:idpreflexives}.
  We must consider three cases corresponding to three possible values of $q_j$.
  
  \textbf{Case: $q_j=(1+x_2)(1+x_3)$}. 
  It is straightforward to verify that 
  \[
    h(b)=b-x_3\left\lfloor \frac{b}{(1+x_3)}\right\rfloor \, ,
  \]
  and using this formula one can check that
  \[
    h(k(1+x_3))=k \, .
  \]
  Combining these two observations, it follows that $h(b)=1$ only when $b=1$ and $b=(1+x_3)$, thus identifying the $b$-values we are required to check in \eqref{eqn 1}.
  To verify that~\eqref{eqn 2} always has the desired solution, we consider three cases.
  If $q_i=(1+x_2)(1+x_3)$, the result is trivial.
  If $q_i=x_1(1+x_2)(1+x_3)+1$, then we may select $c=1$, from which it follows that both sides of~\eqref{eqn 2} are equal to $x_1(b-1)$.
  If $q_i=(1+x_2)(x_1(1+x_2)(1+x_3)+1)$, then we set $c=(1+x_3)$, from which it is straightforward to compute that both sides of~\eqref{eqn 2} are equal to $b(1+x_2)x_1-x_1(1+x_2)(1+x_3)-1+\lfloor b/(1+x_3) \rfloor$.
  This completes our first case.

  \textbf{Case: $q_j=x_1(1+x_2)(1+x_3)+1$}.
  It is straightforward to verify that 
  \[
    h(b)=b-x_1\left\lfloor \frac{b(1+x_2)(1+x_3)}{x_1(1+x_2)(1+x_3)+1} \right\rfloor \, ,
  \]
  where the values of $b$ range from $1$ to $x_1(1+x_2)(1+x_3)$.
  To verify that~\eqref{eqn 2} always has the desired solution, we consider three cases.
  If $q_i=x_1(1+x_2)(1+x_3)+1$, the result is trivial.
  If $q_i=(1+x_2)(1+x_3)$, then we write $b = \alpha x_1 + \beta$, where $0\leq \beta < x_1$ and $0\leq \alpha \leq (1+x_2)(1+x_3)$ for $\alpha,\beta \in \Z$.
  Consequently, we have
  \begin{align*}
    \left\lfloor \frac{b(1+x_2)(1+x_3)}{x_1(1+x_2)(1+x_3) + 1}\right\rfloor &=  \left\lfloor \frac{(\alpha x_1 + \beta)(1+x_2)(1+x_3)}{x_1(1+x_2)(1+x_3) + 1}\right\rfloor \\
    &= \alpha +  \left\lfloor \frac{\beta(1+x_2)(1+x_3) - \alpha}{x_1(1+x_2)(1+x_3) + 1}\right\rfloor \\
    &= \alpha + \left. \begin{cases}\phantom{-}0, & \beta > 0 \\ -1, & \beta = 0 \end{cases} \right\} .
  \end{align*}
  In the case that $\beta>0$, we have that $h(b) = b - x_1\alpha = \beta$.
  Hence, the viable candidates for $c$-values from 1 to $x_1(1+x_2)(1+x_3)$ that satisfy equation~\eqref{eqn 3} are precisely those $c$ such that $c\equiv 1\bmod x_1$. 
  Therefore, the $b$-values we are required to check in~\eqref{eqn 1} are all $b = \alpha x_1 + \beta$, where $2\leq \beta \leq x_1-1$. 
  In this case, we may choose $c = \alpha x_1 + 1$, from which it follows that both sides of~\eqref{eqn 2} are equal to 0, giving the desired result.
  On the other hand, if $\beta = 0$, then we have that $h(b) = h(\alpha x_1) = \alpha x_1 - x_1(\alpha - 1) = x_1$. If $x_1=1$, then $h(b) = 1$.
  Thus, we need only consider when $x_1 > 1$. 
  In order to satisfy~\eqref{eqn 1}, it must be that $\alpha > 0$.
  Given that $x_1 > 1$ and $\alpha >0$, it is straightforward to verify that both sides of~\eqref{eqn 2} when $c=1$ are equal to $\alpha - 1$. 
  Finally, if $q_i=(1+x_2)(x_1(1+x_2)(1+x_3)+1)$, then we can set $c=1$ and the result is immediate.
  This completes our second case.

  \textbf{Case: $q_j=(1+x_2)(x_1(1+x_2)(1+x_3)+1)$}.
  We first identify those values of $b$ that satisfy~\eqref{eqn 1} and~\eqref{eqn 3}.
  It is straightforward to verify that
  \[
    h(b)=b-x_1\left\lfloor \frac{b(1+x_3)}{x_1(1+x_2)(1+x_3)+1}\right\rfloor -((1+x_2)-1)\left\lfloor \frac{b}{(1+x_2)} \right\rfloor \, .
  \]
  Writing $b=m(1+x_2)+t$ where $0\leq m\leq x_1(1+x_2)(1+x_3)$ and $0\leq t\leq (1+x_2)$, it follows that
  \[
    h(b)=h(m(1+x_2)+t)=m+t-x_1\left\lfloor \frac{m(1+x_2)(1+x_3)+t (1+x_3)}{x_1(1+x_2)(1+x_3)+1} \right \rfloor \, .
  \]
  We can now further divide into cases: either we have $m=x_1(1+x_2)(1+x_3)$ or we have $m=kx_1+w$ where $0\leq k<(1+x_2)(1+x_3)$ and $0\leq w<x_1$, which yields
  \[
    h(b)=h((kx_1+w)(1+x_2)+t)=w+t-x_1\left\lfloor \frac{w(1+x_2)(1+x_3)+t (1+x_3) -k}{x_1(1+x_2)(1+x_3)+1}\right\rfloor \, .
  \]
  For $m\neq x_1(1+x_2)(1+x_3)$, observe that since $0\leq w(1+x_2)(1+x_3)\leq x_1-1$ and $0\leq t (1+x_3)<(1+x_2)(1+x_3)$, with $0\leq k<(1+x_2)(1+x_3)$, we have that $0\leq w(1+x_2)(1+x_3)+t (1+x_3)<x_1(1+x_2)(1+x_3)$.
  Thus, $\left\lfloor \frac{w(1+x_2)(1+x_3)+t (1+x_3) -k}{x_1(1+x_2)(1+x_3)+1}\right\rfloor$ is equal to either $0$ or $-1$.

  \textbf{Subcase 1 of 3:} Suppose $m=x_1(1+x_2)(1+x_3)$.
  Since $0\leq t (1+x_3)<(1+x_2)(1+x_3)$, we have
  \[
    h((1+x_2)(1+x_3)x_1(1+x_2)+t)=t-x_1\left\lfloor \frac{t (1+x_3)-(1+x_2)(1+x_3)}{x_1(1+x_2)(1+x_3)+1}\right\rfloor = t+x_1.
  \]
  If this is equal to $1$, then it must be that $t=0$ and $x_1=1$.
  Thus, if $x_1=1$, we have that $h((1+x_2)(1+x_3)(1+x_2))=1$.

  \textbf{Subcase 2 of 3:} Suppose now that $m\neq x_1(1+x_2)(1+x_3)$ and that
  \[
    \left\lfloor \frac{w(1+x_2)(1+x_3)+t (1+x_3) -k}{x_1(1+x_2)(1+x_3)+1}\right\rfloor = -1 \, .
  \]
  Then since $w,t\geq 0$ and $x_1\geq 1$, we have $h(b)=w+t +x_1=1$ which forces $w=t=0$ and $x_1=1$.
  In this case, $h(k(1+x_2))=1$ any time that $k>0$.
  Thus, if $x_1=1$, we have that $h(k(1+x_2))=1$ when $0<k<(1+x_2)(1+x_3)$.

  \textbf{Subcase 3 of 3:} Suppose again that $m\neq x_1(1+x_2)(1+x_3)$ and that
  \[
    \left\lfloor \frac{w(1+x_2)(1+x_3)+t (1+x_3) -k}{x_1(1+x_2)(1+x_3)+1}\right\rfloor = 0 \, .
  \]
  Then $0\leq k\leq w(1+x_2)(1+x_3)+t (1+x_3)$, which implies that either (A) $0<w<x_1$ with $0\leq t <(1+x_2)$ or (B) $w=0$ with $k\leq t (1+x_3)$.
  If (A) holds, then $h(b)=w+t=1$ forces $w=1$ and $t=0$ since $w>0$, which means that $h(b)=1$ when $b=(kx_1+1)(1+x_2)$ for $0\leq k <(1+x_2)(1+x_3)$.
  If (B) holds, then our same equation forces $w=0$ and $t=1$ when $k\leq (1+x_3)$, which means that $h(b)=1$ when $b=kx_1(1+x_2)+1$ for $0\leq k<(1+x_2)(1+x_3)$ and $k\leq (1+x_3)$.

  We summarize the values of $b$ for which $h(b)=1$ that were just derived:
  \begin{itemize}
  \item If $x_1=1$ and $0<k\leq (1+x_2)(1+x_3)$ we have $b=k(1+x_2)$. 
  \item If $x_1\geq 1$ and $0\leq k<(1+x_2)(1+x_3)$, we have $b=(kx_1+1)(1+x_2)$.
  \item If $x_1\geq 1$ and $0\leq k<(1+x_2)(1+x_3)$ and $k\leq (1+x_3)$, we have $b=kx_1(1+x_2)+1$.
  \end{itemize}

  Our next goal is to establish that~\eqref{eqn 2} is always satisfied; recall that we are in the case where $q_j=(1+x_2)(x_1(1+x_2)(1+x_3)+1)$.
  If $q_i=(1+x_2)(x_1(1+x_2)(1+x_3)+1)$, then~\eqref{eqn 2} is trivially satisfied.
  If $q_i=x_1(1+x_2)(1+x_3)+1$, we write $b=m(1+x_2)+t$ where $0\leq m<x_1(1+x_2)(1+x_3)+1$ and $0\leq t<(1+x_2)$.
  Substituting this form of $b$ into~\eqref{eqn 2} yields the equation
  \[
    -\left\lfloor \frac{c}{(1+x_2)}\right\rfloor =\left\lfloor \frac{t -c}{(1+x_2)}\right\rfloor \, .
  \]
  If $b>(1+x_2)$, we set $c=(1+x_2)$ and the equation is satisfied.
  If $2<b<(1+x_2)$, then we set $c=1$ and the equation is satisfied.

  If $q_i=(1+x_2)(1+x_3)$, the analysis becomes more complicated.
  We write $b=m(1+x_2)+t$ where $0\leq m<d$ and $0\leq t <(1+x_2)$.
  Our argument will proceed by considering $x_1=1$ and $x_1>1$ separately.

  Suppose $x_1=1$.
  Then the left-hand-side of~\eqref{eqn 2} is reduced to
  \[
    \left\lfloor \frac{t (1+x_3)-m}{(1+x_2)(1+x_3)+1}\right \rfloor - \left \lfloor \frac{c(1+x_3)}{(1+x_2)(1+x_3)+1}\right \rfloor
  \]
  and the right-hand-side to
  \[
    \left\lfloor \frac{t (1+x_3)-m-c(1+x_3)}{(1+x_2)(1+x_3)+1}\right \rfloor \, .
  \]
  Since $m<(1+x_3)$, if $t (1+x_3)-m<0$ this forces $t=0$ and $0<m$, thus $b$ is a multiple of $(1+x_2)$, and we found earlier that $h(m(1+x_2))=1$.
  Thus, we need proceed no further.
  If $t (1+x_3)-m\geq 0$, then since $m<(1+x_3)$ we must have $t\geq 1$, and we also have $t (1+x_3)-m<(1+x_2)(1+x_3)$.
  Thus, $\left\lfloor \frac{t (1+x_3) -m}{(1+x_2)(1+x_3)+1}\right\rfloor =0$, from which it follows that~\eqref{eqn 2} reduces to
  \[
    -\left\lfloor \frac{c(1+x_3)}{(1+x_2)(1+x_3)+1}\right\rfloor = \left\lfloor \frac{t (1+x_3) -m -c(1+x_3)}{(1+x_2)(1+x_3)+1}\right\rfloor \, .
  \]
  If $m=0$, set $c=1$ and this equation is solved.
  If $m\geq 1$, set $c=m(1+x_2)+1$ which is less than $b$ in this case, and this equation is again satisfied.
  This completes our proof for $x_1=1$.

  We next consider when $x_1\geq 2$, maintaining our previous notation of $b=m(1+x_2)+t$.
  Write $m=fx_1+g$ where $0\leq f\leq \lfloor (1+x_3)/x_1\rfloor$ and $0\leq g<x_1$ except in the case where $f=\lfloor (1+x_3)/x_1\rfloor$ in which case $g$ is bounded above by $(1+x_3)-(1+x_3)\lfloor (1+x_3)/x_1\rfloor$.
  This leads to the left-hand-side of~\eqref{eqn 2} having the form
  \[
    f+\left\lfloor \frac{g(1+x_2)(1+x_3)+t (1+x_3) -f}{x_1(1+x_2)(1+x_3)+1}\right\rfloor -\left\lfloor \frac{c(1+x_3)}{x_1(1+x_2)(1+x_3)+1}\right\rfloor 
  \]
  while the right-hand-side has the form
  \[
    f+\left\lfloor \frac{g(1+x_2)(1+x_3)+t (1+x_3) -f-c(1+x_3)}{x_1(1+x_2)(1+x_3)+1}\right\rfloor \, .
  \]
  We thus need to solve the equation
  \begin{align*}
    & \, \left\lfloor \frac{g(1+x_2)(1+x_3)+t (1+x_3) -f}{x_1(1+x_2)(1+x_3)+1}\right\rfloor -\left\lfloor \frac{c(1+x_3)}{x_1(1+x_2)(1+x_3)+1}\right\rfloor \\
    = & \, \left\lfloor \frac{g(1+x_2)(1+x_3)+t (1+x_3) -f-c(1+x_3)}{x_1(1+x_2)(1+x_3)+1}\right\rfloor
  \end{align*}
  subject to the constraints $0\leq g<x_1$ (with the exception mentioned above), $0\leq t \leq (1+x_2)$, and $0 \leq f\leq \lfloor (1+x_3)/x_1\rfloor$.
  Note that the first two inequalities imply that $0\leq g(1+x_2)(1+x_3)+t (1+x_3)<x_1(1+x_2)(1+x_3)$, and also $f\leq \lfloor (1+x_3)/x_1\rfloor \leq (1+x_2)(1+x_3)$, hence 
  \[
    \left\lfloor \frac{g(1+x_2)(1+x_3)+t (1+x_3) -f}{x_1(1+x_2)(1+x_3)+1}\right\rfloor = 
    \left\{
      \begin{array}{ll}
        0 & \text{ if } g(1+x_2)(1+x_3)+t (1+x_3)\geq f \\
        -1 & \text{ if } g(1+x_2)(1+x_3)+t (1+x_3) < f
      \end{array}
    \right.
  \]

  \textbf{Subcase 1 of 2:} 
  Suppose $g(1+x_2)(1+x_3)+t (1+x_3)-f\geq 0$.
  Then~\eqref{eqn 2} reduces to 
  \[
    -\left\lfloor \frac{c(1+x_3)}{x_1(1+x_2)(1+x_3)+1}\right\rfloor = \left\lfloor \frac{g(1+x_2)(1+x_3)+t (1+x_3) -f-c(1+x_3)}{x_1(1+x_2)(1+x_3)+1}\right\rfloor \, .
  \]
  Note that $f\leq \lfloor (1+x_3)/x_1\rfloor <(1+x_3)$, and thus we can set $c=fx_1(1+x_2)+1$ which is less than $b$.
  The left-hand-side of our above equation is given by
  \begin{align*}
    & \, -\left\lfloor \frac{(fx_1(1+x_2)+1)(1+x_3)}{x_1(1+x_2)(1+x_3)+1}\right\rfloor \\
    = & \, -\left\lfloor \frac{fx_1(1+x_2)(1+x_3)+(1+x_3)}{x_1(1+x_2)(1+x_3)+1}\right\rfloor \\
    = & \, -\left\lfloor \frac{fx_1(1+x_2)(1+x_3)+f+(1+x_3)-f}{x_1(1+x_2)(1+x_3)+1}\right\rfloor \\
    = & \, -f \, .
  \end{align*}
  Similarly, the right-hand-side of our equation is given by
  \begin{align*}
    & \, -f+\left \lfloor \frac{f-(1+x_3)+g(1+x_2)(1+x_3)+t (1+x_3)-f}{x_1(1+x_2)(1+x_3)+1}\right\rfloor \\
    = & \,-f+\left \lfloor \frac{-(1+x_3)+g(1+x_2)(1+x_3)+t (1+x_3)}{x_1(1+x_2)(1+x_3)+1}\right\rfloor \, .
  \end{align*}
  Since $h(b)$ is assumed to be at least $2$, we have that one or both of $g$ and $t$ are non-zero.
  Combining this observation with $g(1+x_2)(1+x_3)+t (1+x_3)-f\geq 0$ it follows that $g(1+x_2)(1+x_3)+t (1+x_3)> 0$.
  Note that $(1+x_3)\mid (g(1+x_2)(1+x_3)+t (1+x_3))$, and thus $x_1(1+x_2)(1+x_3)+1>g(1+x_2)(1+x_3)+t (1+x_3) -(1+x_3)\geq 0$, which forces the right-hand-side of our equation to equal $-f$, satisfying~\eqref{eqn 2}.

  \textbf{Subcase 2 of 2:} 
  Suppose $g(1+x_2)(1+x_3)+t (1+x_3)-f<0$.
  Note that since $g,(1+x_2)(1+x_3),t,(1+x_3)\geq 0$, it follows that $f\geq 1$ and thus $b=(fx_1+g)(1+x_2)+t\geq (1+x_2)$.
  Set $c=(1+x_2)$, which is less than $b$.
  With these conditions, the left-hand-side of~\eqref{eqn 2} is easily seen to equal $-1$.
  The right-hand-side of~\eqref{eqn 2} is given by
  \[
    \left\lfloor \frac{g(1+x_2)(1+x_3)+t (1+x_3)-f-(1+x_2)(1+x_3)}{x_1(1+x_2)(1+x_3)+1}\right\rfloor \, .
  \]
  Since $g(1+x_2)(1+x_3)+t (1+x_3) -f<0$ and $-(1+x_2)(1+x_3)<0$, the numerator above is strictly negative.
  Also, since $g,(1+x_2)(1+x_3),t, (1+x_3)\geq 0$, the numerator is minimized by $-f-(1+x_2)(1+x_3)>-(1+x_3)-(1+x_2)(1+x_3)\geq -2(1+x_2)(1+x_3)$.
  But, since we assumed that $x_1\geq 2$, it follows that $x_1(1+x_2)(1+x_3)+1>2(1+x_2)(1+x_3)$ and thus the floor function above is equal to $-1$, satisfying equality for~\eqref{eqn 2}.
  This completes the proof establishing IDP for $\br$-vectors of type $(iii)$.

\subsection{Proof of IDP for type $(v)$ in Theorem~\ref{thm:3suppclassify}}
We next verify IDP for $\br$-vectors of type $(v)$ using Theorem~\ref{thm:idpreflexives}. 
  Again, we must consider three cases corresponding to three possible values of $q_j$.
  
  \textbf{Case: $q_j=1+(1+x_3)x_2$}. 
  It is straightforward to verify that 
 \begin{align*}
 h(b) = b - x_2\floor{\frac{b(1+x_3)}{1+(1+x_3)x_2}},
 \end{align*}
 where $1\leq b \leq (1+x_3)x_2$. 
  To verify that~\eqref{eqn 2} always has the desired solution, we consider three cases.
  If $q_i=1+(1+x_3)x_2$, the result is trivial.
  If $q_i=(1+x_3)(1+x_1(1+(1+x_3)x_2))$, then we write $b = \alpha x_2 + \beta$, where $0\leq \beta < x_2$ and $0\leq \alpha \leq 1+x_3$ for $\alpha,\beta \in \Z$. 
  Consequently, observe that
  \begin{align*}
  \floor{\frac{b(1+x_3)}{1+(1+x_3)x_2}} = \floor{\frac{(\alpha x_2 + \beta)(1+x_3}{1+(1+x_3)x_2}} = \alpha + \floor{\frac{\beta(1+x_3) - \alpha}{1+(1+x_3)x_2}} = \alpha + \left. \begin{cases} \phantom{-}0, & \beta > 0 \\ -1, & \beta = 0 \end{cases} \right\} .
\end{align*}  
In the case that $\beta > 0$, this formula implies that $h(b) = b - x_2\alpha = \beta$. 
Hence, the viable candidates for $c$-values that satisfy~\eqref{eqn 3} are precisely those $c$ such that $c \equiv 1\bmod x_2$.
Therefore, the $b$-values we are required to check in~\eqref{eqn 1} are all $b = \alpha x_2 + \beta$, where $2\leq \beta \leq x_2 - 1$.
In this case, we may choose $c = \alpha x_2 + 1$, from which it follows that both sides of~\eqref{eqn 2} are equal to $(\beta - 1)(1+x_3)x_1$.
On the other hand, if $\beta = 0$, then we have that $h(b) = h(\alpha x_2) = \alpha x_2 - x_2(\alpha - 1) = x_2$. If $x_2 = 1$, then $h(b) = 1$. 
Thus, we need only consider when $x_2 > 1$. 
Given that $x_2 > 1$ and $0\leq \alpha \leq 1+x_3$, it is straightforward to verify that both sides of~\eqref{eqn 2} when $c=1$ are equal to $(\alpha x_2 - 1)(1+x_3)x_1 + \alpha - 1$. 
  If $q_i=(1+(1+(1+x_3)x_2)x_1)(1+(1+x_3)x_2)$, then we may again set $c=1$, from which it is straightforward to compute that both sides of~\eqref{eqn 2} are equal to $(b-1)(1+(1+(1+x_3)x_2)x_1)$.
  This completes our first case.
  
  \textbf{Case: $q_j=(1+x_3)(1+x_1(1+(1+x_3)x_2))$}.
  We first identify those values of $b$ that satisfy~\eqref{eqn 1} and~\eqref{eqn 3}.
  It is straightforward to verify that
  \[
    h(b)=b-x_1\left\lfloor \frac{b(1+(1+x_3)x_2)}{(1+x_3)(1+x_1(1+(1+x_3)x_2))}\right\rfloor -x_3\left\lfloor \frac{b}{1+x_3} \right\rfloor \, .
  \]
  Writing $b=\alpha (1+x_3)+\beta$, where $0\leq \beta < 1+x_3$ and $0\leq \alpha \leq x_1(1+(1+x_3)x_2)$, it follows that
  \begin{align*}
    h(b) &= h(\alpha(1+x_3)+\beta) \\
    &= \alpha + \beta - x_1\left\lfloor \frac{\alpha(1+x_3)(1+(1+x_3)x_2) + \beta(1+(1+x_3)x_2)}{(1+x_3)(1+x_1(1+(1+x_3)x_2))}\right \rfloor  \, .
  \end{align*}
  We can now further divide into cases: either we have $\alpha = x_1(1+(1+x_3)x_2)$ or we have $\alpha = \gamma x_1 + \delta$ where $0\leq \gamma < 1+(1+x_3)x_2$ and $0\leq \delta < x_1$, which yields
  \begin{align*}
  h(b) &= h((\gamma x_1 + \delta)(1+x_3) + \beta) \\ 
  &= \delta + \beta - x_1 \floor{\frac{\delta(1+x_3)(1+(1+x_3)x_2) + \beta(1+(1+x_3)x_2) - \gamma(1+x_3)}{(1+x_3)(1+x_1(1+(1+x_3)x_2))}} \, .
  \end{align*}
  For $\alpha \neq x_1(1+(1+x_3)x_2)$, observe that since $0\leq \delta(1+x_3)(1+(1+x_3)x_2) \leq (x_1-1)(1+x_3)(1+(1+x_3)x_2)$ and $0\leq \beta(1+(1+x_3)x_2) < (1+x_3)(1+(1+x_3)x_2)$ with $0\leq \gamma(1+x_3)<(1+x_3)(1+(1+x_3)x_2)$, we have that $0\leq \delta(1+x_3)(1+(1+x_3)x_2) + \beta(1+(1+x_3)x_2) < x_1(1+x_3)(1+(1+x_3)x_2)$. 
  Thus, $\floor{\frac{\delta(1+x_3)(1+(1+x_3)x_2) + \beta(1+(1+x_3)x_2 - \gamma(1+x_3)}{(1+x_3)(1+x_1(1+(1+x_3)x_2))}}$ is equal to either $0$ or $-1$. 
  
  \textbf{Subcase 1 of 3:} Suppose $\alpha = x_1(1+(1+x_3)x_2)$. Since $0\leq \beta(1+(1+x_3)x_2) < (1+x_3)(1+(1+x_3)x_2)$, we  have
  \begin{align*}
  h((1+x_3)x_1(1+(1+x_3)x_2)+\beta) &= \beta - x_1\floor{\frac{\beta(1+(1+x_3)x_2) - (1+x_3)(1+(1+x_3)x_2)}{(1+x_3)(1+x_1(1+(1+x_3)x_2))}} \\
   &= \beta + x_1.
  \end{align*}
  If this is equal to 1, then it must be that $\beta = 0$ and $x_1 = 1$. Thus, if $x_1 = 1$, we have that $h((1+x_3)(1+(1+x_3)x_2)) = 1$.
  
  \textbf{Subcase 2 of 3:} Suppose now that $\alpha \neq x_1(1+(1+x_3)x_2)$ and that
  \begin{align*}
  \floor{\frac{\delta(1+x_3)(1+(1+x_3)x_2) + \beta(1+(1+x_3)x_2 - \gamma(1+x_3)}{(1+x_3)(1+x_1(1+(1+x_3)x_2))}} = -1.
  \end{align*}
  Then, writing $\alpha = \gamma x_1 + \delta$ where $0\leq \gamma < 1+(1+x_3)x_2$ and $0\leq \delta < x_1$, we have $h(b) = \delta + \beta + x_1$. 
Since $\delta,\beta \geq 0$ and $x_1 \geq 1$, for this to equal 1, we must have $\delta = \beta = 0$ and $x_1 = 1$. 
In this case, $h(\gamma(1+x_3)) = 1$ whenever $\gamma > 0$.
Thus, if $x_1=1$, we have that $h(\gamma(1+x_3)) = 1$ when $0<\gamma < 1+(1+x_3)x_2$. 
  
  \textbf{Subcase 3 of 3:} Suppose now that $\alpha \neq x_1(1+(1+x_3)x_2)$ and that
  \begin{align*}
  \floor{\frac{\delta(1+x_3)(1+(1+x_3)x_2) + \beta(1+(1+x_3)x_2 - \gamma(1+x_3)}{(1+x_3)(1+x_1(1+(1+x_3)x_2))}} = 0.
  \end{align*}
  Then, it follows that $0\leq \gamma(1+x_3)\leq \delta(1+x_3)(1+(1+x_3)x_2)+\beta(1+(1+x_3)x_2)$, which given the bounds on $\gamma$, $\delta$, and $\beta$, implies that either (A) $0<\delta<x_1$ or (B) $\delta = 0$ with $\gamma(1+x_3)\leq \beta(1+(1+x_3)x_2)$. 
 If (A) holds, then $h(b) = \delta + \beta = 1$ forces $\delta = 1$ and $\beta = 0$ since $\delta > 0$. 
Therefore, $h(b) = 1$ when $b=(\gamma x_1+1)(1+x_3)$ for $0\leq \gamma < 1+(1+x_3)x_2$. 
If (B) holds, then our same equation forces $\delta = 0$ and $\beta = 1$ when $\gamma(1+x_3) \leq 1+1(1+x_3)x_2$, which further implies $0\leq \gamma \leq x_2$. 
This means $h(b) = 1$ when $b=\gamma x_1(1+x_3)+1$ for $0\leq \gamma \leq x_2$.
  
  We summarize the values of $b$ for which $h(b) = 1$ that were just derived:
  \begin{itemize}
  \item If $x_1=1$ and $0<\gamma\leq 1+(1+x_3)x_2$, we have $b=\gamma(1+x_3)$.
  \item If $x_1\geq 1$ and $0\leq \gamma < 1+(1+x_3)x_2$, we have $b=(\gamma x_1+1)(1+x_3)$.
  \item If $x_1\geq 1$ and $0\leq \gamma < 1+(1+x_3)x_2$ and $\gamma < x_2$, we have $b = \gamma x_1(1+x_3)+1$.
  \end{itemize}
  Next, we establish that~\eqref{eqn 2} is always satisfied; recall that we are in the case where $q_j = (1+x_3)(1+x_1(1+(1+x_3)x_2))$. 
  If $q_i = (1+x_3)(1+x_1(1+(1+x_3)x_2))$, then~\eqref{eqn 2} is trivially satisfied. 
  Now suppose $q_i = (1+(1+(1+x_3)x_2)x_1)(1+(1+x_3)x_2)$. Note that $b\neq 1+x_3$ since $h(1+x_3)=1$. 
  If $b>1+x_3$, we set $c=1+x_3$ from which it is straightforward to compute that both sides of~\eqref{eqn 2} are equal to $bx_2 - (1+(1+x_3)x_2) + \floor{\frac{b}{1+x_3}}$.
  Otherwise, if $b<1+x_3$, we may choose $c=1$ from which it follows that both sides of~\eqref{eqn 2} are equal to $(b-1)x_2$ since $1\leq b < 1+x_3$ implies $0\leq b-1 < 1+x_3$.
  Finally, if $q_i = 1+(1+x_3)x_2$, we write $b=\alpha(1+x_3)+\beta$, where $0\leq \beta < 1+x_3$ and $0\leq \alpha < 1+x_1(1+(1+x_3)x_2)$. 
  Moreover, we write $\alpha = \gamma x_1 + \delta$, where $0\leq \delta < x_1$ and $0\leq \gamma \leq 1+(1+x_3)x_2$. 
  We consider the following possible cases:
  
  \textbf{Subcase 1 of 4:} Suppose $\beta > 0$ and $\delta > 0$. Then, choosing $c = (\gamma x_1+1)(1+x_3) < (\gamma x_1+\delta)(1+x_3) + \beta = b$, it follows that both sides of~\eqref{eqn 2} are equal to 0.
  
  \textbf{Subcase 2 of 4:} Suppose $\beta = 0$ and $\delta > 0$. Note that $\delta \neq 1$ since $h((\gamma x_1+1)(1+x_3))= 1$. Therefore, $2\leq \delta < x_1$, so we may consider $c=(\gamma x_1+1)(1+x_3) < (\gamma x_1+\delta)(1+x_3) = b$. Since $2\leq \delta < x_1$ and $0\leq \gamma < 1+(1+x_3)x_2$, both sides of~\eqref{eqn 2} are equal to 0.
  
  \textbf{Subcase 3 of 4:} Suppose $\beta>0$ and $\delta = 0$. If $0\leq \gamma \leq x_2$, then $\beta \neq 1$ since $h(\gamma x_1(1+x_3)+1) = 1$. Thus, it must be that $\beta > 1$, thereby allowing us to consider $c=\gamma x_1(1+x_3)+1 < \gamma x_1(1+x_3) + \beta = b$. With this choice of $c$, it is straightforward to verify both sides of~\eqref{eqn 2} will again be equal to 0. Otherwise, if $x_2 < \gamma \leq 1+(1+x_3)x_2$, consider $c=x_2x_1(1+x_3)+1 < \gamma x_1(1+x_3) + \beta = b$. Then, the left-hand side of~\eqref{eqn 2} simplifies to 
  \begin{align*}
  \gamma - x_2 + \floor{\frac{\beta(1+(1+x_3)x_2)-\gamma(1+x_3)}{(1+x_3)(1+x_1(1+(1+x_3)x_2))}},
  \end{align*}
  and the right-hand side of~\eqref{eqn 2} simplifies to 
  \begin{align*}
  \gamma - x_2 + \floor{\frac{\beta(1+(1+x_3)x_2)-\gamma(1+x_3) - 1}{(1+x_3)(1+x_1(1+(1+x_3)x_2))}}.
  \end{align*}
  Indeed, these two quantities are equivalent because $\beta(1+(1+x_3)x_2) \neq \gamma(1+x_3)$. 
  To see this, assume otherwise. 
  Then, it would follow that $\beta = (\gamma - \beta x_2)(1+x_3)$. 
  However, this is a contradiction as we assumed $0<\beta < 1+x_3$. 
  Therefore, we have that~\eqref{eqn 2} is satisfied.
  
  \textbf{Subcase 4 of 4:} Suppose $\beta = \delta = 0$. 
  In this case, $b = \gamma x_1(1+x_3)$.
  Moreover, note that $\gamma > 0$ since otherwise, $b = 0$ contradicting our bounds on $b$. 
  If $x_1=1$, then $h(b) = h(\gamma(1+x_3)) = 1$. 
  Hence, we need only consider when $x_1>1$. 
  Since $\gamma > 0$ and $x_1 > 1$, we may consider $c = 1+x_3 < \gamma x_1(1+x_3) = b$, from which it is straightforward to find that both sides of~\eqref{eqn 2} are equal to $\gamma - 1$. 
  
  As these cases cover all possible values for $\beta$ and $\delta$, this completes our second case.
  
  \textbf{Case:} $q_j = (1+(1+(1+x_3)x_2)x_1)(1+(1+x_3)x_2)$. Thus, we consider $1\leq b \leq (1+(1+(1+x_3)x_2)x_1)(1+(1+x_3)x_2) - 1$.
  Again, we will start by identifying those values of $b$ that satisfy~\eqref{eqn 1} and~\eqref{eqn 3}. 
  It is straightforward to verify that 
  \begin{align*}
  h(b) = b - x_1\floor{\frac{b}{1+(1+(1+x_3)x_2)x_1}} - x_2\floor{\frac{b(1+x_3)}{1+(1+x_3)x_2}}.
  \end{align*}
  Writing $b=\alpha(1+(1+x_3)x_2)+\beta$, where $0\leq \beta < 1+(1+x_3)x_2$ and $0\leq \alpha < 1+(1+(1+x_3)x_2)x_1$, it follows that
  \begin{align*}
  	h(b) &= h(\alpha(1+(1+x_3)x_2)+\beta) \\ 
  	&= \alpha + \beta - x_1\floor{\frac{\alpha(1+(1+x_3)x_2)+\beta}{1+(1+(1+x_3)x_2)x_1}} - x_2\floor{\frac{\beta(1+x_3)}{1+(1+x_3)x_2}}. 
  \end{align*}
  Now, writing $\beta = \gamma x_2 + \delta$ , where $0\leq \delta < x_2$ and $0\leq \gamma \leq 1+x_3$, and $\alpha = \vep x_1+\eta$, where $0\leq \eta < x_1$ and $0\leq \vep \leq 1+(1+x_3)x_2$, it follows that
  \begin{align*}
  h(b) = \eta + \delta - x_1\floor{\frac{\eta(1+(1+x_3)x_2)+\gamma x_2 + \delta - \vep}{1+(1+(1+x_3)x_2)x_1}} - x_2\floor{\frac{\delta(1+x_3)-\gamma}{1+(1+x_3)x_2}}.
  \end{align*}
Since $0\leq \delta < x_2$ and $0\leq \gamma \leq 1+x_3$, observe that
\begin{align*}
\floor{\frac{\delta(1+x_3)-\gamma}{1+(1+x_3)x_2}} = \begin{cases} -1, & \delta = 0, \gamma > 0 \\ \phantom{-}0, & \text{otherwise} \end{cases}.
\end{align*}
Given the bounds on $\vep$, $\eta$, $\gamma$, and $\delta$, note that $-(1+(1+x_3)x_2) \leq \eta(1+(1+x_3)x_2)+\gamma x_2 + \delta - \vep < 1+(1+(1+x_3)x_2)x_1$.
Consequently, it follows that
\begin{align*}
\abs{\eta(1+(1+x_3)x_2)+\gamma x_2 + \delta - \vep} < 1+(1+(1+x_3)x_2)x_1,
\end{align*}
and this implies that $\floor{\frac{\eta(1+(1+x_3)x_2)+\gamma x_2 + \delta - \vep}{1+(1+(1+x_3)x_2)x_1}}$ is equal to either $0$ or $-1$. 
To resolve this floor function, we consider the following subcases which analyze the sign of the numerator of its argument. 
Let $n$ denote that numerator, i.e., $n=\eta(1+(1+x_3)x_2)+\gamma x_2 + \delta - \vep$.

\textbf{Subcase 1 of 3:} Suppose $\eta = 0$ and $\gamma x_2 + \delta \geq \vep$. 
Then, we have that $n \geq 0$, implying
\begin{align*}
\floor{\frac{\eta(1+(1+x_3)x_2)+\gamma x_2 + \delta - \vep}{1+(1+(1+x_3)x_2)x_1}} = 0.
\end{align*}
Therefore, our equation for $h(b)$ simplifies to
\begin{align*}
h(b) = \delta - x_2\cdot \left. \begin{cases} -1, & \delta = 0, \gamma > 0 \\ \phantom{-}0, & \text{otherwise} \end{cases} \right\}.
\end{align*}
If $\delta = 0$ and $\gamma > 0$, then $h(b) = h(\vep x_1(1+(1+x_3)x_2)+\gamma x_2) = x_2$. 
Thus, if $x_2=1$, we have that $h(b) = h(\vep x_1(2+x_3)+\gamma) = 1$ whenever $\gamma\geq \vep$ with $\gamma > 0$. 
Otherwise, $h(b) = \delta$, which forces $\delta = 1$, i.e., $h(b) = h(\vep x_1(1+(1+x_3)x_2) + \gamma x_2 + \delta) = 1$ whenever $\gamma x_2 + \delta \geq \vep$.

\textbf{Subcase 2 of 3:} Suppose $\eta = 0$ and $\gamma x_2 + \delta < \vep$.
Then, we have that $n < 0$, implying
\begin{align*}
\floor{\frac{\eta(1+(1+x_3)x_2)+\gamma x_2 + \delta - \vep}{1+(1+(1+x_3)x_2)x_1}} = -1.
\end{align*}
Therefore, our equation for $h(b)$ simplifies to
\begin{align*}
h(b) = \delta + x_1 - x_2\cdot \left. \begin{cases} -1, & \delta = 0, \gamma > 0 \\ \phantom{-}0, & \text{otherwise} \end{cases} \right\}.
\end{align*}
If $\delta = 0$ and $\gamma > 0$, then $h(b) = h(\vep x_1(1+(1+x_3)x_2)+\gamma x_2) = x_1 + x_2 > 1$. 
Otherwise, $h(b) = \delta + x_1$, which forces $\delta = 0$ and $x_1 = 1$ since $x_1\geq 1$, i.e., $h(b) = h(\vep(1+(1+x_3)x_2) + \gamma x_2) = 1$ whenever $\gamma x_2 < \vep$.

\textbf{Subcase 3 of 3:} Suppose $\eta \geq 1$. 
Then, we have that $n \geq 0$, implying
\begin{align*}
\floor{\frac{\eta(1+(1+x_3)x_2)+\gamma x_2 + \delta - \vep}{1+(1+(1+x_3)x_2)x_1}} = 0.
\end{align*}
Therefore, our equation for $h(b)$ simplifies to
\begin{align*}
h(b) = \eta + \delta - x_2\cdot \left. \begin{cases} -1, & \delta = 0, \gamma > 0 \\ \phantom{-}0, & \text{otherwise} \end{cases} \right\}.
\end{align*}
If $\delta = 0$ and $\gamma > 0$, then $h(b) = h((\vep x_1+\eta)(1+(1+x_3)x_2)+\gamma x_2) = \eta + x_2 \geq 1+x_2 > 1$. 
Otherwise, $h(b) = \eta + \delta$, which forces $\delta = 0$ and $\eta = 1$ since we assume $\eta \geq 1$, i.e., $h(b) = h((\vep x_1 + 1)(1+(1+x_3)x_2) + \gamma x_2) = 1$.

We summarize the values of $b$ for which $h(b) = 1$ that were just derived:
\begin{itemize}
\item If $x_1=1$, $\eta=\delta=0$,  and $0\leq \gamma x_2 < \vep \leq 1+(1+x_3)x_2$, we have $b=\vep(1+(1+x_3)x_2)+\gamma x_2$.
\item If $x_2=1$, $\eta = \delta = 0$, and $0\leq \vep \leq \gamma \leq 1+x_3$ with $\gamma > 0$, we have $b=\vep x_1(2+x_3)+\gamma$. 
\item If $x_1,x_2\geq 1$, $\eta = 0$, $\delta = 1$, and $0\leq \vep \leq \gamma x_2 + 1 \leq 1+(1+x_3)x_2$, we have $b=\vep x_1(1+(1+x_3)x_2)+\gamma x_2 + 1$.
\item If $x_1,x_2\geq 1$, $\eta = 1$, $\delta = 0$, $0\leq \gamma \leq 1+x_3$, and $0\leq \vep \leq 1+(1+x_3)x_2$, we have $b=(\vep x_1+1)(1+(1+x_3)x_2)+\gamma x_2$.
\end{itemize}

Our next goal is to establish that~\eqref{eqn 2} is always satisfied; recall that we are in the case where $q_j=(1+(1+(1+x+3)x_2)x_1)(1+(1+x_3)x_2)$. 
If $q_i=(1+(1+(1+x+3)x_2)x_1)(1+(1+x_3)x_2)$, the result is trivial. 
If $q_i = 1+(1+x_3)x_2$, we write $b=\alpha(1+(1+x_3)x_2)+\beta$, where $0\leq \beta < 1+(1+x_3)x_2$ and $0\leq \alpha < 1+(1+(1+x_3)x_2)x_1$.
Furthermore, we write $\beta = \gamma x_2+\delta$, where $0\leq \delta<x_2$ and $0\leq\gamma\leq 1+x_3$, and we write $\alpha = \vep x_1 + \eta$, where $0\leq \eta < x_1$ and $0\leq \vep \leq 1+(1+x_3)x_2$.
We consider the following possible cases:

\textbf{Subcase 1 of 4:} Suppose $\eta > 0$ and $\delta > 0$. 
We consider $c=(\vep x_1+1)(1+(1+x_3)x_2)+\gamma x_2 < (\vep x_1+\eta)(1+(1+x_3)x_2)+\gamma x_2 +\delta = b$.
Since $0<\eta<x_1$ and $0\leq \delta < x_2$, it follows that $0<(\eta-1)(1+(1+x_3)x_2)+\delta < x_1(1+(1+x_3)x_2)$.
Therefore, it is straightforward to verify that both sides of~\eqref{eqn 2} are equal to $0$. 

\textbf{Subcase 2 of 4:} Suppose $\eta = 0$ and $\delta > 0$. 
If $\gamma x_2+1 \geq \vep$, note that $\delta\neq 1$ since otherwise, $h(b)=h(\vep x_1(1+(1+x_3)x_2)+\gamma x_2 + 1)=1$. 
Thus, we have that $\delta>1$, and we consider $c=\vep x_1(1+(1+x_3)x_2)+\gamma x_2+1 < \vep x_1(1+(1+x_3)x_2)+\gamma x_2+\delta = b$.
Given that $1<\delta <x_2$ implies $\gamma x_2+\delta > \gamma x_2+1 \geq \vep$, it is straightforward to verify that our choice of $c$ gives that both sides of~\eqref{eqn 2}  are equal to $0$.
Otherwise, if $\gamma x_2+1<\vep$ (and hence, $\vep > 1$), we consider $c=1+x_1(1+(1+x_3)x_2) \leq \delta + x_1(1+(1+x_3)x_2) < \vep x_1(1+(1+x_3)x_2)+\gamma x_2 +\delta = b$.
With this choice of $c$, it is straightforward to verify both sides of~\eqref{eqn 2} are equal to
\begin{align*}
\vep - 1 + \floor{\frac{\gamma x_2 + \delta - \vep}{1+(1+(1+x_3)x_2)x_1}}\, .
\end{align*}

\textbf{Subcase 3 of 4:} Suppose $\eta > 0$ and $\delta = 0$. 
Note that $\eta \neq 1$ since otherwise, $h(b)=h((\vep x_1+1)(1+(1+x_3)x_2)+\gamma x_2)=1$. 
Thus, we have that $\eta > 1$, and we consider $c=(\vep x_1+1)(1+(1+x_3)x_2)+\gamma x_2 < (\vep x_1+\eta)(1+(1+x_3)x_2)+\gamma x_2 = b$. 
Since $1<\eta<x_1$, it is straightforward to verify that both sides of~\eqref{eqn 2} are equal to $0$. 

\textbf{Subcase 4 of 4:} Suppose $\eta = \delta = 0$. 
Further suppose $\gamma x_2 < \vep$ (and hence, $\vep > 0$).
If $x_1=1$, then $h(b) = h(\vep(1+(1+x_3)x_2)+\gamma x_2) = 1$, so we may assume $x_1>1$.
Consider $c=1+(1+x_3)x_2 < \vep x_1(1+(1+x_3)x_2)+\gamma x_2 = b$. 
Since $x_1 > 1$ and $\gamma x_2 < \vep$, it follows that $-x_1(1+(1+x_3)x_2) \leq \gamma x_2 - \vep - (1+(1+x_3)x_2) < 0$.
Given this, it is straightforward to verify that both sides of~\eqref{eqn 2} are equal to $\vep - 1$.
Now, suppose $\gamma x_2 > \vep$. 
Note that $\gamma \neq 0$ since otherwise, $\vep < 0$ contradicting our initial bounds on $\vep$. 
Thus, we have that $\gamma > 0$.
Moreover, if $x_2=1$, it follows that $h(b) = h(\vep x_1(2+x_3)+\gamma)=1$, so we may assume $x_2 > 1$.
Taking $c=1$, the inequality $\gamma x_2 > \vep$ readily implies that both sides of~\eqref{eqn 2} are equal to $\vep$. 
Finally, suppose $\gamma x_2 = \vep$. 
Note that neither $\gamma$ nor $\vep$ can be equal to 0 since otherwise, we would have $\eta = \delta = \gamma = \vep = 0$, implying $b=0$. 
This, of course, contradicts the bounds on $b$.
Moreover, we may again assume $x_2>1$ (and thus, $\vep > 1$) since otherwise, $h(b) = h(\vep x_1(2+x_3)+\gamma) = 1$. 
Since $\gamma x_2 = \vep$, observe that $b = \vep(1+x_1(1+(1+x_3)x_2))$.
As such, we may consider $c = 1+x_1(1+(1+x_3)x_2)$ which is strictly less than $b$ since $\vep > 1$. 
This choice of $c$ readily gives that both sides of~\eqref{eqn 2} are equal to $\vep - 1$.

Finally, if $q_i = (1+x_3)(1+x_1(1+(1+x_3)x_2))$, we begin again by writing $b=\alpha(1+(1+x_3)x_2)+\beta$, where $0\leq \beta < 1+(1+x_3)x_2$ and $0\leq \alpha < 1+(1+(1+x_3)x_2)x_1$.
Furthermore, we write $\beta = \gamma x_2+\delta$, where $0\leq \delta<x_2$ and $0\leq\gamma\leq 1+x_3$, and we write $\alpha = \vep x_1 + \eta$, where $0\leq \eta < x_1$ and $0\leq \vep \leq 1+(1+x_3)x_2$.
We consider the following possible cases:

\textbf{Subcase 1 of 3:} Suppose $\delta > 1$. 
Combined with our bounds on $\gamma$, this implies that $0\leq (\delta-1)(1+x_3)-\gamma <1+(1+x_3)x_2$.
Taking $c=1$, it follows that both sides of~\eqref{eqn 2} are equal to $\alpha(1+x_3)+\gamma$.

\textbf{Subcase 2 of 3:} Suppose $\delta = 1$. 
If $\gamma x_2+1\geq \vep$, note that $\eta \neq 0$ since otherwise, $h(b) = h(\vep x_1(1+(1+x_3)x_2)+\gamma x_2+1) = 1$.
Thus, $\eta \geq 1$, and we may consider $c=(\vep x_1+1)(1+(1+x_3)x_2) < (\vep x_1+\eta)(1+(1+x_3)x_2)+\gamma x_2+1 = b$. 
Since $\gamma\leq 1+x_3$, it is straightforward to show both sides of~\eqref{eqn 2} are equal to $(\eta-1)(1+x_3)+\gamma$.
On the other hand, if $\gamma x_2 +1 < \vep$ (and hence, $\vep > 1$), we may consider $c=1+(1+x_3)x_2< (\vep x_1+\eta)(1+(1+x_3)x_2)+\gamma x_2+1 = b$, from which it follows that both sides of~\eqref{eqn 2} are equal to $(\vep x_1 + \eta - 1)(1+x_3) + \gamma$.

\textbf{Subcase 3 of 3:} Suppose $\delta = 0$. Note that $\eta \neq 1$ since otherwise, $h(b) = 1$. 
This lends itself to two possibilities: (A) $\eta = 0$ or (B) $\eta > 1$. 
If (A) holds, we first suppose $\gamma x_2 < \vep$ (and hence, $\vep > 0$). 
If $x_1=1$, then $h(b)=1$, so we may assume $x_1>1$.
Since $\vep > 0$ and $x_1>1$, we may consider $c=1+(1+x_3)x_2 < \vep x_1(1+(1+x_3)x_2)+\gamma x_2 = b$. 
For this choice of $c$, it readily follows that both sides of~\eqref{eqn 2} are equal to
\begin{align*}
(\vep x_1 -1)(1+x_3) + \floor{\frac{\gamma x_2}{1+(1+x_3)x_2}}.
\end{align*}
Now, suppose $\gamma x_2 > \vep$. 
Since $\vep\geq 0$ by construction, this inequality implies $\gamma \neq 0$.
Moreover, note that if $x_2=1$, then $h(b) = 1$.
Thus, we may assume $x_2>1$, and we simply consider $c=1$.
Given that $0<\gamma \leq 1+x_3$ and $x_2>1$, it follows that $0<\gamma + (1+x_3) < 1+(1+x_3)x_2$.
Therefore, with $c=1$, we find that both sides of~\eqref{eqn 2} are equal to $\vep x_1(1+x_3) + \gamma - 1$.
Finally, suppose $\gamma x_2 = \vep$. 
Since $b\neq 0$, this equality implies that both $\gamma$ and $\vep$ cannot be $0$. Moreover, we may again assume $x_2>1$ since $x_2=1$ would imply $h(b) = 1$. 
Since $\gamma x_2 = \vep$, it follows that $b=\vep(1+x_1(1+(1+x_3)x_2))$, and we also get that $\vep > 1$ since we assume $x_2>1$.
Combining, we may consider $c = 1+x_1(1+(1+x_3)x_2) < \vep(1+x_1(1+(1+x_3)x_2)) = b$. 
As before, the inequality $0<\gamma + (1+x_3) < 1+(1+x_3)x_2$ still holds in this case, from which it is straightforward to verify both sides of~\eqref{eqn 2} are equal to $(\vep-1)x_1(1+x_3) + \gamma -1$.

On the other hand, if (B) holds, we have that $\eta > 1$. 
Therefore, we consider $c=(\vep x_1+1)(1+(1+x_3)x_2) + \gamma x_2 < (\vep x_1+\eta)(1+(1+x_3)x_2) + \gamma x_2 = b$, from which it follows that both sides of~\eqref{eqn 2} simplify to $(\eta-1)(1+x_3)$. 
In any case, we find that both sides of~\eqref{eqn 2} are equivalent for each possible $q_i$, thereby completing our third and final case.
Thus, we have established IDP for $\br$-vectors of type $(v)$.

\subsection{Proof of IDP for type $(vi)$ in Theorem~\ref{thm:3suppclassify}}
We next verify IDP for $\br$-vectors of type $(vi)$ using Theorem~\ref{thm:idpreflexives}. 
  Again, we must consider three cases corresponding to three possible values of $q_j$.
  
  \textbf{Case: $q_j=(1+x_3)(1+(1+x_3)x_2)$}. 
  We first identify those values of $b$ that satisfy~\eqref{eqn 1} and~\eqref{eqn 3}. It is straightforward to verify that 
 \begin{align*}
 h(b) = b - x_2\floor{\frac{b}{1+(1+x_3)x_2}} - x_3\floor{\frac{b}{1+x_3}},
 \end{align*}
 where $1\leq b \leq (1+x_3)(1+(1+x_3)x_2)-1$. 
 Writing $b = \alpha (1+x_3) + \beta$, where $0\leq \beta < 1+x_3$ and $0\leq \alpha < 1+(1+x_3)x_2$ for $\alpha,\beta \in \Z$, it follows that
\begin{align*}
h(b) &= h(\alpha(1+x_3)+\beta) \\ 
&= \alpha + \beta - x_2\floor{\frac{\alpha(1+x_3)+\beta}{1+(1+x_3)x_2}}.
\end{align*}
We can now further divide into cases: either we have $\alpha = (1+x_3)x_2$ or we have $\alpha = \gamma x_2 + \delta$ where $0\leq \gamma < 1+x_3$ and $0\leq \delta < x_2$, which yields
\begin{align*}
h(b) &= h((\gamma x_2+\delta)(1+x_3)+\beta) \\
&= \delta + \beta - x_2\floor{\frac{\delta(1+x_3)+\beta - \gamma}{1+(1+x_3)x_2}}.
\end{align*}
For $\alpha \neq (1+x_3)x_2$, let $n=\delta(1+x_3)+\beta - \gamma$. 
Observe that since $0\leq \delta(1+x_3)+\beta < 1+(1+x_3)x_2$ and $0\leq \gamma < 1+x_3$, it follows that $\abs{n} < 1+(1+x_3)x_2$. 
Thus, $\floor{\frac{\delta(1+x_3)+\beta - \gamma}{1+(1+x_3)x_2}}$ is equal to either $0$ or $-1$. 

\textbf{Subcase 1 of 4:} Suppose $\alpha = (1+x_3)x_2$. Since $0\leq \beta < 1+x_3$, we have
\begin{align*}
h((1+x_3)^2x_2 + \beta) &= (1+x_3)x_2 + \beta - (1+x_3)x_2 - x_2\floor{\frac{\beta-(1+x_3)}{1+(1+x_3)x_2}} \\
&= \beta + x_2.
\end{align*}
If this is equal to $1$, then it must be that $\beta = 0$ and $x_2=1$. Thus, if $x_2=1$, we have that $h((1+x_3)^2) = 1$.\

For the next three subcases, we assume $\alpha \neq (1+x_3)x_2$, so we may write $\alpha = \gamma x_2 + \delta$ where $0\leq \delta < x_2$ and $0\leq \gamma < 1+x_3$.

\textbf{Subcase 2 of 4:} Suppose $\delta = 0$ and $\beta \geq \gamma$. Then, we have that $n\geq 0$, implying 
\begin{align*}
\floor{\frac{\delta(1+x_3)+\beta - \gamma}{1+(1+x_3)x_2}} = 0.
\end{align*}
Therefore, our equation for $h(b)$ simplifies to $h(b) = \beta$, which forces $\beta = 1$, i.e., $h(b) = h(\gamma x_2(1+x_3)+1) = 1$ whenever $0\leq \gamma \leq 1$.

\textbf{Subcase 3 of 4:} Suppose $\delta = 0$ and $\beta < \gamma$. Then, we have that $n<0$, implying
\begin{align*}
\floor{\frac{\delta(1+x_3)+\beta - \gamma}{1+(1+x_3)x_2}} = -1.
\end{align*}
Therefore, our equation for $h(b)$ simplifies to $h(b) = \beta + x_2$, which forces $\beta = 0$ and $x_2=1$ since $x_2\geq 1$, i.e., $h(b) = h(\gamma(1+x_3)) = 1$ for $0<\gamma<1+x_3$.

\textbf{Subcase 4 of 4:} Suppose $\delta \geq 1$. Then, since $0\leq \gamma < 1+x_3$, we have that $n\geq 0$, implying
\begin{align*}
\floor{\frac{\delta(1+x_3)+\beta - \gamma}{1+(1+x_3)x_2}} = 0.
\end{align*}
Therefore, our equation for $h(b)$ simplifies to $h(b) = \delta + \beta$, which forces  $\beta = 0$ and $\delta = 1$ since we assume $\delta \geq 1$, i.e., $h(b) = h((\gamma x_2 +1)(1+x_3)) = 1$ for $0\leq \gamma < 1+x_3$. 

We summarize the values of $b$ for which $h(b) = 1$ that were just derived:

\begin{itemize}
\item If $x_2=1$ and $0< \gamma \leq 1+x_3$, we have $b=\gamma(1+x_3)$.
\item If $x_2\geq 1$, $\delta = 0$, $\beta = 1$, and $0\leq \gamma \leq 1$, we have $b=\gamma x_2(1+x_3) + 1$.
\item If $x_2\geq 1$, $\delta = 1$, $\beta = 0$, and $0\leq \gamma < 1+x_3$, we have $b=(\gamma x_2+1)(1+x_3)$.
\end{itemize}

Our next goal is to establish that~\eqref{eqn 2} is always satisfied; recall that we are in the case where $q_j = (1+x_3)(1+(1+x_3)x_2)$.
If $q_i = (1+x_3)(1+(1+x_3)x_2)$, the result is trivial. If $q_i = (1+(1+x_3)(1+(1+x_3)x_2)x_1)(1+(1+x_3)x_2)$, we write $b=\alpha(1+x_3)+\beta$, where $0\leq \beta < 1+x_3$ and $0\leq \alpha < 1+(1+x_3)x_2$. 
Note that $b\neq 1+x_3$ since $h(1+x_3)=1$. 
If $b>1+x_3$, we set $c=1+x_3$ from which it is straightforward to compute that both sides of~\eqref{eqn 2} are equal to $((\alpha-1)(1+x_3)+\beta)(1+(1+x_3)x_2)x_1 + \alpha - 1$.
  Otherwise, if $b<1+x_3$, note that $\alpha = 0$ and hence $b=\beta$. 
  To satisfy~\eqref{eqn 1}, we need only consider $1<\beta<1+x_3$. 
  Thus, we may choose $c=1$ from which it follows that both sides of~\eqref{eqn 2} are equal to $(\beta-1)(1+(1+x_3)x_2)x_1$ since $1< b < 1+x_3$ implies $1\leq \beta-1 < 1+x_3$.
Finally, if $q_i = (1+x_3)(1+x_1(1+x_3)(1+(1+x_3)x_2))$, we again write $b=\alpha(1+x_3)+\beta$, where $0\leq \beta < 1+x_3$ and $0\leq \alpha < 1+(1+x_3)x_2$. 
Moreover, in the case that $\alpha \neq (1+x_3)x_2$, we write $\alpha = \gamma x_2 + \delta$, where $0\leq \delta < x_2$ and $0\leq \gamma < (1+x_3)$.
We consider the following possible cases:

\textbf{Subcase 1 of 5:} Suppose $\alpha = (1+x_3)x_2$. If $x_2=1$, then $\beta \neq 0$ since $h((1+x_3)^2) = 1$. 
Thus, we may consider $c=(1+x_3)^2 < (1+x_3)^2 + \beta = b$. 
With this choice of $c$, it is straightforward to verify that both sides of~\eqref{eqn 2} are equal to $\beta x_1(1+x_3)$. 
Otherwise, if $x_2 > 1$, we consider $c=1+x_3 < (1+x_3)^2x_2+\beta = b$. Since $x_2>1$ implies $-x_2(1+x_3) < \beta - 2(1+x_3) < 0$, it is straightforward to verify that both sides of~\eqref{eqn 2} are equal to $(b-(1+x_3))x_1(1+x_3) + x_3$.

\textbf{Subcase 2 of 5:} Suppose $\alpha \neq (1+x_3)x_2$ with $\delta > 0$ and $\beta > 0$. Then, choosing $c = (\gamma x_2 + 1)(1+x_3) < (\gamma x_2+\delta)(1+x_3)+\beta = b$, it follows that both sides of~\eqref{eqn 2} are equal to $((\delta - 1)(1+x_3) + \beta)x_1(1+x_3)$.

\textbf{Subcase 3 of 5:} Suppose $\alpha \neq (1+x_3)x_2$ with $\delta > 0$ and $\beta = 0$. Note that $\delta \neq 1$ since $h((\gamma x_2+1)(1+x_3))=1$. 
Therefore, $2\leq \delta < x_2$, so we may consider $c=(\gamma x_2+1)(1+x_3) < (\gamma x_2 + \delta)(1+x_3)=b$. 
Given this choice, we find that both sides of~\eqref{eqn 2} are equal to $(\delta - 1)(1+x_3)x_1(1+x_3)$.

\textbf{Subcase 4 of 5:} Suppose $\alpha \neq (1+x_3)x_2$ with $\delta = 0$ and $\beta > 0$. If $0\leq \gamma \leq 1$, then $\beta\neq 1$ since $h(\gamma x_2(1+x_3)+1) = 1$.
Thus, it must be that $\beta > 1$, thereby allowing us to consider $c = \gamma x_2(1+x_3)+1 < \gamma x_2(1+x_3)+\beta = b$.
With this choice of $c$, it is straightforward to verify both sides of~\eqref{eqn 2} are equal to $(\beta-1)x_1(1+x_3)$.
Otherwise, if $1<\gamma < 1+x_3$, consider $c=1+(1+x_3)x_2 <\gamma x_2(1+x_3) + \beta = b$. 
Then, both sides of~\eqref{eqn 2} are equal to
\begin{align*}
((\gamma-1)x_2(1+x_3)+\beta - 1)x_1(1+x_3) + \gamma - 1 + \floor{\frac{\beta - \gamma}{1+(1+x_3)x_2}}.
\end{align*}

\textbf{Subcase 5 of 5:} Suppose $\alpha \neq (1+x_3)x_2$ with $\delta = \beta = 0$. In this case, $b=\gamma x_2(1+x_3)$.
Moreover, note that $\gamma > 0$ since otherwise, $b=0$ contradicting our bounds on $b$. If $x_2=1$, then $h(b) = h(\gamma(1+x_3)) = 1$.
Hence, we need only consider when $x_2>1$.
Since $\gamma > 0$ and $x_2 > 1$, we may take $c=1+x_3 < \gamma x_2(1+x_3) = b$, from which it is straightforward to find that both sides of~\eqref{eqn 2} are equal to $(\gamma x_2 - 1)(1+x_3)^2 x_1 + \gamma - 1$.

This completes our first case.

\textbf{Case:} $q_j = (1+x_3)(1+x_1(1+x_3)(1+(1+x_3)x_2))$. We first identify those values of $b$ that satisfy~\eqref{eqn 1} and~\eqref{eqn 3}.
  It is straightforward to verify that
  \[
    h(b)=b-x_1\left\lfloor \frac{b(1+(1+x_3)x_2)}{1+x_1(1+x_3)(1+(1+x_3)x_2)}\right\rfloor -x_3\left\lfloor \frac{b}{1+x_3} \right\rfloor \, .
  \]
  Writing $b=\alpha (1+x_3)+\beta$, where $0\leq \beta < 1+x_3$ and $0\leq \alpha \leq 1+ x_1(1+x_3)(1+(1+x_3)x_2)$, it follows that
  \begin{align*}
    h(b) &= h(\alpha(1+x_3)+\beta) \\
    &= \alpha + \beta - x_1\left\lfloor \frac{(\alpha(1+x_3)+\beta)(1+(1+x_3)x_2)}{1+x_1(1+x_3)(1+(1+x_3)x_2)}\right \rfloor  \, .
  \end{align*}
  We can now further divide into cases: either we have $\alpha = x_1(1+x_3)(1+(1+x_3)x_2)$ or we have $\alpha = \gamma x_1 + \delta$ where $0\leq \gamma < (1+x_3)(1+(1+x_3)x_2)$ and $0\leq \delta < x_1$. If $\alpha = x_1(1+x_3)(1+(1+x_3)x_2)$, then since $0\leq \beta < 1+x_3$, we have
\begin{align*}
 h(b) &= \beta - x_1\floor{\frac{\beta(1+(1+x_3)x_2)-(1+x_3)(1+(1+x_3)x_2)}{1+x_1(1+x_3)(1+(1+x_3)x_2)}} \\
 &= \beta + x_1.
\end{align*}  
If this is equal to 1, then it must be that $\beta=0$ and $x_1=1$. Thus, if $x_1=1$, we have that $h((1+x_3)^2(1+(1+x_3)x_2))=1$. Otherwise, if $\alpha \neq x_1(1+x_3)(1+(1+x_3)x_2)$, we write $\alpha = \gamma x_1 + \delta$ where $0\leq \gamma < (1+x_3)(1+(1+x_3)x_2)$ and $0\leq \delta < x_1$. Thus, $h(b)$ simplifies as follows
  \begin{align*}
  h(b) &= h((\gamma x_1 + \delta)(1+x_3) + \beta) \\ 
  &= \delta + \beta - x_1 \floor{\frac{(\delta(1+x_3)+\beta)(1+(1+x_3)x_2) - \gamma}{1+x_1(1+x_3)(1+(1+x_3)x_2}} \, .
  \end{align*}
  For $\alpha \neq x_1(1+(1+x_3)x_2)$, observe that since $0\leq \delta(1+x_3)(1+(1+x_3)x_2) \leq (x_1-1)(1+x_3)(1+(1+x_3)x_2)$ and $0\leq \beta(1+(1+x_3)x_2) < (1+x_3)(1+(1+x_3)x_2)$ with $0\leq \gamma<(1+x_3)(1+(1+x_3)x_2)$, we have that $0\leq (\delta(1+x_3)+\beta)(1+(1+x_3)x_2) < x_1(1+x_3)(1+(1+x_3)x_2)$. 
Let $n=(\delta(1+x_3)+\beta)(1+(1+x_3)x_2)-\gamma$. 
The inequalities above readily imply $\abs{n} < 1+x_1(1+x_3)(1+(1+x_3)x_2)$.
  Thus, 
  \[
  \floor{\frac{\delta(1+x_3)(1+(1+x_3)x_2) + \beta(1+(1+x_3)x_2 - \gamma}{1+x_1(1+x_3)(1+(1+x_3)x_2)}}
  \]
  is equal to either $0$ or $-1$.
  We further write $\gamma = \vep (1+(1+x_3)x_2)+\eta$, where $0\leq \eta < 1+(1+x_3)x_2$ and $0\leq \vep <1+x_3$. 
  Then, $h(b)$ becomes
  \begin{align*}
  h(b) = \delta + \beta -x_1\floor{\frac{(\delta(1+x_3)+\beta)(1+(1+x_3)x_2) - (\vep(1+(1+x_3)x_2)+\eta)}{1+x_1(1+x_3)(1+(1+x_3)x_2)}}\, .
  \end{align*}
  
  \textbf{Subcase 1 of 4:} Suppose $\delta = 0$ and $\beta > \vep$. 
  Then, we have that $n>0$, implying
  \begin{align*}
  \floor{\frac{\delta(1+x_3)(1+(1+x_3)x_2) + \beta(1+(1+x_3)x_2 - \gamma}{1+x_1(1+x_3)(1+(1+x_3)x_2)}} = 0.
  \end{align*}
  Therefore, our equation for $h(b)$ simplifies to $h(b) = \beta$, which forces $\beta = 1$.
  Note that $\beta = 1$ implies $\vep =0$ as $\beta > \vep$. 
  Thus, we have $h(b) = h(\eta x_1(1+x_3)+1) = 1$ whenever $0\leq \eta \leq (1+x_3)x_2$.
  
  \textbf{Subcase 2 of 4:} Suppose $\delta = 0$ and $\beta < \vep$. 
  Note that $\vep > 0$ since $\beta < \vep$.
  Then, we have that $n<0$, implying
  \begin{align*}
  \floor{\frac{\delta(1+x_3)(1+(1+x_3)x_2) + \beta(1+(1+x_3)x_2 - \gamma}{1+x_1(1+x_3)(1+(1+x_3)x_2)}} = -1.
  \end{align*}
  Consequently, our equation for $h(b)$ simplifies to $h(b) = \beta + x_1$, which forces $\beta = 0$ and $x_1=1$ since $x_1\geq 1$, i.e., $h(b) = h((\vep(1+(1+x_3)x_2)+\eta)(1+x_3)) = 1$ whenever $0<\vep<1+x_3$ and $0\leq \eta \leq (1+x_3)x_2$.
  
  \textbf{Subcase 3 of 4:} Suppose $\delta=0$ and $\beta = \vep$. If $\eta > 0$, then we have that $n < 0$, implying
    \begin{align*}
  \floor{\frac{\delta(1+x_3)(1+(1+x_3)x_2) + \beta(1+(1+x_3)x_2 - \gamma}{1+x_1(1+x_3)(1+(1+x_3)x_2)}} = -1.
  \end{align*}
  In this case, $h(b)$ simplifies to $h(b) = \beta + x_1$, which again forces $\beta = 0$ and $x_1=1$. 
  Since $\beta = \vep$, it follows that $\vep = 0$, and so we have that $h(b) = h(\eta(1+x_3)=1$ whenever $0<\eta \leq (1+x_3)x_2$.
  Otherwise, if $\eta = 0$, then we have that $n=0$, implying
    \begin{align*}
  \floor{\frac{\delta(1+x_3)(1+(1+x_3)x_2) + \beta(1+(1+x_3)x_2 - \gamma}{1+x_1(1+x_3)(1+(1+x_3)x_2)}} = 0.
  \end{align*}
  Therefore, our equation for $h(b)$ simplifies to $h(b) = \beta$, which forces $\beta = 1$ (and thus, $\vep=1$), i.e., $h(b) = h((1+(1+x_3)x_2)x_1(1+x_3)+1) = 1$. 
  
  \textbf{Subcase 4 of 4:} Suppose $\delta \geq 1$. 
  Then, we have that $n\geq 0$, implying
    \begin{align*}
  \floor{\frac{\delta(1+x_3)(1+(1+x_3)x_2) + \beta(1+(1+x_3)x_2 - \gamma}{1+x_1(1+x_3)(1+(1+x_3)x_2)}} = 0.
  \end{align*}
  Therefore, our equation for $h(b)$ simplifies to $h(b) = \delta + \beta$, which forces $\delta = 1$ and $\beta = 0$ since $\delta \geq 1$. That is, we have $h(b) = h(((\vep(1+(1+x_3)x_2)+\eta)x_1+1)(1+x_3)) = 1$ whenever $0\leq \vep < 1+x_3$ and $0\leq \eta \leq (1+x_3)x_2$. 
  
  We summarize the values of $b$ for which $h(b)=1$ that were just derived:
  \begin{itemize}
  \item If $x_1=1$, $\beta = 0$, and $\alpha = (1+x_3)(1+(1+x_3)x_2)$, we have $b=(1+x_3)^2(1+(1+x_3)x_2$. 
  \item If $x_1=1$, $\beta = \delta = 0$, $0<\vep<1+x_3$, and $0\leq \eta\leq (1+x_3)x_2$, we have $b=(\eta(1+(1+x_3)x_2)+\eta)(1+x_3)$.
  \item If $x_1=1$, $\beta = \delta = \vep = 0$, and $0<\eta \leq (1+x_3)x_2$, we have $b=\eta(1+x_3)$.
  \item If $x_1\geq 1$, $\delta = \vep = 0$, $\beta = 1$, and $0\leq\eta\leq (1+x_3)x_2$, we have $b=\eta x_1(1+x_3)+1$.
  \item If $x_1\geq 1$, $\delta = \eta = 0$, and $\beta = \vep = 1$, we have $b=(1+(1+x_3)x_2)x_1(1+x_3)+1$.
  \item If $x_1\geq 1$, $\delta = 1$, $\beta = 0$, $0\leq \vep < 1+x_3$, and $0\leq \eta \leq (1+x_3)x_2$, we have $b = ((\vep(1+(1+x_3)x_2)+\eta)x_1+1)(1+x_3)$.
  \end{itemize}
  
  Our next goal is to establish that~\eqref{eqn 2} is always satisfied; recall that we are in the case where $q_j = (1+x_3)(1+x_1(1+x_3)(1+(1+x_3)x_2))$.
If $q_i = (1+x_3)(1+x_1(1+x_3)(1+(1+x_3)x_2))$, the result is trivial. If $q_i = (1+(1+x_3)(1+(1+x_3)x_2)x_1)(1+(1+x_3)x_2)$, note that $b\neq 1+x_3$ since $h(1+x_3)=1$. 
If $b>1+x_3$, we set $c=1+x_3$ from which it is straightforward to compute that both sides of~\eqref{eqn 2} are equal to $bx_2 - (1+(1+x_3)x_2) + \floor{\frac{b}{1+x_3}}$.
  Otherwise, if $b<1+x_3$, we may choose $c=1$ from which it follows that both sides of~\eqref{eqn 2} are equal to $(b-1)x_2$ since $1< b < 1+x_3$ implies $1\leq b-1 < 1+x_3$.
Finally, if $q_i = (1+x_3)(1+(1+x_3)x_2)$, we again write $b=\alpha(1+x_3)+\beta$, where $0\leq \beta < 1+x_3$ and $0\leq \alpha < 1+x_1(1+x_3)(1+(1+x_3)x_2)$. 
Moreover, in the case that $\alpha \neq x_1(1+x_3)(1+(1+x_3)x_2)$,  we write $\alpha = \gamma x_1 + \delta$ with $\gamma = (\vep(1+(1+x_3)x_2)+\eta)$, where $0\leq \delta < x_1$, $0\leq \vep < (1+x_3)$, and $0\leq \eta \leq (1+x_3)x_2$.
We consider the following possible cases:

\textbf{Subcase 1 of 5:} Suppose $\alpha = x_1(1+x_3)(1+(1+x_3)x_2)$. 
If $x_1=1$, note that $\beta \neq 0$ since $h((1+x_3)^2(1+(1+x_3)x_2)) = 1$. 
Thus, we may consider $c=(1+x_3)^2(1+(1+x_3)x_2) < (1+x_3)^2(1+(1+x_3)x_2)+\beta = b$.
With this choice of $c$, it is straightforward to verify that both sides of~\eqref{eqn 2} are equal to $0$. 
Otherwise, if $x_1 > 1$, we consider $c=1$. 
Since $x_1>1$ implies $-x_1(1+x_3)(1+(1+x_3)x_2) < (\beta - x_3 - 2)(1+(1+x_3)x_2) < 0$, it is straightforward to verify that both sides of~\eqref{eqn 2} are equal to $(1+x_3)(1+(1+x_3)x_2)-1$.

\textbf{Subcase 2 of 5:} Suppose $\alpha \neq x_1(1+x_3)(1+(1+x_3)x_2)$ with $\delta > 0$ and $\beta > 0$. Then, choosing $c = (\gamma x_1+ 1)(1+x_3) < (\gamma x_1+\delta)(1+x_3)+\beta = b$, it follows that both sides of~\eqref{eqn 2} are equal to $0$ since $0\leq \delta - 1 < x_1 - 1$ and $0<\beta < 1+x_3$ together imply $0 < (\delta-1)(1+x_3)+\beta < x_1(1+x_3)$. 

\textbf{Subcase 3 of 5:} Suppose $\alpha \neq x_1(1+x_3)(1+(1+x_3)x_2)$ with $\delta > 0$ and $\beta = 0$. Note that $\delta \neq 1$ since $h((\gamma x_1+1)(1+x_3))=1$. 
Therefore, $2\leq \delta < x_2$, so we may consider $c=(\gamma x_1+1)(1+x_3) < (\gamma x_1 + \delta)(1+x_3)=b$. 
Given this choice, we find that both sides of~\eqref{eqn 2} are equal to $0$ since $0<\delta < x_1$.

\textbf{Subcase 4 of 5:} Suppose $\alpha \neq x_1(1+x_3)(1+(1+x_3)x_2)$ with $\delta = 0$ and $\beta > 0$. 
If $\beta > \vep$, then $\beta \neq 1$ since otherwise, $\beta = 1$ would force $\vep = 0$ and $h(\eta x_1(1+x_3)+1) = 1$.
Thus, it must be that $\beta > 1$, thereby allowing us to consider $c = \eta x_1(1+x_3) < (\vep(1+(1+x_3)x_2)+\eta)x_1(1+x_3) + \beta = b$.
With this choice of $c$ and since $\beta - 1\geq \vep$ with $\beta > 1$, it is straightforward to verify both sides of~\eqref{eqn 2} are equal to $\vep(1+(1+x_3)x_2)$.
Now, if $\beta< \vep$, we have that $\vep > 1$ since we assumed $\beta > 0$. 
Therefore, we consider $c=(1+(1+x_3)x_2)x_1(1+x_3)+1 < (\vep(1+(1+x_3)x_2)+\eta)x_1(1+x_3) + \beta = b$. 
Then, both sides of~\eqref{eqn 2} are equal to $(\vep - 1)(1+(1+x_3)x_2)+\eta - 1$.
Lastly, if $\beta = \vep$, we consider two possibilities. If $\eta > 0 $, we may again choose $c=(1+(1+x_3)x_2)x_1(1+x_3)+1 < (\vep(1+(1+x_3)x_2)+\eta)x_1(1+x_3) + \beta = b$ and find that both sides of~\eqref{eqn 2} are equal to $(\vep - 1)(1+(1+x_3)x_2)+\eta - 1$.
Otherwise, if $\eta = 0$, note that $\vep = \beta \neq 1$ since otherwise, $h(b) = h((1+(1+x_3)x_2)x_1(1+x_3) + 1) = 1$.
Therefore, the same value of $c$, namely $c=(1+(1+x_3)x_2)x_1(1+x_3)+1$, will again be strictly less than $b$, from which it follows that both sides of~\eqref{eqn 2} are equal to $(\vep - 1)(1+(1+x_3)x_2)+\eta - 1$.

\textbf{Subcase 5 of 5:} Suppose $\alpha \neq x_1(1+x_3)(1+(1+x_3)x_2)$ with $\delta = 0$ and $\beta = 0$.
In this case, $b=\gamma x_1(1+x_3)$.
Moreover, note that $\gamma > 0$ since otherwise, $b=0$ contradicting our bounds on $b$. 
If $x_1=1$, then $h(b) = h(\gamma(1+x_3)) = 1$.
Hence, we need only consider when $x_1>1$.
Since $\gamma > 0$ and $x_2 > 1$, we may take $c=1+x_3 < \gamma x_1(1+x_3) = b$, from which it is straightforward to find that both sides of~\eqref{eqn 2} are equal to $\gamma - 1$.

This completes our second case.

\textbf{Case:} $q_j = (1+(1+x_3)(1+(1+x_3)x_2)x_1)(1+(1+x_3)x_2)$. 
Thus, we consider $1\leq b \leq (1+(1+x_3)(1+(1+x_3)x_2)x_1)(1+(1+x_3)x_2) - 1$. Again, we will start by identifying those values of $b$ that satisfy~\eqref{eqn 1} and~\eqref{eqn 3}. 
It is straightforward to verify that
\begin{align*}
h(b) = b - x_1\floor{\frac{b(1+x_3)}{1+x_1(1+x_3)(1+(1+x_3)x_2)}} - x_2\floor{\frac{b(1+x_3)}{1+(1+x_3)x_2}}\, .
\end{align*}
Writing $b=\alpha(1+(1+x_3)x_2)+\beta$, where $0\leq \beta < 1+x_3$ and $0\leq \alpha < 1+x_1(1+x_3)(1+(1+x_3)x_2)$ for $\alpha,\beta \in \Z$, it follows that
\begin{align*}
h(b) &= h(\alpha(1+(1+x_3)x_2)+\beta) \\
&= \alpha + \beta - x_1\floor{\frac{(\alpha(1+(1+x_3)x_2)+\beta)(1+x_3)}{1+x_1(1+x_3)(1+(1+x_3)x_2)}} - x_2\floor{\frac{\beta(1+x_3)}{1+(1+x_3)x_2}}\, .
\end{align*}
There are two different possibilities for both $\alpha$ and $\beta$: either $\alpha = x_1(1+x_3)(1+(1+x_3)x_2)$ or $\alpha = \vep x_1+\eta$ where $0\leq \eta < x_1$ and $0\leq \vep < (1+x_3)(1+(1+x_3)x_2)$, and either $\beta = (1+x_3)x_2$ or $\beta = \gamma x_2+\delta$ where $0\leq \delta < x_2$ and $0\leq \gamma < 1+x_3$.
We consider the following subcases.

\textbf{Subcase 1 of 4:} Suppose $\alpha = x_1(1+x_3)(1+(1+x_3)x_2)$ and $\beta = (1+x_3)x_2$.
Then,
\begin{align*}
h(b) &= x_1(1+x_3)(1+(1+x_3)x_2) +(1+x_3)x_2  \\ &\quad - x_1\floor{\frac{x_1(1+x_3)^2(1+(1+x_3)x_2)^2+(1+x_3)^2x_2}{1+x_1(1+x_3)(1+(1+x_3)x_2)}} - x_2\floor{\frac{(1+x_3)^2x_2}{1+(1+x_3)x_2}}  \\
&= x_1 + x_2 \\
&> 1.
\end{align*}

\textbf{Subcase 2 of 4:} Suppose $\alpha = x_1(1+x_3)(1+(1+x_3)x_2)$ and $\beta\neq (1+x_3)x_2$. 
Writing $\beta = \gamma x_2+\delta$, where $0\leq \delta < x_2$ and $0\leq \gamma < 1+x_3$, we have that
\begin{align*}
h(b) &= x_1(1+x_3)(1+(1+x_3)x_2) + \gamma x_2 + \delta \\ &\quad - x_1\floor{\frac{(x_1(1+x_3)(1+(1+x_3)x_2)^2 + \beta)(1+x_3)}{1+x_1(1+x_3)(1+(1+x_3)x_2)}} - x_2\floor{\frac{(\gamma x_2 + \delta)(1+x_3)}{1+(1+x_3)x_2}}  \\
&= \delta - x_1\floor{\frac{\beta(1+x_3) - (1+x_3)(1+(1+x_3)x_2)}{1+x_1(1+x_3)(1+(1+x_3)x_2)}} - x_2\floor{\frac{\delta(1+x_3)-\gamma}{1+(1+x_3)x_2}} \\
&= \delta + x_1 - x_2\cdot \left. \begin{cases}-1, & \delta = 0, \gamma > 0 \\ \phantom{-}0, & \text{otherwise} \end{cases}\right\} \, .
\end{align*}
If $\delta = 0$ and $\gamma > 0$, then $h(b) = x_1 + x_2 > 1$.
Otherwise, we have that $h(b) = \delta + x_1$.
For this to be equal to $1$, it must be the case that $\delta = 0$ and $x_1 = 1$ since $x_1\geq 1$. 
Moreover, note that $\beta = 0$ here since $\delta = 0$ implies $\gamma  = 0$ (otherwise, we are in the previous case that $\delta = 0$ and $\gamma > 0$).
Therefore, when $x_1=1$, we have that $h((1+x_3)(1+(1+x_3)x_2)^2) = 1$.

\textbf{Subcase 3 of 4:} Suppose $\alpha \neq x_1(1+x_3)(1+(1+x_3)x_2)$ and $\beta = (1+x_3)x_2$. 
Writing $\alpha = \vep x_1+\eta$, where $0\leq \eta < x_1$ and $0\leq \vep < (1+x_3)(1+(1+x_3)x_2)$, we have that
\begin{align*}
h(b) &= \vep x_1+\eta + (1+x_3)x_2 - x_1\floor{\frac{((\vep x_1+\eta)(1+(1+x_3)x_2)+(1+x_3)x_2)(1+x_3)}{1+x_1(1+x_3)(1+(1+x_3)x_2)}} \\ &\qquad \qquad - x_2\floor{\frac{(1+x_3)^2x_2}{1+(1+x_3)x_2}}  \\
&= \eta - x_1\floor{\frac{(\eta(1+(1+x_3)x_2)+(1+x_3)x_2)(1+x_3)-\vep}{1+x_1(1+x_3)(1+(1+x_3)x_2)}}  + x_2\, .
\end{align*}
Given that $0\leq \eta < x_1$ and $0\leq \vep < (1+x_3)(1+(1+x_3)x_2)$, observe that 
\begin{align*}
\abs{(\eta(1+(1+x_3)x_2)+(1+x_3)x_2)(1+x_3)-\vep}<1+x_1(1+x_3)(1+(1+x_3)x_2).
\end{align*}
Therefore, $\floor{\frac{(\eta(1+(1+x_3)x_2)+(1+x_3)x_2)(1+x_3)-\vep}{1+x_1(1+x_3)(1+(1+x_3)x_2)}}$ is equal to either $0$ or $-1$. 
If
\begin{align*}
\floor{\frac{(\eta(1+(1+x_3)x_2)+(1+x_3)x_2)(1+x_3)-\vep}{1+x_1(1+x_3)(1+(1+x_3)x_2)}} = -1\, ,
\end{align*}
then $h(b) = \eta + x_1 + x_2 > 1$.
Otherwise, if
\begin{align*}
\floor{\frac{(\eta(1+(1+x_3)x_2)+(1+x_3)x_2)(1+x_3)-\vep}{1+x_1(1+x_3)(1+(1+x_3)x_2)}} = 0\, ,
\end{align*}
then either (A) $0<\eta < x_1$ or (B) $\eta = 0$ with $0\leq \vep \leq (1+x_3)^2x_2$.
If (A) holds, then $h(b) = \eta + x_2 > 1$ since $\eta > 0$. If (B) holds, then our same equation forces $\eta = 0$ and $x_2=1$ when $\vep \leq (1+x_3)^2$, which means that $h(\vep x_1(2+x_3)+(1+x_3)) = 1$ whenever $0\leq \vep \leq (1+x_3)^2$.

\textbf{Subcase 4 of 4:} Suppose $\alpha \neq x_1(1+x_3)(1+(1+x_3)x_2)$ and $\beta \neq (1+x_3)x_2$. 
Writing $\alpha = \vep x_1+\eta$ where $0\leq \eta < x_1$ and $0\leq \vep < (1+x_3)(1+(1+x_3)x_2)$ and $\beta = \gamma x_2+\delta$ where $0\leq \delta < x_2$ and $0\leq \gamma < 1+x_3$, we have that
\begin{align*}
h(b) &= \vep x_1+\eta + \gamma x_2 + \delta - x_1\floor{\frac{((\vep x_1+\eta)(1+(1+x_3)x_2)+\gamma x_2+\delta)(1+x_3)}{1+x_1(1+x_3)(1+(1+x_3)x_2)}} \\ &\qquad \qquad - x_2\floor{\frac{(\gamma x_2+\delta)(1+x_3)}{1+(1+x_3)x_2}}  \\
&= \eta + \delta - x_1\floor{\frac{(\eta(1+(1+x_3)x_2)+\gamma x_2+\delta)(1+x_3)-\vep}{1+x_1(1+x_3)(1+(1+x_3)x_2)}}  - x_2\floor{\frac{\delta(1+x_3)-\gamma}{1+(1+x_3)x_2}} \\
&= \eta + \delta - x_1\floor{\frac{(\eta(1+(1+x_3)x_2)+\gamma x_2+\delta)(1+x_3)-\vep}{1+x_1(1+x_3)(1+(1+x_3)x_2)}}  - x_2\cdot \left. \begin{cases}-1, & \delta = 0, \gamma > 0 \\ \phantom{-}0, & \text{otherwise} \end{cases}\right\} \, .
\end{align*}
Given the bounds on $\eta$, $\gamma$, $\delta$, and $\vep$, observe that
\begin{align*}
\abs{(\eta(1+(1+x_3)x_2)+\gamma x_2+\delta)(1+x_3)-\vep}<1+x_1(1+x_3)(1+(1+x_3)x_2).
\end{align*}
Therefore, \[
\floor{\frac{(\eta(1+(1+x_3)x_2)+\gamma x_2 + \delta)(1+x_3)-\vep}{1+x_1(1+x_3)(1+(1+x_3)x_2)}}
\]
is equal to either $0$ or $-1$. 
\begin{enumerate}
\item[(i)] Suppose $\delta=0$ and $\gamma > 0$. If
\begin{align*}
\floor{\frac{(\eta(1+(1+x_3)x_2)+\gamma x_2 + \delta)(1+x_3)-\vep}{1+x_1(1+x_3)(1+(1+x_3)x_2)}} = -1\, ,
\end{align*}
then $h(b) = \eta + x_1+x_2 > 1$.
Otherwise, if
\begin{align*}
\floor{\frac{(\eta(1+(1+x_3)x_2)+\gamma x_2 + \delta)(1+x_3)-\vep}{1+x_1(1+x_3)(1+(1+x_3)x_2)}} = 0\, ,
\end{align*}
then either (A) $0<\eta < x_1$ or (B) $\eta = 0$ with $0\leq \vep \leq \gamma x_2(1+x_3)$. 
If (A) holds, then $h(b)=\eta + x_2 > 1$ since $\eta > 0$ and $x_2\geq 1$. If (B) holds, then our same equation forces $\eta = 0$ and $x_2=1$ when $\vep \leq \gamma (1+x_3)$, which means that $h(\vep x_1(2+x_3)+\gamma) = 1$ whenever $0\leq \vep \leq \gamma(1+x_3)$. 
\item[(ii)] Suppose otherwise, i.e., $\delta=0$ and $\gamma > 0$ does not hold. If
\begin{align*}
\floor{\frac{(\eta(1+(1+x_3)x_2)+\gamma x_2 + \delta)(1+x_3)-\vep}{1+x_1(1+x_3)(1+(1+x_3)x_2)}} = -1\, ,
\end{align*}
note that $\eta = 0$. Then, $h(b) = \delta + x_1$. For this to equal $1$, it must be the case that $\delta = 0$ (and hence $\gamma = 0$ since otherwise, we would be in the previous case) and $x_1 = 1$. 
Since $\gamma = \delta = 0$, we have that $\beta = 0$.
Therefore, given that $\eta = 0$ as well, it must be that $\vep > 0$ since otherwise, $b=0$ contradicting our bounds on $b$. 
In this case, we have that $h(\vep(1+(1+x_3)x_2))=1$ whenever $0< \vep < (1+x_3)(1+(1+x_3)x_2)$.
Otherwise, if
\begin{align*}
\floor{\frac{(\eta(1+(1+x_3)x_2)+\gamma x_2 + \delta)(1+x_3)-\vep}{1+x_1(1+x_3)(1+(1+x_3)x_2)}} = 0\, ,
\end{align*}
then either (A) $0<\eta < x_1$ or (B) $\eta = 0$ with $0\leq \vep \leq (\gamma x_2+\delta)(1+x_3)$.
If (A) holds, then $h(b) = \eta + \delta$, which setting equal to $1$ forces $\eta = 1$ and $\delta = 0$ since $\eta > 0$. 
Note that $\delta = 0$ forces $\gamma = 0$ since otherwise, we would be in the previous case. 
Therefore, we have that $h((\vep x_1 + 1)(1+(1+x_3)x_2)) = 1$ whenever $0\leq \vep < (1+x_3)(1+(1+x_3)x_2)$.
On the other hand, if (B) holds, then our same equation forces $\eta = 0$ and $\delta = 1$ when $0\leq \vep \leq (\gamma x_2+1)(1+x_3)$, which means that $h(\vep x_1(1+(1+x_3)x_2)+\gamma x_2 + 1) = 1$ whenever $0\leq \gamma < 1+x_3$ and $0\leq \vep \leq (\gamma x_2 +1)(1+x_3)$.
\end{enumerate}

We summarize the values of $b$ for which $h(b) = 1$ that were just derived:
\begin{itemize}
\item If $x_1 = 1$, $\beta = 0$, and $0<\vep\leq (1+x_3)(1+(1+x_3)x_2)$, we have $b = \vep(1+(1+x_3)x_2)$.
\item If $x_2 = 1$, $\delta = \eta = 0$, $0<\gamma \leq 1+x_3$, and $0\leq \vep \leq \gamma(1+x_3)$, we have $b=\vep x_1(2+x_3)+\gamma$.
\item If $\beta = 0$, $\eta = 1$, and $0\leq \vep <(1+x_3)(1+(1+x_3)x_2)$, we have $b=(\vep x_1+1)(1+(1+x_3)x_2)$.
\item If $\eta = 0$, $\delta = 1$, $0\leq \gamma < 1+x_3$, and $0\leq \vep \leq (\gamma x_2+1)(1+x_3)$, we have $b=\vep x_1(1+(1+x_3)x_2)+\gamma x_2 + 1$.
\end{itemize}

Our next goal is to establish that~\eqref{eqn 2} is always satisfied; recall that we are in the case where $q_j = (1+(1+x_3)(1+(1+x_3)x_2)x_1)(1+(1+x_3)x_2)$. 
If $q_i = (1+(1+x_3)(1+(1+x_3)x_2)x_1)(1+(1+x_3)x_2)$, the result is trivial. 
If $q_i = (1+x_3)(1+x_1(1+x_3)(1+(1+x_3)x_2))$, we write $b=\alpha(1+(1+x_3)x_2)+\beta$, where $0\leq \beta < 1+(1+x_3)x_2$ and $0\leq \alpha < 1+x_1(1+x_3)(1+(1+x_3)x_2)$.
Note that $b\neq 1+(1+x_3)x_2$ since $h(1+(1+x_3)x_2) = 1$. 
If $b > 1+(1+x_3)x_2$, we set $c=1+(1+x_3)x_2$ from which it is straightforward to compute that both sides of~\eqref{eqn 2} are equal to
\begin{align*}
(\alpha - 1)(1+x_3) + \floor{\frac{\beta(1+x_3)}{1+(1+x_3)x_2}}\, .
\end{align*}
Otherwise, if $b<1+(1+x_3)x_2$, note that $\alpha = 0$ and $b = \beta$. 
Therefore, to ensure we satisfy~\eqref{eqn 1}, we consider $2\leq \beta < 1+(1+x_3)x_2$. 
Moreover, note that $\beta \not\equiv 1\bmod x_2$ since otherwise $h(b) = h(\beta) = 1$. 
Now, suppose $\beta = (1+x_3)x_2$. If $x_2 = 1$, then $h(\beta) = h(1+x_3) = 1$, so we may assume $x_2 > 1$.
Setting $c=1$ and since $x_2 > 1$, it is straightforward to compute that both sides of~\eqref{eqn 2} are equal to $x_3$. 
Otherwise, if $\beta \neq (1+x_3)x_2$, we write $\beta = \gamma x_2 + \delta$, where $0\leq \delta < x_2$ and $0\leq \gamma < 1+x_3$. 
Note that $\delta \neq 1$ since $h(\gamma x_2 + 1) = 1$ for $0\leq \gamma < 1+x_3$. 
Suppose $\delta > 1$. Then, choosing $c=\gamma x_2 + 1 < \gamma x_2 + \delta = b$, it is straightforward to show that both sides of~\eqref{eqn 2} are equal to $0$. 
On the other hand suppose $\delta = 0$, so $\beta = \gamma x_2$, where $\gamma > 0$. 
If $x_2 = 1$, then $h(\beta) = 1$, so we may assume $x_2 > 1$. 
Taking $c=1$ and observing that $x_2>1$ implies $-x_2(1+x_3)< -(1+x_3)-\gamma$, it is straightforward to compute that both sides of~\eqref{eqn 2} are equal to $\gamma - 1$. 

Finally, if $q_i = (1+x_3)(1+(1+x_3)x_2)$, the analysis becomes a bit more complicated. 
We start by again writing $b = \alpha(1+(1+x_3)x_2)+\beta$, where $0\leq \beta < 1+(1+x_3)x_2$ and $0\leq \alpha < 1+x_1(1+x_3)(1+(1+x_3)x_2)$.
Suppose $\alpha = x_1(1+x_3)(1+(1+x_3)x_2)$.
If $x_1=1$, then $\beta \neq 0$ since otherwise $h(b) = 1$.
Thus, we may consider $c=(1+x_3)(1+(1+x_3)x_2)^2 < (1+x_3)(1+(1+x_3)x_2)^2+\beta = b$ from which it is straightforward to compute that both sides of~\eqref{eqn 2} are equal to $0$.
If $x_1>1$, observe that \[
-x_1(1+x_3)(1+(1+x_3)x_2) < -(1+x_3)(2+(1+x_3)x_2) \leq (\beta-1)(1+x_3) - (1+x_3)(1+(1+x_3)x_2) < 0 \, .
\]
Choosing $c=1$, the previous inequality readily gives that both sides of~\eqref{eqn 2} are equal to $(1+x_3)(1+(1+x_3)x_2) - 1$.

Now, suppose $\alpha \neq x_1(1+x_3)(1+(1+x_3)x_2)$.
Then, we may write $\alpha = \vep x_1 + \eta$, where $0\leq \eta < x_1$ and $0\leq \vep < (1+x_3)(1+(1+x_3)x_2)$, and so $b=(\vep x_1 + \eta)(1+(1+x_3)x_2)+\beta$. 
Suppose $\beta = (1+x_3)x_2$.
If $\eta \geq 1$, then we may consider $c=(\vep x_1+1)(1+(1+x_3)x_2)<(\vep x_1+\eta)(1+(1+x_3)x_2) + (1+x_3)x_2 = b$ from which it is straightforward to show that both sides of~\eqref{eqn 2} are equal to $0$. 
Otherwise, if $\eta = 0$, note that we may assume $x_2>1$ since $x_2=1$ gives that $h(b) = h(\vep x_1(2+x_3)+(1+x_3)) = 1$.

We consider two possible cases, namely when $0 \leq \vep < 1+x_3$ and when $1+x_3\leq \vep <(1+x_3)(1+(1+x_3)x_2)$. 
If $0 \leq \vep < 1+x_3$, we consider $c = \vep x_1(1+(1+x_3)x_2)+x_2x_3 + 1$ which is strictly less than $b=\vep x_1(1+(1+x_3)x_2)+(1+x_3)x_2$ as $x_2 > 1$. 
With this choice of $c$ and since $0\leq \vep < 1+x_3$, it is straightforward to show that both sides of~\eqref{eqn 2} are equal to $0$. 
On the other hand, if $1+x_3\leq \vep <(1+x_3)(1+(1+x_3)x_2)$, we may consider $c=1+x_1(1+x_3)(1+(1+x_3)x_2) < \vep x_1(1+(1+x_3)x_2)+(1+x_3)x_2 = b$ from which it is straightforward to compute that both sides of~\eqref{eqn 2} are equal to
\begin{align*}
\vep - (1+x_3) + \floor{\frac{(1+x_3)^2x_2 - \vep}{1+x_1(1+x_3)(1+(1+x_3)x_2)}}\, .
\end{align*}
Now, suppose $\beta \neq (1+x_3)x_2$. 
Then, we may write $\beta = \gamma x_2 + \delta$, where $0\leq \delta < x_2$ and $0\leq \gamma < 1+x_3$, and so $b$ can be written as $b=(\vep x_1 + \eta)(1+(1+x_3)x_2) + \gamma x_2 + \delta$. 
We consider the following possible subcases.

\textbf{Subcase 1 of 4:} Suppose $\eta > 0$ and $\delta > 0$. 
We consider $c=(\vep x_1 + 1)(1+(1+x_3)x_2) < (\vep x_1 + 1)(1+(1+x_3)x_2) + \gamma x_2 + \delta = b$. 
Since $0 < \eta < x_1$ and $0 < \gamma x_2 + \delta <(1+x_3)x_2$, it follows that $0 < (\eta - 1)(1+(1+x_3)x_2)+\gamma x_2 +\delta < x_1(1+(1+x_3)x_2)$.
Therefore, it is straightforward to verify that both sides of~\eqref{eqn 2} are equal to $0$.

\textbf{Subcase 2 of 4:} Suppose $\eta = 0$ and $\delta > 0$. 
If $\vep \leq (\gamma x_2 + 1)(1+x_3)$, note that $\delta \neq 1$ since otherwise, $h(b) = h(\vep x_1(1+(1+x_3)x_2) + \gamma x_2 + 1) = 1$.
Thus, we have that $\delta > 1$, and we consider $c=\vep x_1(1+(1+x_3)x_2) + \gamma x_2 + 1 < \vep x_1(1+(1+x_3)x_2) + \gamma x_2 + \delta = b$. 
Given that $\vep \leq (\gamma x_2 + 1)(1+x_3)$ and $1< \delta < x_2$, it is straightforward to verify that our choice of $c$ gives that both sides of~\eqref{eqn 2} are equal to $0$. 
Otherwise, if $(\gamma x_2 + 1)(1+x_3) < \vep$ (and hence, $\vep > 1+x_3$), we consider $c=1+x_1(1+x_3)(1+(1+x_3)x_2) < \vep x_1(1+(1+x_3)x_2)+\gamma x_2 + \delta = b$.
With this choice of $c$, it is straightforward to verify that both sides of~\eqref{eqn 2} are equal to 
\begin{align*}
\vep - (1+x_3) + \floor{\frac{(\gamma x_2 + \delta)(1+x_3) - \vep}{1+x_1(1+x_3)(1+(1+x_3)x_2)}}\, .
\end{align*}

\textbf{Subcase 3 of 4:} Suppose $\eta > 0$ and $\delta = 0$. 
If $\gamma > 0$, we may consider $c=(\vep x_1 + 1)(1+(1+x_3)x_2) < (\vep x_1 + \eta)(1+(1+x_3)x_2) + \gamma x_2 = b$ from which it is straightforward to compute that both sides of~\eqref{eqn 2} are equal to $0$. 
On the other hand, if $\gamma = 0$, note that $\eta \neq 1$ since otherwise, $h(b) = h((\vep x_1 + \eta)(1+(1+x_3)x_2)) = 1$.
Therefore, we may again choose $c = (\vep x_1 + 1)(1+(1+x_3)x_2) < (\vep x_1 + \eta)(1+(1+x_3)x_2) = b$.
Since $1< \eta < x_1$, it is straightforward to verify that both sides of~\eqref{eqn 2} are equal to $0$. 

\textbf{Subcase 4 of 4:} Suppose $\eta = \delta = 0$. 
Further suppose $\gamma x_2(1+x_3) < \vep$. 
If $\gamma >0$, then it follows that $\gamma x_2 \geq 1$.
Therefore, our assumed inequality implies  $\vep > 1+x_3$, so we may consider $c=1+x_1(1+x_3)(1+(1+x_3)x_2) < \vep x_1 (1+(1+x_3)x_2)+\gamma x_2 = b$.
Since $\gamma x_2(1+x_3) < \vep$, our choice of $c$ readily gives that both sides of~\eqref{eqn 2} are equal to $\vep - x_3 -2$. 
Otherwise, if $\gamma = 0$ (and hence, $\vep > 0$ since we assumed $\gamma x_2(1+x_3) < \vep$), we may assume $x_1 > 1$ since otherwise, we would have that $h(b) = h(\vep (1+(1+x_3)x_2)) = 1$. 
Thus, since $x_1 > 1$, we may take $c=1+(1+x_3)x_2 < \vep x_1 (1+(1+x_3)x_2) = b$. 
Observe that the bounds on $\vep$ and $x_1 > 1$ imply $-x_1(1+x_3)(1+(1+x_3)x_2) \leq -2(1+x_3)(1+(1+x_3)x_2) < -(1+x_3)(1+(1+x_3)x_2) - \vep < 0$.
Consequently, it is straightforward to verify that both sides of~\eqref{eqn 2} are equal to $\vep - 1$. 
Now, suppose $\gamma x_2(1+x_3) > \vep$. 
Note that $\gamma\neq 0$ since otherwise, $\vep < 0$ contradicting our initial bounds on $\vep$. 
Thus, we have that $\gamma > 0$.
Moreover, if $x_2 = 1$, it follows that $h(b) = h(\vep x_1 (2+x_3)+\gamma) = 1$, so we may assume $x_2 > 1$. 
Given the addition restriction that $\vep \leq (\gamma x_2-1)(1+x_3)$, we may choose $c=1$ from which the inequality $\vep \leq (\gamma x_2-1)(1+x_3)$ readily implies both sides of~\eqref{eqn 2} are equal to $\vep$.
However, for $(\gamma x_2-1)(1+x_3) < \vep < \gamma x_2(1+x_3)$, we consider $c=1+x_1(1+x_3)(1+(1+x_3)x_2)$.
Note that $x_2>1$ and $\gamma > 0$ together with our restriction on $\vep$ imply that $\vep > 1+x_3$. 
Therefore, we satisfy $c < \vep x_1(1+(1+x_3)x_2)+\gamma x_2 = b$, and a straightforward computation gives that both sides of~\eqref{eqn 2} are equal to $\vep - (1+x_3)$.
Finally, suppose $\vep = \gamma x_2 (1+x_3)$.
Given this equality, note that neither $\gamma$ nor $\vep$ can be equal to $0$ since otherwise, we would have $\eta = \delta = \gamma = \vep = 0$, implying $b=0$. 
This, of course, contradicts the bounds on $b$. 
Moreover, we may again assume $x_2>1$ (and thus, $\vep > 1+x_3$) since otherwise, $h(b) = h(\vep x_1(2+x_3)+\gamma) = 1$.
Since $\vep > 1+x_3$, we may consider $c=1+x_1(1+x_3)(1+(1+x_3)x_2) < \vep x_1(1+(1+x_3)x_2)+\gamma x_2 = b$.
This choice of $c$ readily gives that both sides of~\eqref{eqn 2} are equal to $\vep - (1+x_3)$. 

In any case, we find that both sides of~\eqref{eqn 2} are equivalent for each possible $q_i$, thereby completing our third and final case.
Thus, we have established IDP for $\br$-vectors of type $(vi)$.

\subsection{Proof of IDP for type $(vii)$ in Theorem~\ref{thm:3suppclassify}}
Here, we verify IDP for $\br$-vectors of type $(vii)$ using Theorem~\ref{thm:idpreflexives}. 
  Again, we must consider three cases corresponding to three possible values of $q_j$.
  
  \textbf{Case: $q_j=1+x_3$}. 
  Since $1\leq b \leq x_3$ in this case, it is straightforward to verify that $h(b) = b$. 
  Hence, the $b$-values we are required to check in~\eqref{eqn 1}  are $2\leq b \leq x_3$. 
  To verify that~\eqref{eqn 2} always has the desired solution, we consider three cases.
  If $q_i=1+x_3$, the result is trivial.
  If $q_i=(1+x_3)(1+x_1(1+x_3))$, then we may select $c=1$, from which it follows that both sides of~\eqref{eqn 2} are equal to $(b-1)(1+x_1(1+x_3))$.
  If $q_i=(1+(1+x_3)x_1)(1+(1+x_3)x_2)$, then we may again set $c=1$, from which it is straightforward to compute that both sides of~\eqref{eqn 2} are equal to $(b-1)(x_1+x_2+x_1x_2(1+x_3))$.
  This completes our first case.
  
  \textbf{Case: $q_j=(1+x_3)(1+x_1(1+x_3))$}.
  It is straightforward to verify that 
  \[
    h(b)=b-x_1\left\lfloor \frac{b}{1+x_1(1+x_3)} \right\rfloor - x_3\left\lfloor \frac{b}{1+x_3} \right\rfloor \, ,
  \]
  where the values of $b$ range from $1$ to $(1+x_3)(1+x_1(1+x_3))-1$.
  To verify that~\eqref{eqn 2} always has the desired solution, we consider three cases.
  If $q_i=(1+x_3)(1+x_1(1+x_3))$, the result is trivial.
  If $q_i=1+x_3$, then we write $b = \alpha (1+x_1(1+x_3)) + \beta$, where $0\leq \beta < 1+x_1(1+x_3)$ and $0\leq \alpha < 1+x_3$ for $\alpha,\beta \in \Z$.
  Consequently, we have that
  \begin{align*}
  	h(b) &= h(\alpha (1+x_1(1+x_3)) + \beta) \\ 
  	&= \alpha (1+x_1(1+x_3)) + \beta - \alpha x_1 - \alpha x_1x_3 - x_3\left\lfloor \frac{\alpha+\beta}{1+x_3} \right\rfloor \\
  	&= \alpha + \beta - x_3\left\lfloor \frac{\alpha+\beta}{1+x_3} \right\rfloor\, .
  \end{align*}
  If $\beta > 0$, we may select $c=1$, from which it follows that both sides of~\eqref{eqn 2} are equal to $\alpha$. 
  If $\beta = 0$, then our formula for $h(b)$ reduces to 
  \begin{align*}
  h(b) = h(\alpha (1+x_1(1+x_3))) = \alpha
  \end{align*}
  since $0\leq \alpha < 1+x_3$. 
  Thus, to satisfy~\eqref{eqn 1}, it must be that $\alpha > 1$, implying $b = \alpha(1+x_1(1+x_3)) > 1+x_1(1+x_3)$. In this case, taking $c = 1+x_1(1+x_3)$, it is straightforward to verify that both sides of ~\eqref{eqn 2} are equal to $\alpha - 1$, and~\eqref{eqn 3} is satisfied as $h(1+x_1(1+x_3)) = 1$.
  Finally, if $q_i = (1+(1+x_3)x_1)(1+(1+x_3)x_2)$, then we write $b = \alpha (1+x_3) + \beta$, where $0\leq \beta < 1+x_3$ and $0\leq \alpha < 1+x_1(1+x_3)$ for $\alpha,\beta \in \Z$.
  Consequently, since $0\leq \beta < 1+x_3$, we have that
  \begin{align*}
  	h(b) &= h(\alpha (1+x_3) + \beta) \\ 
  	&= \alpha (1+x_3) + \beta - x_1\left\lfloor \frac{\alpha(1+x_3)+\beta}{1+x_1(1+x_3)} \right\rfloor - \alpha x_3 - x_3\left\lfloor \frac{\beta}{1+x_3} \right\rfloor \\
  	&= \alpha  + \beta - x_1\left\lfloor \frac{\alpha(1+x_3)+\beta}{1+x_1(1+x_3)} \right\rfloor  \, .
  \end{align*}
  If $\beta > 0$, we may select $c=1$, from which it is straightforward to verify that both sides of~\eqref{eqn 2} are equal to $\alpha(1 + (1+x_3)x_2) + (\beta-1)x_2$ (since $0\leq \beta - 1 < x_3$).
  On the other hand, if $\beta = 0$, then our formula for $h(b)$ reduces to 
  \begin{align*}
  h(b) = h(\alpha(1+x_3)) = \alpha - \left\lfloor \frac{\alpha(1+x_3)}{1+x_1(1+x_3)} \right\rfloor.
  \end{align*}
  In order to satisfy~\eqref{eqn 1}, it must be that $\alpha > 1$, which implies $b=\alpha(1+x+3) > 1+x_3$. Thus, in this case, we consider $c = 1+x_3$. Clearly, $h(1+x_3) = 1$, giving~\eqref{eqn 3}, and moreover, it is straightforward to verify that both sides of ~\eqref{eqn 2} when $c=1+x_3$ are equal to $(\alpha - 1)(1+(1+x_3)x_2)$. This completes our second case.
  
   \textbf{Case: $q_j=(1+(1+x_3)x_1)(1+(1+x_3)x_2)$}.
  We first identify those values of $b$ that satisfy~\eqref{eqn 1} and~\eqref{eqn 3}.
  It is straightforward to verify that
  \[
    h(b)=b-x_1\left\lfloor \frac{b(1+x_3)}{(1+(1+x_3)x_1)(1+(1+x_3)x_2)}\right\rfloor -x_2\left\lfloor \frac{b(1+x_3)}{(1+(1+x_3)x_2)} \right\rfloor \, .
  \]
  Writing $b=\alpha (1+(1+x_3)x_2)+\beta$, where $0\leq \beta \leq (1+x_3)x_2$ and $0\leq \alpha \leq (1+x_3)x_1$, it follows that
  \begin{align*}
    h(b) &= h(\alpha(1+(1+x_3)x_2)+\beta) \\
    &= \alpha + \beta - x_1\left\lfloor \frac{\alpha(1+x_3)(1+(1+x_3)x_2) + \beta(1+x_3)}{(1+(1+x_3)x_1)(1+(1+x_3)x_2)}\right \rfloor - x_2\left\lfloor \frac{\beta(1+x_3)}{1+(1+x_3)x_2} \right\rfloor \, .
  \end{align*}
  Now, writing $\beta = \gamma x_2 + \delta$, where $0\leq \delta < x_2$ and $0\leq \gamma \leq 1+x_3$, it follows that
  \begin{align*}
  h(b) = \alpha + \delta - x_1\left\lfloor \frac{\alpha(1+x_3)(1+(1+x_3)x_2) + (\gamma x_2 + \delta)(1+x_3)}{(1+(1+x_3)x_1)(1+(1+x_3)x_2)}\right \rfloor - x_2\left\lfloor \frac{\delta(1+x_3) - \gamma}{1+(1+x_3)x_2} \right\rfloor \, .
  \end{align*}
  Since $0\leq \delta < x_2$ and $0\leq \gamma \leq 1+x_3$, observe that
  \begin{align*}
  \left\lfloor \frac{\delta(1+x_3) - \gamma}{1+(1+x_3)x_2} \right\rfloor = \begin{cases}
  -1, & \delta=0, \gamma >0 \\
  \phantom{-}0, & \text{otherwise} \quad .
\end{cases}
  \end{align*}
  We further write $\alpha = \vep x_1 + \eta$, where $0\leq \eta < x_1$ and $0 \leq \vep \leq 1+x_3$. Then,
  \begin{align*}
  	h(b) &= \eta + \delta - x_1\left\lfloor \frac{(\eta(1+x_3)-\vep)(1+(1+x_3)x_2) + (\gamma x_2 + \delta)(1+x_3)}{(1+(1+x_3)x_1)(1+(1+x_3)x_2)}\right \rfloor - x_2\left\lfloor \frac{\delta(1+x_3) - \gamma}{1+(1+x_3)x_2} \right\rfloor \\
    &= \eta + \delta - x_1\left\lfloor \frac{(\eta(1+x_3) - \vep + \gamma)(1+(1+x_3)x_2) + \delta(1+x_3)-\gamma}{(1+(1+x_3)x_1)(1+(1+x_3)x_2)}\right \rfloor - x_2\left\lfloor \frac{\delta(1+x_3) - \gamma}{1+(1+x_3)x_2} \right\rfloor   \, .
  \end{align*}
  Given the bounds on $\vep$, $\eta$, $\gamma$, and $\delta$, note that $-(1+x_3) \leq \delta(1+x_3) - \gamma < 1+(1+x_3)x_2$ and $-(1+x_3) \leq \eta(1+x_3) -\vep + \gamma \leq (1+x_3)x_1$. 
  Consequently, it follows that 
  \begin{align*}
  \abs{(\eta(1+x_3) - \vep + \gamma)(1+(1+x_3)x_2) + \delta(1+x_3)-\gamma}< (1+(1+x_3)x_1)(1+(1+x_3)x_2),
  \end{align*} 
  and this implies that $\left\lfloor \frac{(\eta(1+x_3) - \vep + \gamma)(1+(1+x_3)x_2) + \delta(1+x_3)-\gamma}{(1+(1+x_3)x_1)(1+(1+x_3)x_2)}\right \rfloor$ is equal to either $0$ or $-1$. 
  To resolve this floor function, we consider the following subcases which analyze the sign of the numerator of its argument.
  
  \textbf{Subcase 1 of 6:} Suppose $\eta =0$ and $\vep > \gamma$. 
  Then, since $\delta(1+x_3)-\gamma < 1+(1+x_3)x_2$, the numerator above will be negative, implying 
  \begin{align*}
  \left\lfloor \frac{(\eta(1+x_3) - \vep + \gamma)(1+(1+x_3)x_2) + \delta(1+x_3)-\gamma}{(1+(1+x_3)x_1)(1+(1+x_3)x_2)}\right \rfloor = -1.
  \end{align*} 
  Therefore, our equation for $h(b)$ simplifies to
  \begin{align*}
  	h(b) = \delta + x_1 - x_2\cdot \left.\begin{cases} -1, & \delta =0, \gamma > 0 \\ \phantom{-}0, & \text{otherwise}  \end{cases} \right\} .
  \end{align*}
  If $\delta = 0$ and $\gamma > 0$, then $h(b) = h(\vep x_1(1+(1+x_3)x_2) + \gamma x_2) = x_1 + x_2 > 1$.
  If $\delta = \gamma = 0$, then $h(b) = h(\vep x_1(1+(1+x_3)x_2)) = x_1$.
  Thus, if $x_1=1$, we have that $h(\vep x_1(1+(1+x_3)x_2)) = 1$ whenever $\vep > 0$. 
  If $\delta > 0$, then  $h(b) = h(\vep x_1(1+(1+x_3)x_2) + \gamma x_2 + \delta) = \delta + x_1 > 1$. 
  
  \textbf{Subcase 2 of 6:} Suppose $\eta = 0$ and $\vep < \gamma$. 
  Then, $\eta(1+x_3) - \vep + \gamma > 0$, and consequently, the numerator of our floor function argument will be positive.
  Hence, 
  \begin{align*}\left\lfloor \frac{(\eta(1+x_3) - \vep + \gamma)(1+(1+x_3)x_2) + \delta(1+x_3)-\gamma}{(1+(1+x_3)x_1)(1+(1+x_3)x_2)}\right \rfloor = 0,
  \end{align*} 
  simplifying our formula for $h(b)$ to
  \begin{align*}
  	h(b) = \delta - x_2\cdot \left.\begin{cases} -1, & \delta =0, \gamma > 0 \\ \phantom{-}0, & \text{otherwise} \end{cases} \right\} .
  \end{align*}
  If $\delta = 0$ and $\gamma > 0$, then $h(b) = h(\vep x_1(1+(1+x_3)x_2) + \gamma x_2) = x_2$.
  Thus, if $x_2=1$, we have that $h(b) = h(\vep x_1(2+x_3) + \gamma) = 1$ whenever $\vep < \gamma$.
  Otherwise, $h(b) = \delta$, which forces $\delta = 1$, i.e., $h(\vep x_1(1+(1+x_3)x_2) + \gamma x_2 + 1) = 1$ whenever $\vep < \gamma$.
  
  \textbf{Subcase 3 of 6:} Suppose $\eta = 0$ and $\vep = \gamma$. 
  Then, $\eta(1+x_3) - \vep + \gamma = 0$, so the numerator of our floor function argument reduces to $\delta(1+x_3) - \gamma$. 
  If $\delta = 0$ and $\gamma > 0$, then $\delta(1+x_3) - \gamma < 0$ which implies 
  \begin{align*}
  \left\lfloor \frac{(\eta(1+x_3) - \vep + \gamma)(1+(1+x_3)x_2) + \delta(1+x_3)-\gamma}{(1+(1+x_3)x_1)(1+(1+x_3)x_2)}\right \rfloor = -1.
  \end{align*}
  Hence, for $\vep = \gamma > 0$, $h(b) = h(\gamma x_1(1+(1+x_3)x_2) + \gamma x_2) = x_1 + x_2 > 1$. 
  If $\delta = \gamma = 0$, then $\delta(1+x_3) - \gamma = 0$ which implies 
  \begin{align*}
  \left\lfloor \frac{(\eta(1+x_3) - \vep + \gamma)(1+(1+x_3)x_2) + \delta(1+x_3)-\gamma}{(1+(1+x_3)x_1)(1+(1+x_3)x_2)}\right \rfloor = 0.
  \end{align*}
  Hence, $h(b) = h(0) = 0$.
  If $\delta > 0$, then $\delta(1+x_3) - \gamma > 0$ since $0\leq \gamma \leq (1+x_3)$. 
  Therefore, 
  \begin{align*}
  \left\lfloor \frac{(\eta(1+x_3) - \vep + \gamma)(1+(1+x_3)x_2) + \delta(1+x_3)-\gamma}{(1+(1+x_3)x_1)(1+(1+x_3)x_2)}\right \rfloor = 0.
  \end{align*} 
  As such, we have that $h(b) = h(\gamma x_1(1+(1+x_3)x_2) + \gamma x_2 + \delta) = \delta$, which forces $\delta = 1$, i.e., $h(b) = h(\gamma x_1(1+(1+x_3)x_2) + \gamma x_2 + 1) = 1$ for $0\leq \gamma \leq 1+x_3$.
  
  \textbf{Subcase 4 of 6:} Suppose $\eta = 1$ and $0\leq \vep < 1+x_3$. 
  Then, it follows that $\eta(1+x_3) - \vep + \gamma>0$.
  Consequently, since $\delta(1+x_3) - \gamma < 1+(1+x_3)x_2$, we have that the numerator of our floor function argument is positive, implying 
  \begin{align*}
  \left\lfloor \frac{(\eta(1+x_3) - \vep + \gamma)(1+(1+x_3)x_2) + \delta(1+x_3)-\gamma}{(1+(1+x_3)x_1)(1+(1+x_3)x_2)}\right \rfloor = 0
\end{align*}  
  for any $0\leq \gamma \leq 1+x_3$ and $0\leq \delta < x_2$. 
  As a result, if $\delta = 0$ and $\gamma > 0$, then $h(b) = h((\vep x_1 + 1)(1+(1+x_3)x_2) + \gamma x_2) = 1 + x_2 > 1$.
  If $\delta = \gamma = 0$, then $h(b) = h((\vep x_1 + 1)(1+(1+x_3)x_2)) = \eta = 1$ whenever $0\leq \vep < 1+x_3$.
  If $\delta > 0$, then $h(b) = h((\vep x_1 + 1)(1+(1+x_3)x_2) + \gamma x_2 + \delta) = 1 + \delta > 1$.
  
  \textbf{Subcase 5 of 6:} Suppose $\eta = 1$ and $\vep = 1+x_3$.
  Then, the numerator of our floor function argument reduces to $\gamma(1+x_3)x_2 + \delta(1+x_3)$, which is certainly nonnegative.
  Therefore, it follows that 
  \begin{align*}
  \left\lfloor \frac{(\eta(1+x_3) - \vep + \gamma)(1+(1+x_3)x_2) + \delta(1+x_3)-\gamma}{(1+(1+x_3)x_1)(1+(1+x_3)x_2)}\right \rfloor = 0,
  \end{align*} 
  which simplifies our formula for $h(b)$ to
  \begin{align*}
  h(b) = h((1+(1+x_3)x_1)(1+(1+x_3)x_2) + \gamma x_2 + \delta) = 1 + \delta - x_2\cdot \left. \begin{cases} -1, & \delta =0, \gamma > 0 \\ \phantom{-}0, & \text{otherwise} \end{cases} \right\} .
\end{align*}   
  Therefore, if $\delta = 0$ and $\gamma > 0$, it follows that $h(b) = h((1+(1+x_3)x_1)(1+(1+x_3)x_2) + \gamma x_2) = 1 + x_2 > 1$.
  Otherwise, $h(b) = h((1+(1+x_3)x_1)(1+(1+x_3)x_2) + \gamma x_2 + \delta) = 1 + \delta$. 
  So, for this to be equal to $1$, it must be that $\delta = \gamma = 0$, i.e., $h((1+(1+x_3)x_1)(1+(1+x_3)x_2)) = 1$. 
  
  \textbf{Subcase 6 of 6:} Suppose $\eta > 1$. 
  Then, it follows that $\eta(1+x_3) - \vep + \gamma>0$, and consequently, since $\delta(1+x_3) - \gamma < 1+(1+x_3)x_2$, we have that the numerator of our floor function argument is positive.
  Therefore, we have that 
  \begin{align*}
  \left\lfloor \frac{(\eta(1+x_3) - \vep + \gamma)(1+(1+x_3)x_2) + \delta(1+x_3)-\gamma}{(1+(1+x_3)x_1)(1+(1+x_3)x_2)}\right \rfloor = 0,
  \end{align*} 
  which simplifies our formula for $h(b)$ to
  \begin{align*}
  h(b) = h((\vep x_1 + \eta)(1+(1+x_3)x_2) + \gamma x_2 + \delta) = \eta + \delta - x_2\cdot \left. \begin{cases} -1, & \delta =0, \gamma > 0 \\ \phantom{-}0, & \text{otherwise} \end{cases} \right\} .
\end{align*}   
  Therefore, if $\delta = 0$ and $\gamma > 0$, it follows that $h(b) = h((\vep x_1 + \eta)(1+(1+x_3)x_2) + \gamma x_2) = \eta + x_2 > 1$.
  Otherwise, $h(b) = h((\vep x_1 + \eta)(1+(1+x_3)x_2) + \gamma x_2+ \delta) = \eta + \delta > 1$. 
  
  We summarize the values of $b$ for which $h(b) = 1$ that were just derived:
  \begin{itemize}
  	\item If $x_1 = 1$, $\eta = \delta = \gamma = 0$, and $0 < \vep \leq 1+x_3$, we have $b = \vep(1+(1+x_3)x_2)$.
  	\item If $x_2 = 1$, $\eta = \delta = 0$, and $0 \leq \vep < \gamma \leq 1+x_3$, we have $b = \vep x_1(2+x_3) + \gamma x_2$.
  	\item If $\eta = 0$, $\delta = 1$, and $0 \leq \vep \leq \gamma \leq 1+x_3$, we have $b = \vep x_1(1+(1+x_3)x_2) + \gamma x_2 + 1$.
  	\item If $\eta = 1$, $\delta = \gamma = 0$, and $0\leq \vep \leq 1+x_3$, we have $b = (1+\vep x_1)(1+(1+x_3)x_2)$.
  \end{itemize}
  	
  	Our next goal is to establish that~\eqref{eqn 2} is always satisfied; recall that we are in the case where $q_j = (1+(1+x_3)x_1)(1+(1+x_3)x_2)$.
  	If $q_i =  (1+(1+x_3)x_1)(1+(1+x_3)x_2)$, the result is trivial. 
  	If $q_i = (1+x_3)(1+(1+x_3)x_1)$, we write $b = \alpha(1+(1+x_3)x_2)+\beta$, where $0\leq \beta < 1+(1+x_3)x_2$ and $0 \leq \alpha < 1+(1+x_3)x_1$. 
  	If $b > 1+(1+x_3)x_2$ (and thus, $\alpha \geq 1$), we can take $c = 1+(1+x_3)x_2$ as this makes both sides of~\eqref{eqn 2} equal to $(\alpha - 1)(1+x_3) + \left\lfloor \frac{\beta(1+x_3)}{1+(1+x_3)x_2} \right\rfloor$.
  	If $2\leq b < 1+(1+x_3)x_2$, note that $\alpha$ must be 0 and hence $b = \beta$. 
  	We write $\beta = \gamma x_2 + \delta$, where $0\leq \delta < x_2$ and $0 \leq \gamma \leq 1+x_3$. 
  	In the case that $\delta > 1$, we set $c = \gamma x_2 + 1$ as this makes both sides of~\eqref{eqn 2} equal to 0. 
  	Moreover, note that we need not consider the case where $\delta = 1$ since $h(\gamma x_2 + 1) = 1$.
  	Therefore, it only remains to find a $c$-value when $\delta = 0$. 
  	If $\delta = 0$, then $b = \beta = \gamma x_2$. 
  	Observe that $\gamma > 0$ since $\gamma = 0$ would imply $h(b) = 0 \ngeq 2$. 
  	In this case, we set $c = (\gamma - 1)x_2 + 1$. 
  	Since $1\leq \gamma \leq 1+x_3$ implies that $1 \leq 2+x_3 - \gamma \leq 1+x_3$, it is straightforward to check that both sides of~\eqref{eqn 2} will again be equal to 0.
  	
  	Finally, if $q_i = 1+x_3$, the analysis becomes slightly more complicated. 
  	As in the previous case, we begin by writing  $b = \alpha(1+(1+x_3)x_2)+\beta$, where $0\leq \beta < 1+(1+x_3)x_2$ and $0 \leq \alpha < 1+(1+x_3)x_1$. 
  	Furthermore, we write $\beta = \gamma x_2 + \delta$, where $0\leq \delta < x_2$ and $0\leq \gamma \leq 1+x_3$, and we write $\alpha = \vep x_1 + \eta$, where $0\leq \eta < x_1$ and $0\leq \vep \leq 1+x_3$. 
  	If $b > 1+(1+x_3)x_2$, we consider $c=1$. Substituting $c=1$ and the alternate form for $b$ into the left-hand side of~\eqref{eqn 2}, yields
  	\begin{align*}
  		\vep + \underbrace{\left\lfloor \frac{(\eta(1+x_3)-\vep+\gamma)(1+(1+x_3)x_2) + \delta(1+x_3) - \gamma}{(1+(1+x_3)x_1)(1+1+x_3)x_2)} \right\rfloor}_{=: \, F_1} \, .
  	\end{align*}
  	On the other hand, substituting into the right-hand side of~\eqref{eqn 2} yields
  	\begin{align*}
  		\vep + \underbrace{\left\lfloor \frac{(\eta(1+x_3)-\vep+\gamma)(1+(1+x_3)x_2) + (\delta-1)(1+x_3) - \gamma}{(1+(1+x_3)x_1)(1+1+x_3)x_2)} \right\rfloor}_{=:  \, F_2} \, .
  	\end{align*}
  	We must show that $F_1 = F_2$. 
  	To this end, let $n_1 = (\eta(1+x_3)-\vep+\gamma)(1+(1+x_3)x_2) + \delta(1+x_3) - \gamma$ and $n_2 = (\eta(1+x_3)-\vep+\gamma)(1+(1+x_3)x_2) + (\delta-1)(1+x_3) - \gamma$, that is, $n_1$ and $n_2$ are the numerators of the arguments in $F_1$ and $F_2$, respectively. 
  	Given the bounds on $\vep$, $\eta$, $\gamma$, and $\delta$, note that $-(1+x_3) \leq \delta(1+x_3) - \gamma < 1+(1+x_3)x_2$, $-2(1+x_3) \leq (\delta-1)(1+x_3) - \gamma < 1+(1+x_3)x_2$, and $-(1+x_3)\leq \eta(1+x_3)-\vep + \gamma \leq (1+x_3)x_1$.
  	Consequently, it follows that $\abs{n_k} < (1+(1+x_3)x_1)(1+(1+x_3)x_2)$ for $k \in \{1,2\}$, and this implies that $F_k$ is equal to either $0$ or $-1$.
  	Therefore, to achieve our goal, we must verify that either $n_1,n_2 < 0$ or $n_1,n_2 \geq 0$. 
  	Now, observe that $b>1+(1+x_3)x_2$ implies that either (A) $\alpha = 1$ and $\beta > 0$, or (B) $\alpha > 1$. For each of these scenarios, we consider subcases. First assume (A) holds, i.e., $\alpha = 1$ and $\beta > 0$.
  	
  	\textbf{Subcase 1 of 2:} Suppose $x_1 = 1$. 
  	Then, since $0 \leq \eta < x_1 = 1$, it follows that $\eta = 0$.
  	Consequently, as $1 = \alpha = \vep x_1 + \eta$, we have that $\vep = 1$. 
  	Thus, $n_1$ and $n_2$ reduce to
  	\begin{align*}
  		n_1 &= (\gamma-1)(1+(1+x_3)x_2) + \delta(1+x_3) - \gamma, \text{ and} \\
  		n_2 &= (\gamma-1)(1+(1+x_3)x_2) + (\delta-1)(1+x_3) - \gamma
  	\end{align*}
  	If $\gamma = 0$, then the numerators $n_1,n_2 < 0$ and hence $F_1 = F_2 = -1$. 
  	If $\gamma = 1$, note that $\delta \neq 1$ (since $\eta = 0, \vep = \gamma$, and $\delta = 1$ imply $h(b) = 1$). So, if $\delta = 0$, then $n_1,n_2 < 0$ and hence $F_1 = F_2 = -1$. Otherwise, if $\delta > 1$, then $n_1,n_2 > 0$ and thus $F_1 = F_2 = 0$.
  	If $\gamma > 1$, then $n_1,n_2 > 0$, implying $F_1 = F_2 = 0$. 
  	
  	\textbf{Subcase 2 of 2:} Suppose $x_1 > 1$.
  	Then, given that $\alpha = 1$, it must be the case that $\vep = 0$ and $\eta = 1$. 
  	Therefore, it immediately follows that $F_1 = F_2 = 0$ since $n_1,n_2 > 0$. 
  	
  	Thus, we can conclude that $F_1 = F_2$ in situation (A). Now, we must consider situation (B), i.e., when $\alpha > 1$. We again consider subcases.
  	
  	\textbf{Subcase 1 of 3:} Suppose $\eta = 0$. Then, it follows that $\vep > 0$, and our numerators reduce to
  	\begin{align*}
  		n_1 &= (\gamma - \vep)(1+(1+x_3)x_2) + \delta(1+x_3) - \gamma, \text{ and} \\
  		n_2 &= (\gamma-\vep)(1+(1+x_3)x_2) + (\delta-1)(1+x_3) - \gamma.
  	\end{align*}
  	If $\gamma > \vep$, then the numerators of both arguments will be positive, implying $F_1 = F_2 = 0$.
  	If $\gamma = \vep$, note that $\delta \neq 1$ (since $\eta = 0$, $\vep = \gamma$, and $\delta = 1$ imply $h(b) = 1$). So, if $\delta = 0$, we have that $n_1,n_2 < 0$, and hence $F_1 = F_2 = -1$. Otherwise, if $\delta > 1$, it follows that $n_1,n_2 \geq 0$ which implies $F_1 = F_2 = 0$.
  	Finally, if $\gamma < \vep$, it follows that $F_1 = F_2 = -1$ since $n_1,n_2 < 0$. 
  	
  	\textbf{Subcase 2 of 3:} Suppose $\eta = 1$. Again, since $\alpha > 1$, this implies $\vep > 0$.
  	As a result, we have the following reduction of $n_1$ and $n_2$:
  	\begin{align*}
  		n_1 &= (1+x_3 - \vep + \vep)(1+(1+x_3)x_2) + \delta(1+x_3) - \gamma, \text{ and} \\
  		n_2 &= (1+x_3 -\vep +\gamma)(1+(1+x_3)x_2) + (\delta-1)(1+x_3) - \gamma.
  	\end{align*}
  	If $\vep < 1+x_3$ then we have $n_1,n_2 > 0$, and thus $F_1 = F_2 = 0$. Otherwise, $\vep = 1+x_3$.
  	Note that since $\eta = 1$, $\delta$ and $\gamma$ cannot both be 0 as this would imply $h(b) = 1$. 
  	Therefore, if $\gamma = 0$, it must be that $\delta > 0$ which implies $n_1,n_2>0$ and $F_1 = F_2 = 0$. 
  	Otherwise, if $\gamma > 0$, $n_1,n_2 > 0$, and thus $F_1 = F_2 = 0$. 
  	
  	\textbf{Subcase 3 of 3:} Suppose $\eta > 1$. Then, it immediately follows that $n_1, n_2 > 0$, and we have that $F_1 = F_2 = 0$.
  	
  	Thus, we find that $F_1 = F_2$. 
  	Therefore, we have that~\eqref{eqn 2} is satisfied with $c=1$ for $b>1+(1+x_3)x_2$.
  	It remains to consider $2\leq b < 1+(1+x_3)x_2$. 
  	If $2\leq b < 1+(1+x_3)x_2$, note that $\alpha$ must be $0$ and hence $b = \beta$. 
  	Therefore, to ensure we satisfy~\eqref{eqn 1}, we consider $2\leq \beta < 1+(1+x_3)x_2$. Since $\beta > 1$ in this case, we may take $c = 1$ from which it is straightforward to verify that both sides of~\eqref{eqn 2} are equal to 0. 
  	This completes our third and final case, thereby establishing IDP for $\br$-vectors of type $(vii)$.

\bibliographystyle{plain}
\bibliography{Braun}
\end{document}